\documentclass[11pt,reqno]{amsart}

\usepackage[portrait,a4paper,margin=3.8cm]{geometry}

\usepackage[pagebackref]{hyperref}
\usepackage{amssymb,dsfont,enumerate}
\usepackage{dsfont,mathtools,amssymb, accents}
\usepackage{tikz}
\mathtoolsset{showonlyrefs}
\usepackage{subfigure}
\usepackage{amsmath}
\usepackage{color}
\usepackage{mathrsfs}
\usepackage{amsthm}
\usepackage{mhequ} 
\usepackage{mhsymb} 
\usepackage{enumerate}
\usepackage{enumitem}
\usepackage{comment}
\usepackage{ulem} 

\usepackage{environ}
\definecolor{calccolor}{rgb}{0.8,0.3,0}
\newenvironment{calc}    {\color{calccolor}     }    {     }

\colorlet{darkred}{red!90!black}
\definecolor{darkgreen}{rgb}{0.1,0.7,0.1}

\RenewEnviron{calc}{}

\long\def\willemText#1{} 
\def\willeminsert#1{} 

\hypersetup{
	citecolor=blue,
	colorlinks=true, 
	breaklinks=true, 
	urlcolor= blue, 
	linkcolor= black, 
	bookmarksopen=true, 
	pdftitle={Schoedinger correlated noise}, 
}

\newtheorem{theorem}{Theorem}[section]
\newtheorem{proposition}[theorem]{Proposition}
\newtheorem{lemma}[theorem]{Lemma}

\theoremstyle{definition}
\newtheorem{remark}[theorem]{Remark}
\newtheorem{definition}[theorem]{Definition}

\newtheorem{assumption}[theorem]{Assumption}

\numberwithin{equation}{section}
\numberwithin{figure}{section}

\usepackage[symbols,nogroupskip,sort=use,automake,toc=false]{glossaries-extra}
\glsxtrRevertMarks 
\setglossarystyle{long3col}
\makeglossaries

\glsxtrnewsymbol[description={order of magnitude maximum $\xi_L$ on $Q_L$}]{a_L}{\ensuremath{a_L}}
\glsxtrnewsymbol[description={discrete Laplacian}]{Delta}{\ensuremath{\Delta}}
\glsxtrnewsymbol[description={$\xi(x_{k,L})$ is the $\ell_L(k)$-th largest value of $\xi$ over $Q_L$; $x_{k,L} = y_{\ell_L(k),L}$}]{ellLk}{\ensuremath{\ell_{L}(k)}}
\glsxtrnewsymbol[description={Gaussian field as in Definition~\ref{d:GF}}]{xiL}{\ensuremath{\xi_L}}

\glsxtrnewsymbol[description={Box in $\Z^d$ of length $L$}]{Q_L}{\ensuremath{Q_L}}
\glsxtrnewsymbol[description={Covariance fct. with long \& short-range decay}]{v_L}{\ensuremath{v_L}}
\glsxtrnewsymbol[description={Tail of $v_L$}]{cTvL}{\ensuremath{\mathcal{T}_{v_L}}}
\glsxtrnewsymbol[description={Parameter of decay of $v_L$ away from $0$}]{d_L}{\ensuremath{d_L}}
\glsxtrnewsymbol[description={Fluctuation field, defined in \eqref{e:FluctField}}]{zetaLx_0}{\ensuremath{\zeta_{L,x_0}}}
\glsxtrnewsymbol[description={Shape, defined in \eqref{e:Shape}}]{cSL}{\ensuremath{\mathcal{S}_L}}
\glsxtrnewsymbol[description={Site in $Q_L$ where $|\varphi_{k,L}|$ reaches its maximum}]{xkL}{\ensuremath{x_{k,L}}}
\glsxtrnewsymbol[description={Site in $Q_L$ where $\xi_L$ reaches its $k$-th largest value}]{ykL}{\ensuremath{y_{k,L}}}
\glsxtrnewsymbol[description={$\Delta+\xi_L$ on $Q_L$ with Dirichlet b.c.}]{cHL}{\ensuremath{\mathcal{H}_L}}
\glsxtrnewsymbol[description={$\Delta - \cS_L$ on $Q_{r_L}$ with Dirichlet b.c.}]{barcHL}{\ensuremath{\bar{\mathcal{H}}_L}}

\glsxtrnewsymbol[description={Principal eigenvalue of $\bar{\mathcal{H}}_L$}]{barlambdaL}{\ensuremath{\bar{\lambda}_L}}
\glsxtrnewsymbol[description={Principal eigenfunction of $\bar{\mathcal{H}}_L$}]{barvarphiL}{\ensuremath{\bar{\varphi}_L}}

\glsxtrnewsymbol[description={Variance,
defined in \eqref{e:tauL}}]{tauL}{\ensuremath{\tau_L}}

\glsxtrnewsymbol[description={Random perturbation of $(1,2,\dots)$}]{ellinftyb}{\ensuremath{\ell_{\infty,b}}}

\glsxtrnewsymbol[description={Size of mesoscopic boxes}]{R_L}{\ensuremath{R_L}}
\glsxtrnewsymbol[description={Size of microscopic boxes}]{r_L}{\ensuremath{r_L}}
\glsxtrnewsymbol[description={Union of the mesoscopic boxes, defined in \eqref{e:UL}}]{U_L}{\ensuremath{U_L}}
\glsxtrnewsymbol[description={Box of length $a$, centered at $x$}]{Qax}{\ensuremath{Q_{a,x}}}
\glsxtrnewsymbol[description={Box of length $a$, centered at $x$ without $y$}]{Qaxney}{\ensuremath{Q_{a,x}^{\ne y}}}

\glsxtrnewsymbol[description={$\Delta + V$ on $C$ with Dirichlet b.c.}]{cHCV}{\ensuremath{\mathcal{H}_{C,V}}}
\glsxtrnewsymbol[description={Sum of $\bar{\varphi_L}^2(x) \zeta_{L,0}(x)$, as defined in \eqref{e:FL}}]{Phi_L}{\ensuremath{\Phi_L}}
\glsxtrnewsymbol[description={Sum of $\xi_L(x)$ and $\Phi_L(x)$, as defined in \eqref{e:XiL}}]{XiL}{\ensuremath{\Xi_L}}
\glsxtrnewsymbol[description={Tail of $v_L$}]{cT_vL_Xi}{\ensuremath{\mathcal{T}_{v_L^\Xi}}}
\glsxtrnewsymbol[description={Covariance function of $\Xi_L$}]{vLXi}{\ensuremath{v_L^{\Xi}}}
\glsxtrnewsymbol[description={Defined in \eqref{e:aLXiL}}]{aLXi}{\ensuremath{a_L^{\Xi}}}

\glsxtrnewsymbol[description={The intersection of $E_{L,x_0}^i$, $i=1,2,3$}]{ELx0}{\ensuremath{E_{L,x_0}}}
\glsxtrnewsymbol[description={Event defined in \eqref{e:E1}}]{ELx01}{\ensuremath{E^1_{L,x_0}}}
\glsxtrnewsymbol[description={Event defined in \eqref{e:E2}}]{ELx02}{\ensuremath{E^2_{L,x_0}}}
\glsxtrnewsymbol[description={Event defined in \eqref{e:E3}}]{ELx03}{\ensuremath{E^3_{L,x_0}}}

\glsxtrnewsymbol[description={Rescaled field for $\xi_L$}]{Thetaxi}{\ensuremath{\Theta_L^\xi}}
\glsxtrnewsymbol[description={Rescaled field for $\Xi_L$}]{ThetaXi}{\ensuremath{\Theta_L^\Xi}}
\glsxtrnewsymbol[description={Event defined in \eqref{e:AjLXi}}]{AChi}{\ensuremath{\mathcal{A}_{L,z}^\chi(s)}}
\glsxtrnewsymbol[description={Argmax of $\xi_L$ over $Q_{R_L,z}$}]{wLz}{\ensuremath{w_{L,z}}}
\glsxtrnewsymbol[description={Argmax of $\Xi_L$ over $Q_{R_L,z}$}]{wLzXi}{\ensuremath{w_{L,z}^\Xi}}

\glsxtrnewsymbol[description={Rescaled field for $\xi_L$ and $\Phi_L$, defined in \eqref{e:ScaledFieldsJoint}}]{ThetaxiPhi}{\ensuremath{\Theta_L^{(\xi,\Phi)}}}

\glsxtrnewsymbol[description={The random (point) measures of $\xi$, $\Xi$ or $(\xi,\Phi)$}]{cPLChi}{\ensuremath{\mathcal{P}_L^\chi}}

\glsxtrnewsymbol[description={Event defined in \eqref{e:LambdaL}}]{LambdaLs}{\ensuremath{\Lambda_L(s)}}

\glsxtrnewsymbol[description={$V_{L,x_0}\eqdef\xi_L-\xi_L(x_0)$}]{VLx}{\ensuremath{V_{L,x_0}}}

\renewcommand{\E}{\mathds{E}}

\renewcommand{\R}{\mathbb{R}}

\renewcommand{\N}{\mathbb{N}}

\renewcommand{\P}{\mathbb{P}}

\newcommand{\eqlaw}{\stackrel{\mbox{\tiny law}}{=}}

\DeclareMathOperator*{\argmax}{arg\,max}

\DeclareMathOperator{\var}{Var}
\DeclareMathOperator{\cov}{Cov}
\DeclareMathOperator{\Hess}{Hess}

\renewcommand{\R}{\mathbb{R}}

\renewcommand{\N}{\mathbb{N}}

\renewcommand{\Z}{\mathbb{Z}}
\renewcommand{\E}{\mathbb{E}}
\renewcommand{\P}{\mathbb{P}}

\newcommand{\YYL}{Y_L}
\newcommand{\plusinfty}{\infty}

\newcommand{\cA}{\ensuremath{\mathcal A}} 
\newcommand{\cB}{\ensuremath{\mathcal B}} 
\newcommand{\cC}{\ensuremath{\mathcal C}} 
\newcommand{\cD}{\ensuremath{\mathcal D}} 
 
\newcommand{\cF}{\ensuremath{\mathcal F}} 
 
\newcommand{\cH}{\ensuremath{\mathcal H}}

\newcommand{\cM}{\ensuremath{\mathcal M}} 
\newcommand{\cN}{\ensuremath{\mathcal N}} 
\newcommand{\cO}{\ensuremath{\mathcal O}} 
\newcommand{\cP}{\ensuremath{\mathcal P}}

\newcommand{\cS}{\ensuremath{\mathcal S}} 
\newcommand{\cT}{\ensuremath{\mathcal T}} 
 
\newcommand{\cV}{\ensuremath{\mathcal V}}

\newcommand{\cZ}{\ensuremath{\mathcal Z}}

\newcommand{\fa}{\mathfrak{a}}
\newcommand{\fb}{\mathfrak{b}}
\newcommand{\fc}{\mathfrak{c}}

\newcommand{\bbZ}{{\ensuremath{\mathbb Z}} } 

\newcommand{\1}{\mathds{1}}
\newcommand{\dd}{\mathrm{d}} 

\newcommand{\gap}{\mathrm{gap}}

\def\one{\mathrm{(A)}}
\def\two{\mathrm{(B)}}

\begin{document}

\title[Anderson Hamiltonians with correlated Gaussian potentials]{Top of the spectrum of discrete Anderson Hamiltonians with correlated Gaussian potentials}


\author{Giuseppe Cannizzaro}
\address{University of Warwick, Mathematical Sciences Building, Coventry CV4 7AL, UK}
\email{giuseppe.cannizzaro@warwick.ac.uk}
\author{Cyril Labb\'e}
\address{Université Paris Cité and Sorbonne Université, CNRS, Laboratoire de Probabilités, Statistique et Modélisation, F-75013 Paris, France and Institut Universitaire de France (IUF).}
\email{clabbe@lpsm.paris}
\author{Willem van Zuijlen}
\address{WIAS Berlin, Anton-Wilhelm-Amo-Str. 39, 10117 Berlin, Germany}
\email{vanzuijlen@wias-berlin.de}

\begin{abstract}
We investigate the top of the spectrum of discrete Anderson Hamiltonians with correlated Gaussian noise in the large volume limit. The class of Gaussian noises under consideration allows for long-range correlations. We show that the largest eigenvalues converge to a Poisson point process and we obtain a very precise description of the associated eigenfunctions near their localisation centres. 
We also relate these localisation centres with the locations of the
maxima of the noise. Actually, our analysis reveals that this relationship depends in a subtle way on the behaviour near $0$ of the covariance function of the noise: in some situations, the largest eigenfunctions are \textit{not} associated with the largest values of the noise.

\end{abstract}

\maketitle

\begin{calc}
This version has some additional explanations and calculations in this colour.
\end{calc}

\small
\setcounter{tocdepth}{1} 
\tableofcontents
\normalsize

\section{Introduction and main results}

The present article is concerned with the behaviour of the top eigenvalues / eigenfunctions of random operators of the form $\Delta+\xi$ on $\glsadd{Q_L} Q_L\eqdef[-L/2,L/2]^d \cap \Z^d$ in the limit $L\to\infty$. Here $\xi$ is a random potential on $\Z^d$ and $\Delta$ is the discrete Laplacian:
$$ \glsadd{Delta} \Delta f(x) = \sum_{y\in \Z^d: y\sim x} (f(y) - f(x))\;,\quad x\in\Z^d\;.$$
Such operators are often called random Schr\"odinger operators, or Anderson Hamiltonians. 
They are considered in physics to model the Hamiltonian of a quantum particle evolving in a crystal subject to defects or impurities. They are named after P.W. Anderson due to his seminal paper \cite{An58} which discusses the localisation of the quantum particle for large enough disorder of the potential and had a profound and lasting impact on the field. 
We refer to 
\cite{CaLa90, Ki08, AiWa15} 
for some references which address a part of this literature, in particular that regarding Anderson localisation, i.e., the property of having a pure point spectrum and exponentially decaying eigenfunctions.

These operators also naturally arise in the mathematical study of the so-called parabolic Anderson model:
$$ \partial_t u = \Delta u +  \xi u\;,\quad u(0,\cdot) = \delta_0(\cdot)\;.$$
Indeed, the behaviour of the solution $u$ at a large time $t$ can be well-approximated by the solution of the same stochastic partial differential equation but restricted to a finite ball of growing size $L=L(t)$, so that the top of the spectrum of the operator typically provides an accurate description of the growth and spreading of this solution, see for instance~\cite{GaKoMo07,KoLaMoSi09, SiTw14,FiMu14} and the book of K\"onig~\cite{Konig}. We also refer to~\cite{GaKoMo00,KoPevZ,GhYi} for articles on this topic in the continuous setting.

Most of the literature on these questions concern potentials $\xi$ made of i.i.d.~random variables with common distribution $\mu$. It is now well-understood that the right tail of $\mu$ plays a prominent role in the behaviour of the top of the spectrum, in particular: the heavier the right tail of $\mu$ is, the more localised the top eigenfunctions are. To illustrate this, let us present informally two important classes of distributions:\begin{itemize}
	\item (Single peak): $\mu([x,\plusinfty))$ decays ``slowly'' as $x\to\plusinfty$ (for instance Gaussian, exponential or Pareto distributions). In the limit $L\to\infty$, the top eigenfunctions are asymptotically given by Dirac masses localised at i.i.d.~uniform r.v.'s~drawn from $[-L/2,L/2]^d \cap \Z^d$.
	\item (Doubly-exponential): $\mu([x,\plusinfty))$ behaves like $\exp(-Ce^{x/\rho})$ for some $C,\rho > 0$ as $x\to \plusinfty$. In the limit $L\to\infty$, the top eigenfunctions are also centered at i.i.d.~uniform r.v.'s~drawn from $[-L/2,L/2]^d \cap \Z^d$, but not given by Dirac masses, but rather by exponentially decaying functions whose $\ell^2$ mass is supported on large balls, see~\cite{BiKo16}.
\end{itemize}
As it will be useful for later comparisons, let us mention a special class of laws, the Weibull distributions, 
which are such that $\mu([x,\plusinfty)) = \exp(-Cx^q)$, $x\ge 0$, for some  $q>0$  and $C>0$. 
They fall into the Single peak {class}, and precise results on the top of the spectrum of the Anderson Hamiltonian 
were established in~\cite{As08,As16}.


The relationship between the localisation centres of the top eigenfunctions and the successive maxima of the potential $\xi$ was investigated by Astrauskas~\cite{As13}. For Weibull tails (and more generally, in the Single peak class), a natural guess would be that, in the limit $L\to\infty$, the localisation center $x_{k,L} \in Q_L$ of the $k$-th eigenfunction is such that $\xi(x_{k,L})$ is the $k$-th largest value reached by $\xi$ on $Q_L$.\\
The situation is actually subtler: if we denote by $\glsadd{ellLk} \ell_L(k) $ the integer such that $\xi(x_{k,L})$ is the $\ell_L(k)$-th largest value of $\xi$ over $Q_L$, then with large probability as $L\to\infty$: \begin{itemize}
	\item if $q < 3$, $\ell_L(k) = k$,
	\item if $q = 3$, $\ell_L(k)$ is a non-trivial r.v.~of order $1$,
	\item if $q > 3$, the r.v.~$\ell_L(k)$ goes to $\plusinfty$.
\end{itemize}
Heuristically, when the right tail of $\mu$ is not so heavy ($q \ge 3$), the peaks of $\xi$ are less pronounced, so that one has to take the nearest neighbours of the successive maxima into consideration: the mass that the eigenfunction puts on these neighbouring sites may produce a shift in the eigenvalue that compensates for the difference between successive maxima and, thereby, makes the correspondence between successive maxima and successive eigenvalues / eigenfunctions non trivial. 
Let us mention that in the article~\cite{As13}, there are no precise statements that explain how this shift is produced.\\

Very little is known on the top of the spectrum of the Anderson Hamiltonian when $\xi$ is a correlated field: in~\cite{As03}, a few results were collected on the asymptotic behaviour of the potential and of the Anderson Hamiltonian for a Gaussian correlated field, while in~\cite{GaMo00} the asymptotic of the moments of the parabolic Anderson model with a correlated field were investigated. Let us also cite~\cite{GaKoMo00} for the almost sure asymptotic of the parabolic Anderson model with a correlated Gaussian field in the continuum.\\

\begin{calc}
More precisely, in \cite{As03} also a Gaussian random field (not necessarily homogeneous/stationary in law under spatial shifts, though for the comparison we act as though it is) is considered with a function $v=v_L$ (see also our Example~\ref{ex:indep}) that satisfies 
\begin{align*}
v_L(x) = O( (\log|x|)^{-1-\mu}), \mbox{ as } x \to \infty, 
\end{align*}
which implies the Long-Range decay condition of Definition~\ref{d:GF}~\ref{i:LR}, and furthermore, satisfies $d=d_L \in (0,\infty)$, where $d_L$ is as in \eqref{e:defdL}, i.e., satisfied Definition~\ref{d:GF}~\ref{i:De}. 

The asymptotics of the number of elements of $\Pi_L \eqdef \{x\in Q_L : \xi(x) \ge a_L\}$ and of the minimal distance and the minimal gap between values of $\xi$ for two sites in $\Pi_L$ is considered. 
From these facts, the asymptotics of the spectrum is considered in comparison with the top eigenvalues $\bar \lambda_L(w)$ of the 'single peak' Hamiltonian, which is $\Delta + V \wedge a_L + (a_L +2d) \1_{\{w\}}$: For $\tilde \lambda_{k,L}$ being the $k$-th maximum of the set $\{\bar \lambda_L(w) : w\in Q_L\}$ and $w_{k,L}$ such that $\tilde \lambda_{k,L} = \bar \lambda_L(w_{k,L})$, exponential decay of the difference $\lambda_{k,L} - \tilde \lambda_{k,L}$ and exponential decay of the $k$-th eigenfunction $\varphi_{k,L}$ in terms of $|x-w_{k,L}|$. 

In \cite{GaKoMo00} a homogeneous centred Gaussian field on $\R^d$ is considered with a twice continuously differentiable covariance function, which satisfies a certain decay so that the total mass of the parabolic Anderson model on $\R^d$ can be compared with the total mass of the same model but on the box $[-R,R]^d$ with Dirichlet boundary conditions, which then can be formulated in terms of the Anderson Hamiltonian on the same box with Dirichlet boundary conditions. 
Similar to our setting, the Anderson Hamiltonian $\Delta + \xi$ is compared with a deterministic Schrödinger operator, but here $\Delta - p$, where $p(x) = |\Sigma x |^2 / (2 \sigma^2)$ and $1 - p(x)$ forms the second-order approximation of $B(x)/B(0)$. 

Regarding \eqref{e:ApproxPot}, there is a clear comparison in strategy, where we have $\cS_L$ instead of $p$. 

\cite{GaMo00} can be found difficultly. 

In Gärtner Molchanov, Parabolic Problems for the Anderson Model II. Structure of High Peaks and Lifshitz Tails, Section 4.2, it is mentioned how to obtain the asymptotic of the moments of the parabolic Anderson model with a correlated field in a similar way as the i.i.d. case. 
\end{calc}

In the present article, we initiate a comprehensive study of the Anderson Hamiltonian with a correlated Gaussian field and we aim at answering the following questions:
\begin{enumerate}[label=\normalfont{(\arabic*)}]
	\item What features of the covariance function of the field are relevant to determine the statistics of the top of the spectrum?
	\item How do the top eigenvalues / eigenfunctions behave?
	\item What is the relationship between the top eigenvalues / eigenfunctions, and the successive maxima of the field?
\end{enumerate}

Actually, we consider a more general framework where the potential is allowed to depend on the size $L$ at which we consider the Anderson Hamiltonian: more precisely, we give ourselves a sequence $(\xi_L)_{L\ge 1}$ of Gaussian potentials, and we investigate the above questions on the operator $\cH_L \eqdef \Delta + \xi_L$ on $Q_L$. We work under two main assumptions on our field. The first condition concerns the long-range decay: roughly speaking, the covariance function is required to decay at infinity faster than $1/\log |x|$. This condition ensures that the  extreme value  statistics of our field behave in a way similar to that of the i.i.d.~case. However, to encompass such long-range correlations in the study of the Anderson Hamiltonian requires substantial technical work. The second condition concerns the short-range decay: the covariance function is assumed to decay fast enough near $0$. Actually, our study reveals that the behaviour near the origin of the covariance function of $\xi_L$ has a subtle impact on Question (3), and we identify non-trivial relationships between the top of the spectrum and the maxima of the field. 

\subsection{The potential}\label{sec:pot}

Let us begin by rigorously introducing the (family of) Gaussian field(s) the present paper is concerned with. 

\begin{definition}\label{d:GF}
For any integer $L\ge 1$, let $\glsadd{xiL}(\xi_L(x))_{x\in\Z^d}$ be a centred Gaussian field, stationary in law 
under spatial shifts, with unit variance at every point, and non-negative covariance 
function $v_L$ on $\Z^d$. 
Further, we assume that $v_L$ is such that 
\begin{enumerate}[label=\normalfont{(\Roman*)}]
\item\label{i:LR} \textit{(Long-Range decay)} its tails $\glsadd{cTvL} \cT_{v_L}$ satisfy
\begin{equ}[e:LongRange]
	\cT_{v_L}\eqdef\sup_{|x|\ge \exp(\sqrt{\ln L})} v_L(x) \ln |x| \longrightarrow 0\;,\quad \mbox{ as }L\to\infty\;.
\end{equ}
\item\label{i:De} \textit{(Short-Range decay)} there exist $\fc,\fc'>0$ such that for all $L\ge 1$
	\begin{equ}[e:AlmostDecay]
		1 - \glsadd{d_L} \frac{e^{\fc'|x|}}{d_L} \le v_L(x) \le 1 - \frac{\fc}{d_L}\;,\quad \forall x\in \bbZ^d\backslash\{0\}\;,
	\end{equ}
	where $d_L\geq 1$ is defined through
\begin{equ}[e:defdL]
	\sup_{|x|=1}v_L(x) = 1 - \frac1{d_L}\;.
\end{equ}
\end{enumerate} 
\end{definition} 

While the structural assumptions on the field $\xi_L$ are somewhat standard, conditions~\ref{i:LR} and~\ref{i:De} 
require some justification. The former controls the long-range behaviour of the potential 
and determines the minimal speed of decay of its correlations. As written,~\eqref{e:LongRange} is very mild. 
Indeed, for $\xi_L$ independent of $L$, it is equivalent to 
$v_L(x) = v(x) = o(1/\ln |x|)$ as $|x|\to\infty$, which is a well-known condition 
in the study of extreme values of correlated Gaussian fields (see~\cite[Chapter 4]{Leadbetter}). 

For $L$-independent potentials, the assumption~\ref{i:De} only imposes that $v$ is strictly below $1$ outside the origin. On the other hand, in the $L$-dependent case, the assumption~\ref{i:De} is non-trivial as~\eqref{e:AlmostDecay} ensures that the parameter $d_L$ in~\eqref{e:defdL} is 
a faithful control of the decay of $v_L$ near the origin.

\medskip

Let us present a few examples of potentials $\xi_L$ that satisfy Definition~\ref{d:GF}. We start with $L$-independent potentials:
\begin{enumerate}
		\item\label{ex:indep} {\bf$L$-independent Gaussian fields.} Take $\xi_L = \xi$ to be a centred, stationary Gaussian field with a covariance function $v_L = v$ independent of $L$, satisfying $v(0)=1$, $v(x) \ln|x| \to 0$ as $|x|\to\plusinfty$ and $\sup_{x\ne 0} v(x) < 1$. Then $d_L$ is independent of $L$ and finite. This covers the i.i.d.~case (where $v(x) = \1_{x=0}$) and, for instance, the discrete Gaussian Free Field in dimension $d\ge 3$ (where $v$ is the Green function associated to the discrete Laplacian). 
\end{enumerate}
We now present examples of potentials whose laws depend on $L$, and that arise by discretising a continuum Gaussian potential on a grid. 
More precisely, we start from a Gaussian potential $\zeta \eqdef u*\eta $ obtained by convolving a white noise 
$\eta$ on $\R^d$ with some function $u\colon \R^d \to \R$ so that the covariance function of $\zeta$ is 
$v(x)\eqdef (u*u)(x)$.
We give ourselves a sequence of $m_L\ge 1$ that diverges to $\plusinfty$ and we set 
$\xi_L(x) \eqdef \zeta(x/m_L)$ for all $x\in \Z^d$ (and $v_L(x)\eqdef v(x/m_L)$). 
As we will see, the regularity of $v$ at the origin has a subtle impact on the top of the spectrum of the Anderson Hamiltonian, and therefore we distinguish two cases:
\begin{enumerate}\setcounter{enumi}{1}
	\item\label{ex:smooth} {\bf Smooth.} Take $u:\R^d\to\R$ to be a smooth, compactly supported, radial function with unit $L^2$-norm. Then, $v_L$ is smooth at the origin and its gradient at $0$ is $0$. The resulting $\xi_L$ satisfies the requirements listed above with  $d_L$ of order $m_L^2$. 
	\begin{calc}
	\begin{align*}
	\frac{1}{d_L} = \inf_{|x|=1} 1- v_L(x),
	\end{align*}
	and because $ u'(0)=0$ as $u$ is radial, we have $v_L'(0)=0$ and so (considering $d=1$ for the following calculation, which suffices due to radiality of $u$ and thus of $v_L$) with $u^{*2}= u*u$, for $x$ such that $|x|=1$, 
	\begin{align*}
	\frac{1}{d_L} 
	&	=1- v_L(x) = \int_0^x v_L'(s) \dd s = \int_0^x \int_0^s v_L''(r) \dd r \dd s \\
	& = \frac{1}{m_L^2} \int_0^x \int_0^s (u^{*2})''(r/m_L) \dd r \dd s .
	\end{align*}
	The latter double integral converges to $\frac12 (u^{*2})''(0)$, which is strictly positive as $u^{*2}(0) = \|u'\|_{L^2}^2$.  
	\end{calc}
	\item\label{ex:indic} {\bf Rough.} Let $u$ be such that $v$ does not have continuous derivatives 
	at the origin, e.g. one can take, in $d=1$, $v(x)=(1-|x|)\vee 0$, corresponding to $u$ being the indicator of $[-1/2,1/2]$, 
	or, in any $d\geq 1$, $v(x)=e^{-|x|}$, for which the existence of $u$ follows by the fact that the Fourier transform of $v$ is 
	strictly positive. Then, $\xi_L$ satisfies the requirements listed above with  $d_L$  of order $m_L$.
		\begin{calc}
		Here one has (considering again $d=1$, which suffices) $u^{*2}(x) = (1-|x|)\vee 0$. 
		Therefore, for $x$ such that $|x|=1$ and large $L$ (so that $\frac{1}{m_L} \le 1$), 
		\begin{align*}
		\frac{1}{d_L} = 1 - v_L(x) = \frac{1}{m_L}. 
		\end{align*}
		\end{calc}
\end{enumerate}

\medskip

For our purposes, we need to collect some properties on $\xi_L$ restricted to the domain $Q_L \eqdef [-L/2,L/2]^d\cap \Z^d$. 
First of all, let us introduce the order of magnitude of the maximum of $\xi_L$ over $Q_L$, 
which  is given by the parameter $\glsadd{a_L} a_L$ implicitly defined through\footnote{$L^d$ should be interpreted as the cardinality of $Q_L$, although the exact value of the cardinality is slightly different but asymptotically equivalent.}
\begin{equ}[e:aL]
	\P(\xi_L(0)>a_L)=\frac{1}{L^d}\,.
\end{equ}
It is elementary to check that $a_L = \sqrt{2 d \ln L}\, (1+o(1))$ as $L\to\infty$.
\medskip

Let us now introduce an approximation of $\xi_L$ near one of its large peaks as it will be instrumental in this article. 

For any given site $x_0$, 
the field $\xi_L$ is the sum of the field $x\mapsto \xi_L(x_0)v_L(x-x_0)$ and the field $\glsadd{zetaLx_0} \zeta_{L,x_0}$ defined by 
\begin{equ}[e:FluctField]
\zeta_{L,x_0}(x) \eqdef	\xi_L(x) - \xi_L(x_0) v_L(x-x_0) ,\quad x\in\Z^d\;,
\end{equ}
which turns out to be Gaussian and independent of $\xi_L(x_0)$. 
Now let $x_0$
 be a site such that $\xi_L(x_0) \ge a_L - \theta$, where $\theta >0$ is some fixed, arbitrary value\footnote{We will take $\theta=2d+1$ later on for definiteness.}. We will show in Section \ref{Sec:Field} that $\xi_L$ can be well-approximated, in a neighborhood of $x_0$, as follows
\begin{equ}[e:ApproxPot]
	\xi_L(x) \approx \xi_L(x_0) - \cS_L(x-x_0) + \zeta_{L,x_0}(x)\;,
\end{equ}
where $\glsadd{cSL} \cS_L$ is the so-called \textit{shape} defined by
\begin{equ}[e:Shape]
	\cS_L(x) \eqdef a_L (1-v_L(x)) \ge 0\;,\quad x\in \Z^d\;.
\end{equ}
{In the sequel, the Gaussian field $\zeta_{L,x_0}$ will be dubbed} \textit{fluctuation field}. The shape should be seen as the first order description of the variation of the potential near a large peak, while the fluctuation field provides a random correction to this deterministic shape.

At this point, we observe that~\eqref{e:AlmostDecay} implies that $\cS_L$ is of order $a_L/d_L$ in the vicinity of the origin. We are thus naturally led to introduce the following assumption.

\begin{assumption}\label{ass:dLaL}
	Let $d_L$ and $a_L$ be respectively defined by~\eqref{e:defdL} and~\eqref{e:aL}. We assume that  
	\begin{equ}[e:StandingAssumptions]
		d_L \ll a_L\;.
	\end{equ}
\end{assumption}

Under this assumption, the shape is very steep so that we should be close to the Single peak class of the i.i.d.~setting. On the other hand, when $d_L$ is of order $a_L$, the shape is of order $1$ and this should correspond to the Doubly Exponential {class} of the i.i.d.~setting.\\

Our first result determines the statistics of the largest peaks of the potential and of the locations where these are achieved. 
To state it, for any $1 \le k \le \# Q_L$, let $\glsadd{ykL} y_{k,L}$ be the\footnote{If this point is not unique, take the smallest one for some arbitrary total order on $\Z^d$.} site in $Q_L$ where $\xi_L$ reaches its $k$-th largest value.

\begin{theorem}\label{Th:Potential}
Under Assumption~\ref{ass:dLaL},  the point process on $[-1,1]^d \times (-\infty,\infty]$ given by   
	$$\Big(\frac{y_{k,L}}{L},a_L(\xi_L(y_{k,L}) - a_L)\Big)_{1 \le k\le \# Q_L}\;,$$
	converges in law  for the topology of vague convergence  as $L\rightarrow \infty$ to a Poisson point process on $[-1,1]^d\times \R$ of intensity $\dd x \otimes e^{-u} \dd u$.
\end{theorem}

\begin{remark}
Observe that the value $a_L$ \textit{is the same} for any potential $\xi_L$ as in Definition~\ref{d:GF}, because $\xi_L(0) \sim \cN(0,1)$.
	Theorem~\ref{Th:Potential} shows that the statistics of the largest peaks \textit{are asymptotically the same} both in the i.i.d.~case and in any of the correlated cases considered here (which {\it a posteriori} justifies the comparison with the 
	Single Peak class {studied for instance in~\cite{As08}}).
\end{remark}

Actually, in Sections \ref{Sec:Field} and \ref{sec:LoctoGlobPot} we will gather much more information on the Gaussian field as we will not only study its largest values but also other functionals (including the fluctuation field) which are instrumental in the study of the top of the spectrum of the Anderson Hamiltonian.
\medskip 

Let us finally mention that the study of extrema of Gaussian fields has been the topic of a large literature. Let us in particular cite the  works~\cite{BiLo16,BiLo18} of Biskup and Louidor on the discrete Gaussian Free Field in dimension 2.  Even though the $2d$ discrete GFF is not stationary and, without suitable boundary conditions, it is not even pointwise defined, 
our Theorem~\ref{Th:Potential} can be viewed as the counterpart of their (much subtler)~\cite[Thm. 2.1]{BiLo18}, and 
indeed our result covers in particular the discrete GFF in $d\geq 3$ (which instead is point-wise defined and stationary). 
We believe it would be interesting to see whether 
the techniques developed below could be used to study the Anderson Hamiltonian with potential given by the $2d$ discrete GFF, 
for which the above mentioned references would thus represent a natural starting point.

\subsection{Main result}\label{sec:MR}

We consider the random operator $\glsadd{cHL} \cH_L \eqdef \Delta + \xi_L$ on $\glsadd{Q_L} Q_L = [-L/2,L/2]^d\cap \Z^d$ endowed with Dirichlet boundary conditions\footnote{The domain of $\cH_L$ is the set of all functions $f:Q_L\to \R$, extended outside $Q_L$ by setting them to $0$, and the value of $\cH_L f(x)$ is simply given by $\Delta f(x) + \xi_L(x) f(x)$ for all $x\in Q_L$.}. This operator is finite dimensional and self-adjoint: we let $(\lambda_{k,L})_{1\le k \le \# Q_L}$ be its successive eigenvalues in non-increasing order, and $(\varphi_{k,L})_{1\le k \le \# Q_L}$ be the associated eigenfunctions normalised in $\ell^2$. We also denote by $\glsadd{xkL} x_{k,L}$ the\footnote{If this point is not unique, take the smallest one for some arbitrary total order on $\Z^d$.} point in $Q_L$ that maximizes $|\varphi_{k,L}|$ for any $1\le k \le \# Q_L$, and w.l.o.g. we can take $\varphi_{k,L}$ positive at $x_{k,L}$.\\

Our analysis of the top of the spectrum of $\cH_L$ relies on a splitting scheme where we analyse the operator restricted to \textit{mesoscopic boxes} of side-length $1 \ll R_L \ll L$ (see below for further details). A key step consists in obtaining a fine description of the top eigenvalue, $\lambda_1(Q_{R_L},\xi_L)$, of the operator $\Delta + \xi_L$ restricted to $Q_{R_L}$, when there is a point $x_0 \in Q_{R_L}$ at which $\xi_L(x_0) \ge a_L - \theta$. Recall that the approximation of the field near $x_0$ given in~\eqref{e:ApproxPot} displays three terms: the value of the field at $x_0$, the shape and the fluctuation field. While the impact on $\lambda_1(Q_{R_L},\xi_L)$ of the value of the field at $x_0$ is rather straightforward as it amounts to a (random) shift by $\xi_L(x_0)$, 
those of the shape and the fluctuation fields are more subtle. Let us consider the deterministic operator 
\begin{equ}[e:DetHam]
\glsadd{barcHL}
	\bar{\cH}_L \eqdef \Delta - \cS_L\,,\qquad \text{on $Q_{r_L}\eqdef[-r_L/2,r_L/2]^d \cap \Z^d$}
\end{equ}
endowed with Dirichlet b.c., where $1 \ll r_L \ll R_L$ will be introduced later on. We let $\glsadd{barlambdaL} \bar{\lambda}_L$ be the largest eigenvalue of this operator, and $\glsadd{barvarphiL}\bar{\varphi}_L$ be its associated normalised eigenfunction (taken non-negative w.l.o.g.). We will show that
$$ \bar{\lambda}_L = -2d +  \cO(\frac{d_L}{a_L}) \;,\quad L\to\infty\;,$$
while 
\begin{equ}[e:BasicphiL]
	\bar{\varphi}_L(0)=1 - \cO(\frac{d_L}{a_L}) \;,\quad \text{ and }\quad \bar{\varphi}_L(x) \asymp \frac{d_L}{a_L}\; \text{for $x\in \Z^d$ with $|x|=1$. }
\end{equ}
Our ansatz, which is detailed in Section \ref{Sec:Local}, is that the principal eigenfunction  of 
$\Delta + \xi_L$ on $Q_{R_L}$ should be well approximated by $\bar{\varphi}_L(\cdot-x_0)$, 
and consequently the variational characterisation of the principal eigenvalue together 
with~\eqref{e:ApproxPot} suggests that $\lambda_1(Q_{R_L},\xi_L)$ should satisfy 
\begin{equs}
\lambda_1(Q_{R_L},\xi_L)&\approx \langle \bar{\varphi}_L(\cdot-x_0), \cH_L \bar{\varphi}_L(\cdot-x_0)\rangle 
= \langle \bar{\varphi}_L(\cdot-x_0), (\Delta+\xi_L) \bar{\varphi}_L(\cdot-x_0)\rangle\\
&\approx \langle \bar{\varphi}_L(\cdot-x_0), (\Delta+\xi_L(x_0)- \cS_L(\cdot -x_0) + \zeta_{L,x_0}(\cdot)) \bar{\varphi}_L(\cdot-x_0)\rangle\\
&=\xi_L(x_0)+\langle \bar{\varphi}_L, \bar\cH_L \bar{\varphi}_L\rangle +\langle \bar{\varphi}_L(\cdot-x_0),  \zeta_{L,x_0}(\cdot) \bar{\varphi}_L(\cdot-x_0)\rangle\\
&=\xi_L(x_0) + \bar{\lambda}_L + \sum_{x\in Q_{r_L}} \bar{\varphi}_L^2(x) \zeta_{L,x_0}(x+x_0)\,.\label{e:competing}
\end{equs}
In \eqref{e:competing}, we observe a competition between the randomness coming from the first and third terms. More precisely, the first term fluctuates at scale $1/a_L$ around a leading order $a_L$ (as shown by Theorem \ref{Th:Potential}), while the third term is of order $\tau_L$ where 
$\tau_L^2$ is the variance of the sum on the right-hand side of \eqref{e:competing}, i.e. 
\begin{equ}[e:tauL]
\glsadd{tauL}
	\tau_L^2 \eqdef \var\Big(\sum_{x\in Q_{r_L}}  \bar{\varphi}_L^2(x) \zeta_{L,0}(x) \Big)\;.
\end{equ}
Therefore, we should expect that the relative values of $1/a_L$ and $\tau_L$ play an important role. 
%
We will work under the following assumption on the strength of $\tau_L$. 

\begin{assumption}\label{a:TechnicalAssum}
	The $L$-dependent constant $\tau_L$ in~\eqref{e:tauL} satisfies
	\begin{equ}[e:TechnicalAssum]
		\tau_L \ll \frac1{a_L} \Big(\frac{a_L}{d_L}\Big)^{\frac12}\;.
	\end{equ}
\end{assumption}

Notice that, since $d_L$ in~\eqref{e:defdL} is bounded below by $1$ and $a_L$ blows up in $L$, 
under Assumption~\ref{a:TechnicalAssum}, $\tau_L$ vanishes. Such restriction is intimately related to the replacement 
of $\lambda_1(Q_{R_L},\xi_L)$ by the quantity in~\eqref{e:competing}. 
More precisely, it guarantees that the error made in the approximation is negligible. 
If this assumption is not satisfied, one certainly needs to take into account many more terms 
induced by the fluctuation field than the sole projection 
$\sum_{x\in Q_{r_L}} \bar{\varphi}_L^2(x) \zeta_{L,x_0}(x+x_0)$. 
We leave this task for future investigations. 

That said, Assumption~\ref{a:TechnicalAssum} is satisfied in most cases: 
when the covariance function $\glsadd{v_L} v_L$ does not depend on $L$ (Example \eqref{ex:indep}) 
or when the covariance function depends on $L$ but is ``regular'' enough (Example \eqref{ex:smooth}), then it holds. 
On the other hand, it fails when the covariance function is not regular enough 
(Example \eqref{ex:indic}) \textit{and} $d_L$ is ``close enough'' to $a_L$.
\medskip

Our first main result concerns the eigenvalue order statistics and the localisation centers.

\begin{theorem}[Eigenvalue order statistics]\label{Th:MainEigenvalue}
	Under Assumptions~\ref{ass:dLaL} and \ref{a:TechnicalAssum}, the point process  on $[-1,1]^d \times (-\infty,\infty]$ given by the points  
		$$ \Big(\frac{x_{k,L}}{L}, a_L\big(\lambda_{k,L}-a_L\sqrt{1+\tau_L^2}-\bar{\lambda}_L\big)\Big)_{1\le k \le \# Q_L}\;,$$
		converges in law  for the topology of vague convergence  as $L\to\infty$ towards a Poisson point process on $[-1,1]^d\times\R$ of intensity $\dd x \otimes e^{-u} \dd u$.
\end{theorem}

If $\tau_L \ll 1/a_L$, one can replace $\sqrt{1+\tau_L^2}$ by $1$ without altering the result. 
However, when $\tau_L$ is order $1/a_L$ or larger, then the correction is crucial and hints at the fact that 
the relationship of the top eigenvalues with the largest values of $\xi_L$ is no longer trivial, 
see Theorem \ref{Th:MainRelationship}.
\medskip

We now address the localisation properties of the main eigenfunctions. Recall that $\bar{\varphi}_L$ is a deterministic function which is almost a Dirac mass at $0$ (see~\eqref{e:BasicphiL}). 

\begin{theorem}[Localisation]\label{Th:MainLocalisation}
	 Under Assumptions~\ref{ass:dLaL} and \ref{a:TechnicalAssum}, for any $k\ge 1$, the r.v.
	$$ \frac{a_L}{d_L} \Big\| \varphi_{k,L}(\cdot)-\bar{\varphi}_L(\cdot-x_{k,L}) \Big\|_{\ell^2(Q_{L})}\;,$$
		converges to $0$ in probability.
\end{theorem}

\begin{remark}
From this bound one may expect that $\varphi_{k,L}$ is of exponential decay, because $\bar \varphi_L$ is (see Lemma~\ref{l:DecayPsi}). 
Using the estimates presented in the proof of this bound, one can deduce an exponential decay in some small neighbourhood of the localization centre. 
Possibly, a bootstrap argument based on Lemma \ref{lemma:BKL2Bound} could lead to a global exponential decay bound but this would require additional work.
%
\end{remark}

\begin{remark}
	Let us point out that the localization results established here concern the top of the spectrum of the operator $\cH_L$ in the large volume limit. Even in the case where the potential is $L$-independent, they do not provide any information on the nature of the spectrum of the infinite volume operator: indeed, the top of the spectrum of $\cH_L$ diverges to $+\infty$ and, thus, is not relevant for the infinite volume operator.
\end{remark}

We now relate the top eigenvalues/eigenfunctions with the maxima of the field. 
For any $k\ge 1$, we define the random variable $\gls{ellLk}$ through 
\begin{equ}[e:lLk]
x_{k,L} = y_{\ell_L(k),L}\,,
\end{equ} 
where, as a reminder, $y_{k,L}$ is the site in $Q_L$ where $\xi_L$ reaches its $k$-th largest value. 
The random variable $\ell_L(k)$ provides the rank of the maximum of $\xi_L$ at which the $k$-th eigenfunction is localised.

To state our result, we need to introduce, for any given parameter $b>0$, a random permutation 
$\glsadd{ellinftyb}(\ell_{\infty,b}(k))_{k\ge 1}$ of $(1,2,\ldots)$. 
\begin{itemize}[noitemsep]
\item Let $u_1 > u_2 > \dots$ be distributed according to a Poisson point process of intensity $e^{-u} \dd u$. 
\item Draw an independent sequence $(v_i)_{i\ge 1}$ of i.i.d.~$\cN(0,b)$ r.v. 
\item Let $(p_i)_{i\ge 1}$ be the (non-increasing) order statistics of the decorated Poisson point process $(u_i+v_i)_{i\ge 1}$. 
\end{itemize}
Then, for any $k\geq 1$, define $\ell_{\infty,b}(k)$ according to 
\begin{equ}[e:Permutation]
p_k = u_{\ell_{\infty,b}(k)}+v_{\ell_{\infty,b}(k)}\,.
\end{equ}

\begin{theorem}[Relationship with the maxima of $\xi_L$]\label{Th:MainRelationship}
	Under Assumptions~\ref{ass:dLaL} and \ref{a:TechnicalAssum}, it holds:
	\begin{enumerate}[label=\normalfont{(\alph*)}]
			\item \label{item:tau_L_small}
			if $\tau_L \ll \frac1{a_L}$ then for any given $k\ge 1$, $\P(\ell_L(k) = k) \to 1$ as $L\to\infty$,
			\item \label{item:tau_L_intermediate}
			if $\tau_L \sim \sqrt{b} \frac1{a_L}$ for some constant $b>0$ then $(\ell_L(k))_{k\ge 1}$ converges in law to $(\ell_{\infty,b}(k))_{k\ge 1}$,
			\item \label{item:tau_L_large}
			if $\tau_L \gg \frac1{a_L}$ then for any given $k\ge 1$, $\ell_L(k)\to \plusinfty$ in probability.
		\end{enumerate}
\end{theorem}

Let us point out the analogy with the i.i.d.~Weibull case presented in the introduction: the regime $\tau_L \ll 1/a_L$ corresponds to $q< 3$, the regime $\tau_L \sim b/a_L$ corresponds to $q=3$ and $\tau_L \gg \frac1{a_L}$ to $q>3$. 
Let us mention that the law of the permutation was not identified in the literature in the Weibull case with $q=3$  and, even though it will certainly not be the same as in the 
present case, we believe that our techniques could be used to determine it.  
\medskip

To conclude the introduction, let us see how the previous statements translate to each of the 
examples of Section \ref{sec:pot} (we refer to Appendix \ref{App:tauL} for details on the computation of the order of $\tau_L$).
\begin{enumerate}[label=\normalfont{(\arabic*)}]
	\item $L$-independent Gaussian fields: this is the case that includes the i.i.d. Gaussian case. The parameter $d_L$ is independent of $L$ and finite, while 
	$\tau_L$ is of order $1/a_L^2$ 
	so that both Assumptions~\ref{ass:dLaL} and~\ref{a:TechnicalAssum} are satisfied. Both Theorems~\ref{Th:MainEigenvalue} and~\ref{Th:MainLocalisation} hold and in the former $\sqrt{1+\tau_L^2}$ can be replaced by $1$. Concerning Theorem~\ref{Th:MainRelationship}, the relationship with the maxima is given by~\ref{item:tau_L_small}.
	\item Smooth: 
	If $m_L^2 \ll a_L$, then 
	$d_L\ll a_L$ and $\tau_L$ is of order $d_L/a_L^2$ so that both Assumptions~\ref{ass:dLaL} and~\ref{a:TechnicalAssum} are satisfied and $\tau_L\ll 1/a_L$. Hence, the same conclusions as in the previous case hold.  
	\item Rough: 	If $m_L \ll a_L$, then once again $d_L\ll a_L$ but $\tau_L$ is of order $d_L^{3/2}/a_L^2$ so that Assumption \ref{a:TechnicalAssum} 
	is satisfied provided $ m_L =  d_L \ll a_L^{3/4}$. The correction $\sqrt{1+\tau_L^2}$ in 
	Theorems~\ref{Th:MainEigenvalue} and~\ref{Th:MainLocalisation} is now crucial and 
	according to whether $d_L \ll a_L^{2/3}$, $d_L \asymp a_L^{2/3}$ or $d_L \gg a_L^{2/3}$, 
	the relationship with the maxima in Theorem~\ref{Th:MainRelationship} is respectively given by \ref{item:tau_L_small}, \ref{item:tau_L_intermediate} or \ref{item:tau_L_large}. 
\end{enumerate}

\begin{calc}
Let us check the statements made before about each of the class of examples (1),(2) and (3). First of all, let us mention that by Lemma~\ref{l:DecayPsi} we have 
	\begin{equ}
		\bar \phi_{L}(x)^2\leq  \Big(1+ c_d\frac{a_L}{d_L}\Big)^{-2|x|}\qquad \forall\, x\in Q_{r_L}\,,
	\end{equ}
	and thus 
	\begin{equ}
		\bar \phi_{L}(x)^2\leq  \Big( c_d^{-1} \frac{d_L}{a_L}\Big)^{2|x|}\qquad \forall\, x\in Q_{r_L}, x \ne 0 \,. 
	\end{equ}
Using the expression of the covariance of $\zeta_{L,0}$ as derived from Lemma~\ref{l:FluctField}, we get 
	\begin{align*}
	\tau_L^2 
	& = \sum_{x,y \in Q_{r_L}} \bar \varphi_L^2(x) [ v_L(x-y) - v_L(x) v_L(y)],
	\end{align*}
	but also for those $x,y$ such that $\frac{e^{\fc'|x|}}{d_L} \le 1$: 
	\begin{align*}
	v_L(x-y) - v_L(x) v_L(y)
	& \le 1 - \frac{\fc}{d_L} - \Big(1 - \frac{e^{\fc'|x|}}{d_L}\Big) \Big(1 - \frac{e^{\fc'|y|}}{d_L}\Big) \\
	& \le \frac{e^{\fc'|x|}}{d_L} + \frac{e^{\fc'|y|}}{d_L} - \frac{\fc}{d_L} -  \frac{e^{\fc'|x|}}{d_L} \frac{e^{\fc'|y|}}{d_L} \\
	& \le \frac{e^{\fc'|x|}}{d_L} + \frac{e^{\fc'|y|}}{d_L} 
	\le \frac{e^{\fc'(|x|+|y|)}}{d_L} 
	\end{align*}
Hence, for large enough $L$
\begin{align*}
\tau_L^2 \lesssim \sum_{x,y \in Q_{r_L} \setminus \{0\}  } \frac{e^{\fc'(|x|+|y|)}}{d_L}  \Big(\frac{d_L^2}{a_L^2} \Big)^{2|x|} \lesssim \frac{d_L^3}{a_L^4}, 
\end{align*}
so that, (for any $v_L$ satisfying the assumptions in Definition~\ref{d:GF})
\begin{align*}
\tau_L \lesssim \frac{d_L^{3/2}}{a_L^2}. 
\end{align*}

(1): As $\xi_L=\xi$, we also have that $\zeta_{L,x_0} = \zeta_{1,x_0}$, $d_L=d_1 \ll a_L $ and $v_L=v$ for all $L$. Therefore, by Lemma~\ref{l:DecayPsi}, using that $v$ is bounded by $1$,  we obtain by \eqref{eqn:sum_estimate}
\begin{align*}
\tau_L^2 \lesssim \sum_{x,y \in Q_{r_L} \setminus \{0\}  } \Big(\frac{d_L}{a_L} \Big)^{2(|x|+|y|)} \lesssim \frac{d_L^4}{a_L^4} \lesssim \frac{1}{a_L^4} 
\end{align*}

(3): This is covered by the above calculation. Furthermore, to show that $\tau_L$ is not better than the above order (remember, by the calculations done in \eqref{ex:indic} here $m_L=d_L$), for large enough $L$
\begin{align*}
v_L(x-y) - v_L(x) v_L(y) 
& =  1 - \frac{|x-y|}{m_L} - (1 - \frac{|x|}{m_L} ) (1 - \frac{|y|}{m_L})\\
& = \frac{|x|}{m_L} + \frac{|y|}{m_L} - \frac{|x|}{m_L} \frac{|y|}{m_L} - \frac{|x-y|}{m_L}. 
\end{align*}

(2): By the fact that $u$ is a radial function, we have $v_L'(0)= v_L'''(0) =\dots  =0$ and thus for some constants $\fa = \frac{(u^{*2})''(0)}{2}$, $\fb = \frac{(u^{*2})^{(4)}(0)}{4!}$, 
\begin{align*}
v_L(x) = 1 + \fa  \frac{|x|^2}{d_L} + \fb \frac{|x|^4}{d_L^2} + O\Big(\frac{|x|^6}{d_L^3}\Big), 
\end{align*}
and thus with $\tilde \fa = \frac{\fa}{d_L}$ and $\tilde \fb = \frac{\fb}{d_L}$, 
\begin{align*}
&v_L(x-y) - v_L(x) v_L(y) \\
& =  \tilde \fa|x-y|^2 + \tilde \fb|x-y|^4 
- (1 + \tilde \fa |x|^2 + \tilde \fb |x|^4 ) (1 + \tilde \fa |y|^2 + \tilde \fb |y|^4 ) \\
& \qquad + O(\frac{|x|^6}{d_L^3}) +O(\frac{|y|^6}{d_L^3}) +O(\frac{|x-y|^6}{d_L^3}) \\
& =  -2\tilde \fa x \cdot y +\tilde \fa^2 |x|^2 |y|^2 + \tilde \fb[|x-y|^4 -|x|^4 -|y|^4] \\
& \qquad + O(\frac{|x|^6}{d_L^3}) +O(\frac{|y|^6}{d_L^3}) +O(\frac{|x-y|^6}{d_L^3}). 
\end{align*}
Due to radiality of $\bar \varphi_L$ we have 
\begin{align*}
&  \sum_{x,y \in Q_{r_L}} \bar \varphi_L^2(x) x \cdot  y =0,
\end{align*}
so that in the sum for $\tau_L^2$ only multiples of $\tilde \fa^2$ and $\tilde \fb$ appear, which are of order $d_L^{-2}$, so that 
\begin{align*}
	\tau_L^2 
	& = \sum_{x,y \in Q_{r_L}} \bar \varphi_L^2(x) [ v_L(x-y) - v_L(x) v_L(y)] \\
	& \le \sum_{x,y \in Q_{r_L}} \frac{1}{d_L^2}  \Big(\frac{d_L^2}{a_L^2} \Big)^{2|x|} O(|x|^4 +|y|^4) \lesssim \frac{d_L^2}{a_L^4}. 
\end{align*}
\end{calc}

From now on and throughout the article, \textbf{we will always work under Assumptions~\ref{ass:dLaL} and~\ref{a:TechnicalAssum}} unless otherwise stated. 

\subsection{Strategy of proof and structure of the article}\label{sec:Splitting}

The study of the order-statistics of (the sequence of) Gaussian field(s) $\xi_L$ in Definition~\ref{d:GF} 
and the spectral properties of the Hamiltonian $\cH_L$ follow distinct but interdependent routes 
that we now outline. 
Both rely on a suitable localisation procedure (or splitting scheme) 
whose aim is to reduce the analysis from the {\it macroscopic} box of side-length $L$ 
to a collection of {\it mesoscopic} boxes of side-length $R_L$, which grows with $L$ but is much smaller than $L$,  
and ultimately, for the Hamiltonian, to an even smaller box of side-length $r_L$. 
Let us fix the sequences $\glsadd{R_L} (R_L)_{L\geq 1}$ and $\glsadd{r_L} (r_L)_{L\geq 1}$ of positive constants in such a way that 
they satisfy\footnote{Concerning $R_L$, the lower bound is needed in the proof of Lemma~\ref{l:RedE} while the upper bound in that of Proposition~\ref{l:Event}. For $r_L$ instead, the lower bound is needed in the proof of Lemma~\ref{l:RedE}, while the upper bound  is needed in the proof  of Lemma~\ref{l:GL}.  } 
\begin{equs}\label{e:RL}
	&a_L\ll \ln R_L \ll a_L\frac{a_L}{d_L}\;, \\
	&\ln a_L\leq \ln r_L\ll \sqrt{a_L}\;,\label{e:rL}
\end{equs} 
where $a_L$ and $d_L$ are defined in~\eqref{e:aL} and~\eqref{e:defdL}, respectively, 
and where, for two sequences $(u_L)_{L\ge 1}, (v_L)_{L\ge 1}$ in $(0,\infty)$, the notation $u_L \ll v_L$ means $u_L/v_L \to 0$ as $L\to\infty$. 
\medskip

Now, the aforementioned localisation procedure can be roughly visualised in the following diagram
\begin{equ}[e:LocPro]
Q_L \quad \stackrel{\mathrm{(a)}}{\longrightarrow} \quad U_L \quad \stackrel{\mathrm{(b)}}{\longrightarrow} \quad Q_{R_L}\quad \stackrel{\mathrm{(c)}}{\longrightarrow} \quad Q_{r_L}
\end{equ}
where, for $a>0$, we denoted by $Q_{a}\eqdef [-a/2,a/2]^d\cap \Z^d$, and   
$U_L$ is obtained by peeling off suitable strips from $Q_L$ and is thus given by 
the union of disjoint and well-separated boxes of side-length $R_L$. 
More precisely, we consider a covering of $Q_L$ into boxes of side-length $R_L + \sqrt{R_L}$ whose interiors are disjoint
\begin{equ}[e:QLnL]
	Q_L = \bigcup_{j=1}^{n_L} Q_{R_L+\sqrt{R_L}, z_{j,L}}\;,\qquad \text{for}\quad n_L\eqdef\frac{\# Q_L}{\# Q_{R_L+\sqrt{R_L}}}\,,
\end{equ}
where\footnote{For notational convenience, we assume that $n_L$ is an integer. To treat the general case, it suffices to adapt the splitting scheme.} { $\glsadd{Qax} Q_{a,x} \eqdef x + Q_a$ and }$(z_{j,L})_{j=1,\ldots,n_L}$ forms implicitly a lattice of points at distance at least $R_L+\sqrt{R_L}$ from each other,
and then we peel off a boundary layer of size $\sqrt{R_L}$ by setting
\begin{equ}[e:UL]
\glsadd{U_L}
	U_L \eqdef \bigcup_{j=1}^{n_L} Q_{R_L, z_{j,L}}\;. 
\end{equ}

\begin{remark}\label{rem:DetLocPro}
In some references on the Anderson Hamiltonian with i.i.d. potential (see e.g.~\cite{BiKo16}), the localisation procedure 
(or splitting scheme) is ``random'' in that the set $U_L$ is chosen to be the union of mesoscopic boxes 
centred around the large peaks of the potential. While possibly we could have proceeded similarly, it 
would have added an additional layer of difficulty as the 
presence of correlations makes such a procedure much more complicated and 
the way in which (the already challenging) step (b) in~\eqref{e:LocPro} should be approached much less transparent. 
\end{remark}

Getting back to the cartoon in~\eqref{e:LocPro}, step (a) is relatively simple: at the level of the Gaussian field, 
since $\sqrt{R_L} \ll R_L$, the cardinality of $Q_L \backslash U_L$ 
is negligible compared to that of $Q_L$ so that, with large probability, $\xi_L$ does not display large peaks in this set 
(see Section~\ref{sec:ThPot}) and thus its order statistics are unaffected by its value therein. 
As we will see in Section~\ref{subsec:ProofMainTh}, 
this also implies that the top eigenfunctions put an exponentially small amount of mass 
on $Q_L \backslash U_L$ and thus the top eigenpairs of $\cH_L$ on $Q_L$ and on $U_L$ (asymptotically) coincide. 
\medskip

Step (b) for $\xi_L$ 
is one of the main novelties of the present work. 
The advantage of $U_L$ over $Q_L$ is that any two distinct boxes 
$Q_{R_L, z_{j,L}}$ and $Q_{R_L, z'_{j,L}}$ lie at a distance at least 
$\sqrt{R_L} \geq  \exp(\sqrt{\ln L})$ so that the restrictions of $\xi_L$ to these boxes 
display negligible correlations thanks to \eqref{e:LongRange}. 
This suggests that it should be possible to regard these as independent, but making this 
rigorous is technically challenging as it amounts to determine 
non-trivial {\it decorrelation estimates} (whose nature, in case $U_L$ were chosen as in Remark~\ref{rem:DetLocPro}, 
is unclear) to which Section~\ref{sec:LongRange} is dedicated. 
In view of these decorrelation estimates, the analysis of $\xi_L$ and $\cH_L$ on $U_L$ is 
rather standard and presented in Sections~\ref{sec:OrderStatsXi} and~\ref{subsec:spectrumUL} respectively. 
\medskip

Thanks to the steps (a) and (b), all that remains to do is to study the behaviour of the Gaussian field and the Anderson Hamiltonian 
on the mesoscopic box $Q_{R_L}$, and this is the second main novelty of the present paper. 
For the former, it consists of, first, formalising the description of the noise close to a large peak  
in terms of the {\it shape} $\cS_L$ and the {\it fluctuation field} $\zeta_{L,\cdot}$  
as in~\eqref{e:ApproxPot}, and this is carried out in Section~\ref{Subsec:Shape}, and 
then deducing its implications as done in Section~\ref{sec:Tail}. 
Such a description is then employed in the study of the principal eigenpair of $\cH_L$ on $Q_{R_L}$. 
In Section~\ref{Sec:Local}, we show that, thanks to an apriori estimate on the exponential 
decay of the eigenfunctions, we can further localise the eigenproblem to $Q_{r_L}$ (step (c) in~\eqref{e:LocPro}) 
and then, more importantly, that the principal eigenvalue and eigenfunction are respectively 
well-approximated by~\eqref{e:competing} and by the principal eigenfunction of the deterministic operator 
$\bar\cH_L$ in~\eqref{e:DetHam}. It is thanks to step (c) that $\bar\cH_L$ can be taken to be independent  
of the point at which $\xi_L$ achieves its maximum. 
The approximation on $Q_{r_L}$ relies on a simple and effective convex analysis argument 
applied to the local quadratic form (see Lemma~\ref{l:Convex}) that crucially allows to identify the correction 
due to the fluctuation field and that we believe could be of independent interest.

\subsection{Notation and basic Gaussian estimates}

Here we introduce (and recall) some notation and conventions we will be using throughout the paper.
For $x\in\Z^d$, we denote by $|x|\eqdef(\sum_{i=1}^d |x_i|^2)^{1/2}$ 
the $\ell^2$-norm of $x\in\R^d$.
As above, we write $Q_{a}$ for the ($\ell^\infty$-)box of side-length $a$, i.e. $Q_a\eqdef [-a/2,a/2]^d\cap\Z^d$,  
and, for $x,y\in \Z^d$, $\glsadd{Qax} Q_{a,x}\eqdef x+ Q_a$ and $\glsadd{Qaxney} Q_{a,x}^{\neq y}\eqdef Q_{a,x}\setminus\{y\}$.
For any subset $C \subset \Z^d$ and any function $V\colon C\to\R$, we let $\glsadd{cHCV} \cH_{C,V}$ be the operator $\Delta + V$ on $C$ endowed with Dirichlet boundary conditions.

For two sequences $(u_L)_{L\ge 1}, (v_L)_{L\ge 1}$ in $(0,\infty)$, we write $u_L \ll v_L$ to express that $u_L/v_L \to 0$ as $L\to\infty$, we write $u_L \gg v_L$ to express that $v_L \ll u_L$, we write $u_L \asymp v_L$ if $0 < \liminf_{L\to\infty} u_L/v_L \le \limsup_{L\to\infty} u_L/v_L <\infty$, and we write $u_L \sim v_L$ if $u_L/v_L$ converges to $1$ as $L\to\infty$.
\medskip

At last, recall that if $X$ is a standard Gaussian r.v., i.e. $X \eqlaw\cN(0,1)$, then for all $x>0$
\begin{equ}[e:TailGauss]
	\P(X \geq x) \leq \frac1{\sqrt{2\pi} x} e^{-\frac{x^2}{2}}\,,
\end{equ}
and
\begin{equ}[e:TailGaussAsymp]
	\P(X \geq x) = \frac1{\sqrt{2\pi} x} e^{-\frac{x^2}{2}} (1+\cO(1/x^2))\,,\quad x\to\infty\,.
\end{equ}
These immediately provide a more explicit characterisation of 
$a_L$ in~\eqref{e:aL}, i.e. 
\begin{equ}[e:Asympa_L]
	\frac1{\sqrt{2\pi} a_L} e^{-\frac{a_L^2}{2}} \sim \frac1{L^d}\;,\qquad \text{as $L\to\infty$,}
\end{equ}
and a perturbative result for the tail of $X$ around $a_L$, namely, 
for any sequence $(b_L)_{L\ge 1}$ satisfying $|b_L| \ll a_L$, it holds
\begin{equ}[e:preBound]
	\P(X\ge a_L + b_L) \sim \frac{1}{L^d} e^{-a_L b_L - \frac{b_L^2}{2}} \;,\quad L\to\infty\;.
\end{equ}
By Gaussian scaling, it is immediate to adapt the above to the case in which $X$ has variance $\sigma>0$.

\subsection*{Acknowledgements} The work of C.~L. was partially supported by the ANR projects Smooth ANR-22-CE40-0017 and RANDOP ANR-24-CE40-3377, and by the Institut Universitaire de France. G.~C. gratefully acknowledges financial support
via the UKRI FL fellowship ``Large-scale universal behaviour of Random
Interfaces and Stochastic Operators'' MR/W008246/1.

\section{Local behaviour of correlated Gaussian fields}\label{Sec:Field}

\noindent This section is devoted to the study of the potential $\xi_L$ 
introduced in Section~\ref{sec:pot} and the Gaussian field introduced in~\eqref{e:competing}. 
We will begin by spelling out a useful orthogonal decomposition of $\xi_L$ and 
stating a few useful properties thereof (Section~\ref{sec:Prelim}). 
Then, we will study the behaviour of the field on a mesoscopic box of diameter $R_L\ll L$: 
we will first establish an approximation of $\xi_L$ near a large peak (Section~\ref{Subsec:Shape}) 
and use it to identify the behaviour of the maxima of $\xi_L$ and of associated random fields (Section~\ref{sec:Tail}). 
%
%

\subsection{Orthogonal decomposition and basic properties}\label{sec:Prelim}

Let $(\xi_L(x))_{x\in\Z^d}$ be a Gaussian field that satisfies Definition~\ref{d:GF}. 
Recall that $\zeta_{L,x_0}$ is defined in \eqref{e:FluctField} so that
\begin{equ}
	\xi_L(x)=\xi_L(x_0) v_L(x-x_0) +\zeta_{L,x_0}(x),\quad x\in\Z^d\;. 
\end{equ} 
Its main properties are summarised in the next lemma. 

\begin{lemma}\label{l:FluctField}
	For any given $x_0\in\Z^d$, $\zeta_{L,x_0}$ 
	is a centred Gaussian field independent of $\xi_L(x_0)$, satisfying $\zeta_{L,x_0}(x_0)=0$. Its covariance is 
	\begin{equ}[e:covZeta] 
		\E[\zeta_{L,x_0}(x)\zeta_{L,x_0}(y)] = v_L(x-y) - v_L(x-x_0) v_L(y-x_0) \;,
	\end{equ}
	 moreover, 
	there exists $c''>0$ such that for all $x\in \Z^d$
	\begin{equ}[e:VarZeta]
		\var[\zeta_{L,x_0}(x)] = 1 - v_L(x-x_0)^2 
		\le  \frac{e^{c'' |x-x_0|}}{d_L}\,.
	\end{equ}
\end{lemma}
\begin{proof}
	The fact that $\zeta_{L,x_0}$ is a centred Gaussian field, independent of $\xi_L(x_0)$ and such that $\zeta_{L,x_0}(x_0)=0$ is obvious 
	by~\eqref{e:FluctField} and the definition of $v_L$. 
	The expression of the covariance follows from a straightforward computation. Since $v_L$ is bounded by $1$ by~\eqref{e:AlmostDecay}, 
	so is $\var[\zeta_{L,x_0}(x)]$. For~\eqref{e:VarZeta}, we immediately have
	$$ v_L(x-x_0)^2 = (1 + v_L(x-x_0) - 1)^2 \ge 1-2 (1-v_L(x-x_0))\;,$$
	and thus, by \eqref{e:AlmostDecay}, we deduce that there exists $c''>0$ such that for all $x\ne x_0$
	$$ 1-v_L(x-x_0)^2 \le 2 (1-v_L(x-x_0)) \le 2 \frac{e^{\fc' |x-x_0|}}{d_L} \le \frac{e^{c'' |x-x_0|}}{d_L}\;.$$
\end{proof}

As discussed in Section~\ref{sec:MR}, our analysis of the Hamiltonian aims at showing that its main 
eigenvalues can be described in terms of the sum of $\xi_L(\cdot)$ and of an additional term  
whose definition only involves the fluctuation field $\zeta_{L,\cdot}$ in~\eqref{e:FluctField}. 
For $y\in\Z^d$, the latter is given by 
\begin{equ}[e:FL]
\glsadd{Phi_L}
\Phi_L(y)\eqdef \sum_{x\in Q_{r_L}^{\neq 0}} \bar\phi_L(x)^2 \zeta_{L,y}(x+y)\,,
\end{equ}
where $r_L$ is as in~\eqref{e:rL} 
and $\bar\phi_L$ is the principal normalised eigenfunction 
of the operator $\bar\cH_L$ in~\eqref{e:DetHam}. 
%

\begin{lemma}\label{l:XiL}
Let $(\Phi_L(x))_{x\in\Z^d}$ be defined according to~\eqref{e:FL}. Then, the field 
\begin{equ}[e:XiL]
\glsadd{XiL}
\Xi_L(x)\eqdef\xi_L(x)+\Phi_L(x)\,,\qquad x\in\Z^d
\end{equ}
is Gaussian with variance $1+\tau_L^2$ at every point, for $\tau_L$ as in~\eqref{e:tauL}, 
translation invariant and, under Assumptions~\ref{ass:dLaL} and~\ref{a:TechnicalAssum}, 
its covariance function $\glsadd{vLXi} v_L^{\Xi}\colon\Z^d\to\R_+$ satisfies
\begin{equs}[e:LongRangeXi]
\glsadd{cT_vL_Xi}
\cT_{v_L^{\Xi}}\eqdef\sup_{|x|\geq \exp(\sqrt{\ln L})}v_L^{\Xi}(x)\ln|x|\longrightarrow 0\,,\qquad\text{as $L\to\infty$. }
\end{equs}
Furthermore, for any $z \in \Z^d$ such that $|z|>\sqrt{d} \,r_L$, we have 
\begin{equ}[e:CovxiPhi]
\cov(\xi_L(z), \Phi_L(0))\vee\cov(\Phi_L(z), \Phi_L(0))\lesssim \sup_{|x|\geq |z|-\sqrt{d} \,r_L}v_L(x)\,.
\end{equ}
\end{lemma}


\begin{proof}
The fact that $\Xi_L$ is Gaussian, translation invariant and at every point has variance $1+\tau_L^2$ 
is an immediate consequence of the definition of $\xi_L$ and Lemma~\ref{l:FluctField}. 
A direct computation (using translation invariance) shows that the covariance function $v_L^{\Xi}$ equals 
\begin{equs}
v_L^{\Xi}(z)&=\cov(\xi_L(z), \xi_L(0))+2\cov(\xi_L(z), \Phi_L(0))+\cov(\Phi_L(z), \Phi_L(0))\\
&=\E[\xi_L(z)\xi_L(0)]+2\E[\xi_L(z)\Phi_L(0)]+\E[\Phi_L(z)\Phi_L(0)]\\
&=v_L(z)+2\sum_{x\in Q_{r_L}^{\neq 0}}\bar\phi_L(x)^2\Big( v_L(z-x)-v_L(z)v_L(x)\Big)\\
+\sum_{x,y\in Q_{r_L}^{\neq 0}}&\bar\phi_L(x)^2\bar\phi_L(y)^2\Big(v_L(z+x-y)-2v_L(x+z)v_L(y)+v_L(z)v_L(x)v_L(y)\Big)\,.
\end{equs}
To control its decay  
and establish both~\eqref{e:LongRangeXi} and~\eqref{e:CovxiPhi}, let us point out that, for $|z|\geq \sqrt{d} \,r_L$ and 
any $|w|\leq \sqrt{d} \,r_L$, we clearly have
\begin{equ}
v_L(z-w)\leq \sup_{|x|\geq |z|-\sqrt{d} \,r_L}v_L(x)\,.
\end{equ}
Therefore, we immediately deduce 
\begin{equ}
\cov(\xi_L(z), \Phi_L(0))\lesssim \sup_{|x|\geq |z|-\sqrt{d} \,r_L}v_L(x) \sum_{y\in Q_{r_L}^{\neq 0}}\bar\phi_L(y)^2\leq \sup_{|x|\geq |z|-\sqrt{d} \,r_L}v_L(x)
\end{equ}
where we used that $\overline \varphi_L$ is normalised. 
Arguing similarly (and recalling that $v_L \le 1$) for $\cov(\Phi_L(z), \Phi_L(0))$,~\eqref{e:CovxiPhi} follows at once. 

The very same procedure also implies that 
\begin{equ}
\cT_{v_L^{\Xi}}\lesssim \cT_{v_L} +\sup_{|z|\geq \exp(\sqrt{\ln L})} \ln|z|\sup_{|x|\geq |z|-\sqrt{d} \,r_L}v_L(x)\leq 2 \cT_{v_L}
\end{equ}
and, by~\eqref{e:LongRange}, the r.h.s. converges to $0$, which completes the proof.  
%
\end{proof}

In what follows, we will derive the order statistics of both fields $\xi_L$ and $\Xi_L$, and 
study how they relate to each other, which in particular requires to identify the 
order of magnitude of their maxima and the size of the fluctuations around these. 
To do so, it is crucial to understand the mechanism that produces high peaks 
and what behaviour $\xi_L$ must display for $\Xi_L$ to be large. 
Notice that, in view of Lemma~\ref{l:FluctField}, for every $x\in\Z^d$, 
$\Xi_L(x)$ is the sum of two independent Gaussian random variables, $\xi_L(x)$ and $\Phi_L(x)$, 
of variance $1$ and $\tau_L^2$, respectively. 

In the next lemma, we 
address the afore-mentioned questions for generic Gaussian random variables satisfying these features 
and after its statement we will translate its content in our context. 
The proof of the lemma is postponed to Appendix~\ref{a:Gauss}.

\begin{lemma}\label{l:GaussianNew}
	Let $X, \YYL$ be two independent centred Gaussian random variables of variance $1$ and $\tau_L^2$, respectively. Then, for any $s\in\R$, as $L\to\infty$
	\begin{equ}[e:GausSum0]
		\P\Big(X+\YYL \ge a_L\sqrt{1+\tau_L^2} + \frac{s}{a_L} \Big) \sim \frac1{L^d} e^{-s}\;,
	\end{equ}
and 
	then
	\begin{equs}[e:GausSum]
		 \lim_{C\to\infty} \limsup_{L\to\infty} L^d \P\Big(X+\YYL \ge a_L\sqrt{1+\tau_L^2} + \frac{s}{a_L} ; X \notin I_L(C)\Big)= 0\;,
	\end{equs}
	where,  for $L,C>0$, $I_L(C)$ is the interval 
	\begin{equ}[e:ILC]
	I_L(C) \eqdef \Big[\tfrac{a_L}{\sqrt{1+\tau_L^2}} - C\max\{\tfrac{1}{a_L},\tau_L\} , \tfrac{a_L}{\sqrt{1+\tau_L^2}} + C\max\{\tfrac{1}{a_L},\tau_L\}\Big]\,.
	\end{equ}
	As a consequence, for any sequence $(\theta_L)_{L\geq 1}$ such that $\max\{a_L\tau_L^2,a_L^{-1}\}\ll\theta_L$, 
	 we have
	\begin{equ}[e:GausSum2]
		\lim_{L\to\infty}L^d\,\P\Big(X+\YYL \ge a_L\sqrt{1+\tau_L^2} + \frac{s}{a_L} ; |X - a_L|> \theta_L\Big)
	=0\;.
	\end{equ}
\end{lemma}

In terms of $\Xi_L$ and $\xi_L$, we can infer from the above a number 
of useful insights. First, in view of~\eqref{e:GausSum0} and~\eqref{e:aL} we set
\begin{equ}[e:aLXiL]
\glsadd{aLXi}
a_L^{\Xi}=a_L\sqrt{1+\tau_L^2}
\end{equ} 
as it is the counterpart of $a_L$ for $\Xi_L$, and the fluctuations around it are of order $a_L^{-1}$ (i.e. 
the same as those of $\xi_L$, see~\eqref{e:preBound}). 
Second, the estimate~\eqref{e:GausSum} shows that, for any $x\in\Z^d$,  
for $\Xi_L(x)$ to be of order $a_{L}^{\Xi}$, $\xi_L(x)$ must 
be of order $a_L/\sqrt{1+\tau_L^2}$ up to an event of probability negligible compared to $1/L^d$.

\subsection{Deterministic shape around local maxima}\label{Subsec:Shape}

In this section, we study the behaviour of the potential $\xi_L$ around its maxima in a 
mesoscopic box of side-length $R_L$. 
To do so, we introduce an event around which our analysis revolves. 
Recall the definition of the \textit{shape} $\cS_L$ in~\eqref{e:Shape}, i.e. 
$\cS_L(\cdot) \eqdef a_L (1-v_L(\cdot))$, and of the fluctuation field $\zeta_{L,\cdot}$ in~\eqref{e:FluctField}. 

\begin{definition}\label{d:Event}
	Set $\theta\eqdef 2d+1$. For $x_0\in\Z^d$ and a constant $\kappa \in (0,1/3)$, we define $E_{L,x_0}=E_{L,x_0}(\kappa)$ as the intersection of the three events $E^i_{L,x_0}$, $i=1,2,3$, respectively given by 
	\begin{equs}
	\glsadd{ELx01}
		E^1_{L,x_0}&\eqdef\{ |\xi_L(x_0)-a_L|  < \theta\}\label{e:E1} \,, \\
		\glsadd{ELx02}
		E^2_{L,x_0}&\eqdef\{|\zeta_{L,x_0}(x)| \le \frac1{10} \cS_L(x-x_0)\quad\forall\; x\in Q_{2R_L, x_0}\} \,, \label{e:E2}
	\end{equs}
	and 
	\begin{equ}\label{e:E3}
	\glsadd{ELx03}
		E^3_{L,x_0}\eqdef \bigcap_{x\in Q_{R_L,x_0}^{\neq x_0}}
		\bigg\{\frac{|\zeta_{L,x_0}(x)|}{\sqrt{\var[\zeta_{L,x_0}(x)]}} \leq \Big(\frac{a_L}{d_L}\Big)^{\kappa|x-x_0|} \sqrt{1\vee (|\xi_L(x_0)-a_L|a_L)}\bigg\}\;.
	\end{equ}
\end{definition}

Let us make some comments on this definition. 
The event $E^1_{L,x_0}$ forces $\xi_L$ to display a large peak at $x_0$. 
The requirement $\xi_L(x_0)< a_L + \theta$ was added for convenience only and, in any case, 
its complement is unlikely (\textit{i.e.}, its probability is negligible compared to $1/L^d$) and can thus be excluded. 
The event $E^2_{L,x_0}$ ensures that the fluctuation field remains ``small'' around $x_0$ (the value $1/10$ 
at the r.h.s. is arbitrary and anything sufficiently small would do) 
so that, as we will see in Proposition \ref{l:Event}, $\xi_L(x_0)$ is a local maximum and 
$\xi_L(x)$ remains below $\xi_L(x_0)$ around $x_0$. 
Finally, $E^3_{L,x_0}$ prescribes the order of $\zeta_{L,x_0}$ on the box $Q_{R_L,x_0}$. 
It morally requires that $\zeta_{L,x_0}\approx \sqrt{{\rm Var}(\zeta_{L,x_0})}$ up to an error 
which suitably depends on the size of the fluctuations of $\xi_L(x_0)$ around $a_L$ 
(which in turn are expected to be $\cO(a_L^{-1})$, see Theorem~\ref{Th:TailJointLaw} below). 


The specific choice of $\theta$ and the control over the $\zeta_{L,x_0}$'s will become 
clearer in the proof of the next two results as well as those of Lemmas~\ref{l:RedE},~\ref{l:Shrink}  (regarding $\theta$ 
for the sake of this section it could have simply been chosen as $\theta>0$). 

\begin{proposition}\label{l:Event}
For $x_0\in \bbZ^d$, let $E_{L,x_0}$ be the event in Definition~\ref{d:Event}. 
Then, there exists an $L_0\geq 1$ such that for all $L\geq L_0$, 
\begin{enumerate}[label=\normalfont{(\arabic*)}]
\item\label{i:event1} on $E_{L,x_0}$, $\xi_L$ admits a unique maximum over $Q_{2R_L, x_0}$ which is attained at $x_0$ and the following bound holds
\begin{equ}[e:LowerShape]
	\xi_L(x) - \xi_L(x_0) \le - \frac{\fc}{2}\frac{a_L}{d_L}\;,\quad \forall x\in Q_{2R_L, x_0}^{\neq x_0}\;,
\end{equ}
where $\fc>0$ is as in~\eqref{e:AlmostDecay}. 
\item\label{i:event2} for any $y_0 \ne x_0$ such that $|x_0 - y_0| \le \sqrt{d} \,r_L$, $E_{L,x_0}\cap E_{L,y_0}=\emptyset$. 
\end{enumerate}
Furthermore, there is $C>0$ such that for all $L\ge L_0$ and $x_0\in\Z^d$, we have
\begin{equ}[e:nonE]
\P(\xi_L(x_0) \ge a_L - \theta ; E_{L,x_0}^\complement )
		\le  \frac{1}{L^d}e^{-C \big(\frac{a_L}{d_L}\big)^{2\kappa}}\,,
	\end{equ}
and 
\begin{equ}[e:MaxNegl]
\lim_{L\to\infty} \Big(\frac{L}{R_L}\Big)^d \P\Big(\max_{x\in Q_{R_L+ r_L}}\xi_L(x)\geq a_L-\theta\,;\,\Big(\bigcup_{x_0\in Q_{R_L}} E_{L,x_0}\Big)^\complement\Big)=0\,.
\end{equ}
\end{proposition}

\begin{remark}\label{rem:Likely}
The rule of thumb underlying our probabilistic analysis is that we can neglect any event ``based at $x_0$'' whose probability is negligible compared to $1/L^d$. Indeed, the union over $x_0\in Q_L$ of such events has a probability which is then negligible compared to $1$.
Also,~\eqref{e:nonE} combined with \eqref{e:preBound}, ensures that
\begin{equ}
	\P(E_{L,x_0}) = \P(\xi_L(x_0) \ge a_L - \theta ; E_{L,x_0} ) \sim \P(\xi_L(x_0) \ge a_L - \theta) \sim \frac{e^{\theta a_L - \frac{\theta^2}{2}}}{L^d}\;.
\end{equ}
\end{remark}

\begin{proof}
We start with~\eqref{e:LowerShape}. Equations~\eqref{e:FluctField},~\eqref{e:E2} and~\eqref{e:Shape} together with the fact that $\xi_L(x_0)> a_L -\theta$ imply that on the event $E_{L,x_0}$ for any $x\in Q_{2R_L,x_0}^{\ne x_0}$
\begin{equs}
		\xi_L(x) - \xi_L(x_0)&\leq \xi_L(x_0)\big(v_L(x-x_0)-1\big) + \frac1{10} \cS_L(x-x_0)\\
		&\leq - \frac{\xi_L(x_0)}{a_L}\cS_L(x-x_0) + \frac1{10} \cS_L(x-x_0)\\
		&\leq -\Big(1-\frac{\theta}{a_L}-\frac1{10}\Big)\cS_L(x-x_0) \le -\frac{\fc}{2} \frac{a_L}{d_L}\;.
\end{equs}
where we used~\eqref{e:AlmostDecay} and, in the last step, 
that, for $L$ large enough, the quantity in parenthesis is larger than $1/2$. This establishes~\eqref{e:LowerShape}, 
which in turn implies both properties~\ref{i:event1} and~\ref{i:event2}. 
\medskip

Now, assume~\eqref{e:nonE} and observe that
\begin{equs}
	\ &\P\Big(\max_{x\in Q_{R_L+ r_L}}\xi_L(x)\geq a_L-\theta\,;\,\Big(\bigcup_{x_0\in Q_{R_L}} E_{L,x_0}\Big)^\complement\Big)\\
	&\le \sum_{x\in Q_{R_L+ r_L} \backslash Q_{R_L}} \P(\xi_L(x) \ge a_L - \theta) + \sum_{x_0\in Q_{R_L}} \P(\xi_L(x_0) \ge a_L - \theta; E_{L,x_0}^\complement)\\
	&\lesssim r_L R_L^{d-1} \frac{e^{a_L \theta}}{L^d} + \frac{R_L^d}{L^d}e^{-C \big(\frac{a_L}{d_L}\big)^{2\kappa}}\;,
\end{equs}
where we have also used~\eqref{e:preBound}. Since $\ln r_L \ll a_L \ll \ln R_L$, 
we deduce that the r.h.s. is negligible compared to $R_L^d/L^d$. Hence, it remains to argue~\eqref{e:nonE}. 

For this,  we show that each of the summands at the r.h.s. of  
	\begin{equ}
		\P\big(\xi_L(x_0) \ge a_L - \theta ; (E_{L,x_0})^\complement \big)\leq \sum_{i=1}^3 \P\big(\xi_L(x_0) \ge a_L - \theta ; (E_{L,x_0}^i)^\complement \big)
	\end{equ}
	can be bounded above by a term of the desired order.
	Now, the first can be controlled via~\eqref{e:preBound}, which gives  
	\begin{equ}
		\P\big(\xi_L(x_0) \ge a_L - \theta ; (E_{L,x_0}^1)^\complement \big)=\P\big(\xi_L(x_0) \ge a_L + \theta\big) \lesssim \frac{e^{-\theta a_L}}{L^d} (1+o(1))
	\end{equ}
	and since $a_L\gg (a_L/d_L)^{2\kappa}$, the r.h.s. is bounded above by the r.h.s. of~\eqref{e:nonE}. 
	
	For the second, we recall that, by Lemma~\ref{l:FluctField}, $\xi_L(x_0)$ and $\zeta_{L,x_0}$ are independent, 
	thus so are the events $\{\xi_L(x_0) \ge a_L - \theta\}$ and $E^2_{L,x_0}$. Hence, 
	\begin{equ}
		\P\big(\xi_L(x_0) \ge a_L - \theta ; (E_{L,x_0}^2)^\complement \big)=\P\big(\xi_L(x_0) \ge a_L - \theta\big)\P\big( (E_{L,x_0}^2)^\complement \big)
	\end{equ} 
	and we only need to focus on the latter factor. A union bound yields
	\begin{equs}
		\P\big( (E_{L,x_0}^2)^c \big)&\leq \sum_{x\in Q_{2R_L,x_0}}\P\Big(|\zeta_{L,x_0}(x)|\geq \frac1{10} \cS_L(x-x_0)\Big)\,.
	\end{equs}
	Since $v_L\le 1$, we find
	\begin{equs}
		\frac{\cS_L(x-x_0)^2}{{\mathrm{Var}[\zeta_{L,x_0}(x)]}} &= \frac{a_L^2 \big[1-v_L(x-x_0)\big]^2}{(1-v_L(x-x_0)^2)}= \frac{a_L^2 \big[1-v_L(x-x_0)\big]}{(1+v_L(x-x_0))}\geq \frac{\fc}{2}\frac{a_L^2}{ d_L} \,,
	\end{equs}
	the last step being a consequence of~\eqref{e:AlmostDecay}. Note that the r.h.s. is larger than $1$ for all $L$ large enough, 
	so that~\eqref{e:TailGauss} and the definition of $R_L$ yield
	\begin{equs}
		\sum_{x\in Q_{2R_L,x_0}}\P\Big(|\zeta_{L,x_0}(x)|\geq \frac1{10} \cS_L(x-x_0)\Big)&\leq 
		\sum_{x\in Q_{2R_L,x_0}} \exp\Big(-\frac{\cS_L(x-x_0)^2}{200{\mathrm{Var}[\zeta_{L,x_0}(x)]}}\Big)\\
		&\leq \#Q_{2R_L} e^{-\fc\frac{a_L^2}{400 d_L}}\leq e^{-C'\frac{a_L}{d_L} a_L}
	\end{equs}
	for some constant $C'>0$, where we used~\eqref{e:RL}.
	
	We turn to the event $E^3_{L,x_0}$. Let $\sigma_{L,x_0}(x)\eqdef \sqrt{{\rm Var}[\zeta_{L,x_0}(x)]}$. Then 
	\begin{equs}
		\,&\P\big(\xi_L(x_0) \ge a_L - \theta ; (E^3_{L,x_0})^\complement \big)\label{e:E3Bound}\\
		&\leq\sum_{x\in Q_{R_L,x_0}^{\neq x_0}}  \P\Big(\xi_L(x_0) \ge a_L - \theta ; \frac{|\zeta_{L,x_0}(x)|}{\sigma_{L,x_0}(x)} > \Big(\frac{a_L}{d_L}\Big)^{\kappa|x-x_0|}\sqrt{1\vee (|\xi_L(x_0)-a_L|a_L)} \Big)\;.
	\end{equs}
	By translation invariance, the probability at the r.h.s.~does not depend on $x_0$. 
	The independence of $\zeta_{L,x_0}(x)$ and $\xi_L(x_0)$ then implies 
	\begin{equs}
		\P&\Big(\xi_L(x_0) \ge a_L - \theta ; \frac{|\zeta_{L,x_0}(x)|}{\sigma_{L,x_0}(x)} > \Big(\frac{a_L}{d_L}\Big)^{\kappa|x-x_0|}\sqrt{1\vee (|\xi_L(x_0)-a_L|a_L)} \Big)\\
		&=\int_{ a_L-\theta}^{\infty} \frac1{\sqrt{2\pi}}e^{-\frac{z^2}{2}} \P\Big(\frac{|\zeta_{L,x_0}(x)|}{\sigma_{L,x_0}(x)}> \Big(\frac{a_L}{d_L}\Big)^{\kappa|x-x_0|}\sqrt{1\vee (|z-a_L|a_L)} \Big) \dd z\\
		&\le \int_{a_L-\theta}^{\infty} \frac1{\sqrt{2\pi}}\exp\Big(-\frac{z^2}{2} - \frac12\Big(\frac{a_L}{d_L}\Big)^{2\kappa|x-x_0|}{1\vee (|z-a_L|a_L)}\Big) \dd z\,.	
	\end{equs}
	where we have used \eqref{e:TailGauss} to go from the second to the third line. 
	We now split the domain of integration into $I_1 \eqdef [a_L - \theta, a_L - \frac{1}{a_L} ] $ and $I_2 \eqdef (a_L - \frac1{a_L},\plusinfty)$. Using \eqref{e:preBound}, the integral over $I_2$ is bounded by
	\begin{equ}
		e^{- \frac12(\frac{a_L}{d_L})^{2\kappa|x-x_0|}}\P(\cN(0,1) \ge a_L - \frac1{a_L}) \lesssim \frac1{L^d} e^{1 - \frac12(\frac{a_L}{d_L})^{2\kappa|x-x_0|}}\,, 
	\end{equ}
	while the integral over $I_1$ can be rewritten as (take $y = (a_L-z)a_L$)
	\begin{equs}
		\int_1^{\theta a_L}&\frac{1}{\sqrt{2\pi} a_L} \exp\Big(-\frac{a_L^2}{2} - \frac{y^2}{2a_L^2} - y\Big(\frac12 \Big(\frac{a_L}{d_L}\Big)^{2\kappa|x-x_0|} -1\Big)\Big) \dd y\\
		&\lesssim \frac{1}{L^d} \int_1^{\infty} \exp\Big(- \frac{y}{4} \Big(\frac{a_L}{d_L}\Big)^{2\kappa|x-x_0|} \Big) \dd y\lesssim \frac1{L^d} \exp\Big(-\frac14 \Big(\frac{a_L}{d_L}\Big)^{2\kappa|x-x_0|}\Big)\,.
	\end{equs}
	thanks to $d_L\ll a_L$ and \eqref{e:Asympa_L}. 
	Plugging these estimates into~\eqref{e:E3Bound} we are left to control
	\begin{equs}
		\frac{1}{L^d}\sum_{x\in Q_{R_L,x_0}^{\neq x_0}}\exp\Big(-\frac14 \Big(\frac{a_L}{d_L}\Big)^{2\kappa|x-x_0|}\Big)=\frac{1}{L^d}\sum_{x\in Q^{\neq 0}_{R_L}}\exp\Big(-\frac14 \Big(\frac{a_L}{d_L}\Big)^{2\kappa|x|}\Big) \,,
	\end{equs}
	from which~\eqref{e:nonE} follows at once.  
\end{proof}

\begin{calc}
For the last estimate: 
There exists a $C>0$ such that for all  $b\in (0,1)$ and $R>0$
\begin{equs}
\label{eqn:sum_estimate} \tag{$\star$}
\sum_{x\in Q^{\neq 0}_{R}} b^{|x|} \le C b. 
\end{equs}
We can estimate the above sum by a multiple of following:
\begin{align*}
\int_1^\infty r^{d-1} e^{r \ln b} \dd r 
& = \frac{1}{\ln b} [r^{d-1} e^{r \ln b}] |_1^\infty - (d-1) \int_1^\infty r^{d-2} e^{r \ln b} \dd r \\
& = \frac{b}{\ln b} - \frac{d-1}{(\ln b)^2} \int_1^\infty r^{d-2} \frac{\mathrm{d}}{\mathrm{d} r} e^{r \ln b} \dd r, 
\end{align*}
so we get something like $b P((\ln b)^{-1})$, where $P$ is a polynomial. As $(\ln b)^{-1} <1$, $P((\ln b)^{-1}) \lesssim 1$. 
\end{calc}

We now derive two properties of the maxima of the fields 
$\xi_L$ and $\Xi_L$ (whose definition is in \eqref{e:XiL}) on $Q_{R_L}$, which are implied by the previous lemma. 
Before that, since thanks to~\eqref{e:preBound} and~\eqref{e:GausSum0}, we know 
that the maxima of $\xi_L$ and $\Xi_L$ are of respective sizes $a_L$ in~\eqref{e:aL} and $a_L^{\Xi}$ 
in~\eqref{e:aLXiL}, and that the order of the fluctuations around them is $a_L^{-1}$, 
let us introduce the rescaled fields
\begin{equ}[e:ScaledFields]
\glsadd{Thetaxi}
\Theta_L^{\xi}(x)\eqdef a_L\big(\xi_L(x)-a_L\big)\,,\quad \glsadd{ThetaXi} \Theta_L^{\Xi}(x)\eqdef a_L\big(\Xi_L(x)-a_L^{\Xi}\big)\,,\qquad x\in\Z^d\;.
\end{equ}

At first, we want to show that on a box of size $R_L$ 
both $\Theta^{\xi}_L$ and $\Theta_L^{\Xi}$ have at most one point at which they are uniformly bounded from below (in contrast, at any other points these fields will take large, negative values) and this is the point at which they achieve their maxima.
For this, for $z\in\Z^d$ and $s\in\R$, we will show that the events 
\begin{equ}[e:AjLXi]
\glsadd{AChi}
\cA^{\chi}_{L,z}(s)\eqdef \bigcup_{x \ne y \in Q_{R_L,z}} \Big\{\Theta_L^{\chi}(x)\geq -s\;;\; \Theta_L^{\chi}(y)\geq -s\Big\}\,,
\end{equ}
for $\chi$ either $\xi$ or $\Xi$,  
are asymptotically negligible. 

To phrase the second property, for $z\in\Z^d$, let
\begin{equ}[e:wLwXiL]
\glsadd{wLz}
w_{L,z}\eqdef\argmax_{Q_{R_L,z}}\xi_L\qquad\text{and}\qquad \glsadd{wLzXi} w_{L,z}^{\Xi}\eqdef\argmax_{Q_{R_L,z}}\Xi_L\,.
\end{equ}
Then, we want to verify that, provided the maximum of $\Xi_L$ is ``large'', $w_L=w_L^{\Xi}$ with high probability. 

For both features, an important step is to show that 
if the maximum of $\Xi_L$ is of order $a_L^{\Xi}$, then the maximum 
of $\xi_L$ must be of order $a_L$. This is the first point of the next lemma. 
The argument we will exploit 
uses the exponential decay of the principal normalised eigenfunction $\bar\phi_L$ 
of the operator $\bar\cH_L$ in~\eqref{e:DetHam}, which is a standard fact (independent of the specific setting of 
the present paper) and will anyway be detailed in 
Lemma~\ref{l:DecayPsi}. For the reader's convenience, let us anticipate that this amounts 
to say that there exists a constant $C>0$ such that
\begin{equ}[e:DecayPhiL]
\bar\phi_{L}(x)^2\leq C\Big(\frac{a_L}{d_L}\Big)^{-2|x|}\qquad \forall\, x\in Q^{\neq 0}_{r_L}\,.
\end{equ}

\begin{lemma}\label{l:UniqueMax}
The following statements hold: 
\begin{enumerate}[label=\normalfont{(\roman*)}]
\item\label{i:OrderMax} For $\theta=2d+1$ as in Definition~\ref{d:Event} and for any $s\in \R$, there exists $L_0\geq 1$ such that for every $L\geq L_0$ we have
\begin{equ}[e:OrderMax]
\Big\{\max_{x\in Q_{R_L,z}} \Xi_L(x)\geq a_L^{\Xi}-\frac{s}{a_L}\Big\} \cap \Big\{\max_{x\in Q_{R_L+r_L,z}}\xi_L(x)\leq a_L-\theta\Big\}=\emptyset\,,
\end{equ}
\item\label{i:UniqueMax} For every $z\in\Z^d$ and every $s\in\R$, it holds 
\begin{equ}[e:UniqueMax]
\lim_{L\to\infty}\Big(\frac{L}{R_L}\Big)^d\P\Big(\cA^{\chi}_{L,z}(s)\Big)=0\,,
\end{equ}
where the events $\cA^{\chi}_{L,z}(s)$ are defined according to~\eqref{e:AjLXi} and
$\chi$ is either $\xi$ or $\Xi$, 
\item\label{i:SameMax} 
For every $z\in \Z^d$, with $w_{L,z}$ and $w_{L,z}^{\Xi_L}$ as in~\eqref{e:wLwXiL}, we have
\begin{equ}[e:SameMax]
\lim_{L\to\infty}\Big(\frac{L}{R_L}\Big)^d \P\Big(\max_{x\in Q_{R_L,z}}\Xi_L(x)\geq a_L^{\Xi}-\frac{s}{a_L}\,;\, w_{L,z}\neq w_{L,z}^{\Xi}\Big)=0\,.
\end{equ}
\end{enumerate}
\end{lemma}

\begin{remark}
The reason why in~\eqref{e:OrderMax} the maximum of $\xi_L$ is taken on $Q_{R_L+r_L}$, while that of $\Xi_L$ on $Q_{R_L}$, 
is that, by definition, see \eqref{e:XiL} and \eqref{e:FL}, $\Xi_L$ depends on the values of $\xi_L$ on a box of side-length $R_L+r_L$. 
\end{remark}

\begin{proof}
W.l.o.g. we can (and will) take $z=0$ throughout the proof and omit the corresponding subscript, i.e.  $\cA^{\chi}_{L}(s)= \cA^{\chi}_{L,0}(s)$, $w_L=w_{L,0}$, etc. 
We begin by proving~\ref{i:OrderMax}. 
On $\{\max_{Q_{R_L+r_L}}\xi_L\leq a_L-\theta\}$, using 
that by definition, $\Xi_L =\xi_L+\Phi_L$~\eqref{e:XiL}, where $\Phi_L$ is as in~\eqref{e:FL}, and recalling $\zeta_{L,\cdot}$ from~\eqref{e:FluctField}, 
we can control $\Xi_L$ at any $y\in Q_{R_L}$ as
\begin{equs}
\Xi_L(y)&=\xi_L(y) \Big( 1 - \sum_{x\in Q_{r_L}^{\ne 0}} \bar\phi_L(x)^2 v_L(x)\Big) + \sum_{x\in Q_{r_L}^{\ne 0}} \bar\phi_L(x)^2 \xi_L(y+x)\\
		&\le (a_L-\theta) \Big( 1 + \sum_{x\in Q_{r_L}^{\ne 0}} \bar\phi_L(x)^2 (1-v_L(x))\Big)\,.
		\end{equs}
Now, the lower bound in~\eqref{e:AlmostDecay} and the bound in~\eqref{e:DecayPhiL} 
ensure that the remaining sum can be bounded by 
\begin{equs}
	\sum_{x\ne 0} \bar\phi_L(x)^2 (1-v_L(x))&\leq \frac{1}{d_L}\sum_{x\in Q_{r_L}^{\neq 0}} \Big(\frac{d_L}{a_L}\Big)^{2|x|} e^{\fc'|x|}\lesssim \frac{d_L}{a_L^2}\,.
\end{equs}
Since further $d_L\ll a_L$ by Assumption~\ref{ass:dLaL}, we deduce for large $L$
\begin{equs}
\Xi_L(y)&\leq (a_L-\theta)\Big(1+C\Big(\frac{d_L}{a_L^2}\Big)\Big)\leq a_L-\theta +C\Big(\frac{d_L}{a_L}\Big)\\
&<a_L-\frac{\theta}{2}<a_L\sqrt{1+\tau_L^2}-\frac{s}{a_L}=a_L^{\Xi}-\frac{s}{a_L}\;.
\end{equs}
Henceforth, if $\max\xi_L\leq a_L-\theta$, also $\max\Xi_L<a_L^{\Xi}-s/a_L$ and~\ref{i:OrderMax} is proved. 
\medskip

 Let us now turn to the proof of \ref{i:UniqueMax}.  

Notice that for $L$ large enough, $\cA^{\chi}_L$ is contained in $\{\max_{Q_{R_L+r_L}}\xi_L\geq a_L-\theta\}$: 
for $\chi=\xi$ this holds by definition, while for $\chi=\Xi$ it follows from 
$\cA_L^{\Xi}\subset \{\max_{Q_{R_L}}\Xi_L\geq a_L^{\Xi}-s/a_L\}$ 
and~\eqref{e:OrderMax}. As a consequence, 
\begin{equs}
\P\Big(\cA_L^{\chi}\cap \Big(\bigcup_{x_0\in Q_{R_L}} E_{L,x_0}\Big)^\complement\Big)\leq  \P\Big(\max_{x\in Q_{R_L+r_L}}\xi_L(x)\geq a_L-\theta\,;\,\Big(\bigcup_{x_0\in Q_{R_L}} E_{L,x_0}\Big)^\complement\Big) \,, 
\end{equs}
and, by~\eqref{e:MaxNegl} since $r_L\ll R_L$, the r.h.s.~is negligible compared to $R_L^d/L^d$. 
Thus, 
\begin{equ}[e:UniqueMaxInter]
\P(\cA^{\chi}_L)= \P\Big(\cA^{\chi}_L\cap \bigcup_{x_0\in Q_{R_L}} E_{L,x_0}\Big)+o\Big(\frac{R_L^d}{L^d}\Big)\,. 
\end{equ}
Now, for $\chi=\xi$, we exploit property~\ref{i:event2} from Proposition~\ref{l:Event} which  implies that 
\begin{equs}
\P\Big(\cA^{\xi}_L\cap \bigcup_{x_0\in Q_{R_L}} E_{L,x_0}\Big)=\sum_{x_0\in Q_{R_L}}\P(\cA^{\xi}_L\cap E_{L,x_0})\,.
\end{equs}
We have $\cA^{\xi}_L\cap E_{L,x_0}=\emptyset$. Indeed, on $E_{L,x_0}$, $\xi_L$ has a unique maximum at $x_0$ and $\xi_L$ is ``small'' nearby due to~\eqref{e:LowerShape}, 
and we conclude that~\eqref{e:UniqueMax} holds. 
\begin{calc}
On $E_{L,x_0}$ we have for $x$ with $1\le |x-x_0| \le 2R_L$, and for large $L$, 
\begin{align*}
\xi_L(x) \le \xi_L(x_0) - \frac{\fc}{2} \frac{a_L}{d_L} \le a_L + \theta - \frac{\fc}{2} \frac{a_L}{d_L} < a_L - \frac{s}{a_L} \,.
\end{align*}
Hence, no other $x$ in $Q_{R_L}$ than $x_0$ itself can be such that $\xi_L(x) \ge a_L - \frac{s}{a_L}$, i.e. $\Theta_L^\xi(x) \ge -s$. 
\end{calc}

For $\chi=\Xi$ instead, let us introduce the event $\cF_L$ which is defined by 
\begin{equ}
\cF_L\eqdef \bigcup_{y\in Q_{R_L}}\Big\{\Xi_L(y)\geq a_L^{\Xi}-\frac{s}{a_L}\,;\,\xi_L(y)<a_L-\theta\Big\}\,.
\end{equ}
Thanks to~\eqref{e:GausSum2}, its probability is negligible compared to $R_L^d/L^d$, so that    
\begin{equs}
\P\Big(\cA^{\Xi}_L\cap \bigcup_{x_0\in Q_{R_L}} E_{L,x_0}\Big)&\leq \P(\cF_L)+\P\Big(\cA^{\Xi}_L\cap \bigcup_{x_0\in Q_{R_L}} E_{L,x_0}\cap\cF_L^\complement\Big)\\
&\leq o\Big(\frac{R_L^d}{L^d}\Big)+\sum_{x_0\in Q_{R_L}}\P\Big(\cA^{\Xi}_L\cap E_{L,x_0}\cap\cF_L^\complement\Big)\,.
\end{equs}
By \eqref{e:LowerShape}, on $E_{L,x_0}$, 
$\xi_L(y) < a_L + \theta - \frac{\fc}{2}(a_L/d_L)<a_L-\theta$ for $L$ large enough and any $y\neq x_0$, 
so that on $E_{L,x_0}\cap\cF_L^\complement$, we must necessarily have that 
for every $y\neq x_0$, $\Xi_L(y)<a_L^{\Xi}-s/a_L$. But this means that 
$\cA^{\Xi}_L\cap E_{L,x_0}\cap\cF_L^\complement=\emptyset$ which, together with~\eqref{e:UniqueMaxInter}, 
implies that~\eqref{e:UniqueMax} holds also for $\chi=\Xi$. 
\medskip

At last, we show~\ref{i:SameMax}. 
We already know from~\eqref{e:OrderMax} that
$\{\max\Xi_L\geq a_L^{\Xi}-s/a_L\,;\, w_{L}\neq w_{L}^{\Xi_L}\}$ 
is contained in $\{\max\xi_L\geq a_L-\theta\}$, so that, by~\eqref{e:MaxNegl}, the probability 
of the former is the same as that of its intersection with the union of the $E_{L, x_0  }$ 
up to an error negligible compared to $R_L^d/L^d$. Moreover, by~\eqref{e:UniqueMax} 
for $\chi = \Xi$ and~\eqref{e:GausSum2}, we have that $\P(\cA^{\Xi}_L)$ and $\P(\cF_L)$ are negligible compared to $R_L^d/L^d$. 
Putting all these together, we deduce 
\begin{equs}
 \,&\P\Big(\max_{x\in Q_{R_L}}\Xi_L\geq a^{\Xi}_L-\frac{s}{a_L}\,;\, w_{L}\neq w_{L}^{\Xi_L}\Big)\\
 &= \P\Big(\max_{x\in Q_{R_L}}\Xi_L\geq a^{\Xi}_L-\frac{s}{a_L}\,;\, w_{L}\neq w_{L}^{\Xi_L}\,;\,\bigcup_{x_0\in Q_{R_L}}E_{L,x_0}\,;\,(\cA^{\Xi}_L)^\complement\,;\,\cF_L^\complement\Big)+o\Big(\frac{R_L^d}{L^d}\Big)\\
 &=\sum_{x_0\in Q_{R_L}}\P\Big(\max_{x\in Q_{R_L}}\Xi_L\geq a^{\Xi}_L-\frac{s}{a_L}\,;\, x_0\neq w_{L}^{\Xi_L}\,;\,E_{L,x_0}\,;\,(\cA^{\Xi}_L)^\complement\,;\,\cF_L^\complement\Big)+o\Big(\frac{R_L^d}{L^d}\Big)\,.
\end{equs}
But now, each of the summands above is $0$. Indeed, on $(\cA^{\Xi}_L)^\complement$ there is at most one point on $Q_{R_L}$ at which $\Xi_L$ is above $a^{\Xi}_L-s/a_L$, and, for every $x_0\in Q_{R_L}$, 
on $E_{L,x_0}\cap \cF_L^\complement$, $\max_{y\neq x_0}\Xi_L(y)<a_L^{\Xi}-s/a_L$, 
which implies that the only point at which $\Xi_L$ can be above $a_L^{\Xi}-s/a_L$ is $x_0$. 
Thus, $w_{L}^{\Xi_L}=x_0$ and the intersection of the events is empty. 
\end{proof}

\subsection{Tail distribution of the maximum}\label{sec:Tail}

A first major consequence of the analysis carried out in the previous section 
is that it provides a rather simple way to determine the tail distributions of $\xi_L(w_L)$, $\Xi_L(w_L)$ 
and of the couple $(\xi_L(w_L), \Phi_L(w_L))$ where $\xi_L$ is our potential, 
$\Xi_L$ and $\Phi_L$ are the fields respectively defined in~\eqref{e:XiL} and~\eqref{e:FL}, 
and $w_L$ is the point in $Q_{R_L}$ at which $\xi_L$ achieves its maximum, i.e. $w_L=w_{L,0}$ and the latter is 
given in~\eqref{e:wLwXiL}. 
The next theorem  states  the rigorous statement we are after.

\begin{theorem}\label{Th:TailJointLaw}
As $L\to\infty$, the Radon measures
\begin{equ}[e:Meas]
\Big(\frac{L}{R_L}\Big)^d  \P\Big(a_L\big(\xi_L(w_L)-a_L\big) \in \dd u \Big)\;, \mbox{ and }\Big(\frac{L}{R_L}\Big)^d  \P\Big(a_L\big(\Xi_L(w_L)-a_L^{\Xi}\big) \in \dd u \Big)
\end{equ}
on $(-\infty,\plusinfty]$ converge vaguely to $e^{-u}\dd u$. 
Furthermore, if $a_L\tau_L\sim \sqrt{b}$ for  $b > 0$, then 
the Radon measure on $(-\infty,\plusinfty] \times [-\infty,\plusinfty]$ given by 
\begin{equ}[e:MeasJoint]
\Big(\frac{L}{R_L}\Big)^d  \P\Big(\Big(a_L\big(\xi_L(w_L)-a_L\big)\,,\,a_L\Phi_L(w_L)\Big) \in \dd u\otimes\dd v \Big)\;,
\end{equ}
converges vaguely to $e^{-u} \dd u \otimes \frac{1}{\sqrt{2\pi b}}e^{-\frac{v^2}{2b}} \dd v$. 
\end{theorem}

\begin{remark}\label{rem:b}
The restriction $a_L\tau_L\sim \sqrt{b}$ for  $b> 0$  in the second part of the statement 
can be easily lifted (and the proof would be unaffected) to cover also the case $\tau_L\gg a_L^{-1}$. 
This would require to scale the second component in~\eqref{e:MeasJoint} to be 
$\Phi_L(w_L)/\tau_L$ instead of $a_L\Phi_L(w_L)$ 
and take $b=1$ in the limiting measure. 

The reason why we stated the result as such is that the joint convergence will 
only be needed when $\tau_L=\cO(a_L^{-1})$, in which case the scaling proposed is more meaningful 
(see the proof of point~\ref{i:main3} in Theorem~\ref{Th:MainRevisited}). 
\end{remark}

\begin{remark}\label{rem:Vague}
The convergence in the second part of the statement means that for any bounded continuous function 
$f\colon(-\infty,\plusinfty] \times [-\infty,\plusinfty]\to\R$ that vanishes outside some set $(-c,\plusinfty]\times[-\infty,\plusinfty]$ 
for some $c>0$, we have
\begin{equ}
\Big(\frac{L}{R_L}\Big)^d \E\Big[f\Big(a_L\big(\xi_L(w_L)-a_L\big)\,,\,a_L\Phi_L(w_L)\Big)\Big] \to \int f(u,v) e^{-u} \dd u \otimes \frac{1}{\sqrt{2\pi b}}e^{-\frac{v^2}{2b}} \dd v\;, 
\end{equ}
and similarly for the first part of the statement.
\end{remark}

\begin{proof}
We present a joint proof of the convergence of the measures in~\eqref{e:Meas} and 
in~\eqref{e:MeasJoint}. Let $\Theta_L^{\xi}$ and $\Theta_L^{\Xi}$ be as in~\eqref{e:ScaledFields} and 
similarly set 
\begin{equ}
\Theta_L^{(\xi,\Phi)}(x) \eqdef \Big(a_L\big(\xi_L(x)-a_L\big)\,,\,a_L\Phi_L(x)\Big)\;, \qquad x\in\Z^d\,.
\end{equ}
Below $\chi$ will denote either $\xi,\Xi$ or $(\xi,\Phi)$.


Now, in either case the limit measures have no atoms, thus it suffices to determine 
the behaviour of the measures when evaluated at sets of the form $I^{\chi}=(u,\infty]$, for $\chi=\xi,\Xi$, and 
$I^\chi = (u,\infty]\times(v,\infty]$ for $\chi=(\xi,\Phi)$, with $u,v\in\R$. 

We claim that  
\begin{equs}\label{Eq:JointSglePoint}
\P\Big(\Theta^{\chi}_L(w_L) \in I^{\chi} \Big)=\sum_{x_0\in Q_{R_L}}\P\Big(\Theta^{\chi}_L(x_0) \in I^{\chi} \Big)+o\Big(\frac{R_L^d}{L^d}\Big)\,.
\end{equs}
Before proving~\eqref{Eq:JointSglePoint}, let us see how it implies the result. 
For $\chi=\xi$ or $\Xi$ 
this is an immediate consequence of the fact that the law of $\xi_L(x_0)$ and $\Xi_L(x_0)$ 
is independent of $x_0$, and of~\eqref{e:preBound} and~\eqref{e:GausSum0} respectively. 
For the other, the independence of $\xi_L(x_0)$ and $\zeta_{L,x_0}$ stated in Lemma~\ref{l:FluctField} 
implies 
\begin{equs}
		\P\Big(\Theta^{\xi}_L(x_0) > u ; a_L\Phi_L(x_0) > v \Big)&= \P\big(a_L(\xi_L(x_0) - a_L) > u\big) \P\Big(a_L\Phi_L(x_0) > v \Big)\\
		&= \frac{1}{L^d}e^{-u}(1+o(1))\int_v^{\plusinfty} \frac{1}{\sqrt{2\pi b}}e^{-\frac{w^2}{2b}} \dd w\,,
\end{equs}
where in the last equality, we used~\eqref{e:preBound} on the first factor, 
and the fact that $a_L\Phi_L(x_0)$ is a centred Gaussian random variable with variance $(a_L\tau_L)^2 \sim b$ 
on the second.

Thus, it remains to show the claim. 
Note that  $\{ \Theta^{\chi}_L(w_L)\in I^{\chi}\}\subset\{\max_{Q_{R_L+r_L}} \xi_L\geq a_L-\theta\}$, 
for $\chi=\Xi$, by Lemma~\ref{l:UniqueMax} while for $\chi=\xi,(\xi,\Phi)$, by definition. 
Thus, thanks to~\eqref{e:MaxNegl} we have 
 \begin{equs}[e:QuasiQuasi]
\P(\Theta^{\chi}_L(w_L)\in I^{\chi})=&\P\Big(\{\Theta^{\chi}_L(w_L)\in I^{\chi}\}\cap \bigcup_{x_0\in Q_{R_L}} E_{L,x_0}\Big)\\
&+ \P\Big(\{\Theta^{\chi}_L(w_L)\in I^{\chi}\}\cap \Big(\bigcup_{x_0\in Q_{R_L}} E_{L,x_0}\Big)^\complement\Big)\\
=&\sum_{x_0\in Q_{R_L}} \P\Big(\Theta^{\chi}_L(x_0)\in I^{\chi}\,;\, E_{L,x_0}\Big)+o\Big(\frac{R_L^d}{L^d}\Big)\,.
 \end{equs}
Now we argue separately for the three measures. For $\chi = \xi$ or  $(\xi,\Phi)$, we write
\begin{equs}
	\P\Big(\Theta^{\chi}_L(x_0)\in I^{\chi}\,;\, E_{L,x_0}\Big) = \P\Big(\Theta^{\chi}_L(x_0)\in I^{\chi}\Big) - \P\Big(\Theta^{\chi}_L(x_0)\in I^{\chi}\,;\, E_{L,x_0}^\complement\Big)
\end{equs}
and the second term is bounded by $\P(\xi_L(x_0) \ge a_L - \theta \,;\, E_{L,x_0}^\complement)$ 
which is negligible compared to $1/L^d$ by \eqref{e:nonE}, so that the claim follows. 
For $\chi = \Xi$, we write
\begin{equs}
\P\Big(\Theta^{\Xi}_L(x_0)\in I^{\Xi}\,;\, E_{L,x_0}\Big)=& \P(\Theta^{\Xi}_L(x_0)\in I^{\Xi}\,;\, |\xi_L(x_0)-a_L|\leq \theta) \label{e:Quasi3}\\
& - \P(\Theta^{\Xi}_L(x_0)\in I^{\Xi}\,;\, |\xi_L(x_0)-a_L|\leq \theta\,;\, E_{L,x_0}^\complement\Big)
\end{equs}
and the second term can be bounded by $\P(\xi_L(x_0) \ge a_L - \theta ; E_{L,x_0}^\complement)$ which, once again, 
is negligible compared to $1/L^d$ by \eqref{e:nonE}. Regarding the first term
 \begin{equs}
  \P(\Theta^{\Xi}_L(x_0)\in I^{\Xi}\,;\, |\xi_L(x_0)-a_L|\leq \theta)=& \P(\Theta_L^{\Xi}(x_0)\in I^{\Xi})\\
  &- \P(\Theta^{\Xi}_L(x_0)\in I^{\Xi}\,;\, |\xi_L(x_0)-a_L|> \theta)
 \end{equs}
 where the second summand is again negligible compared to $1/L^{d}$ in view of~\eqref{e:GausSum2} since, 
 by Assumption \ref{a:TechnicalAssum},
 $\theta=2d+1\gg \max\{a_L\tau_L^2, a_L^{-1}\}$, 
 and thus the proof is complete. 
\end{proof}

\section{Statistics of the maxima of correlated Gaussian fields}\label{sec:LoctoGlobPot}

\noindent The primary goal of this section is to identify the statistics of the maxima of the potential 
$\xi_L$ on $Q_L$ as $L\to\infty$ and thus establish Theorem~\ref{Th:Potential}. Actually, we will not only consider the maxima of $\xi_L$, but also of the fields $\Xi_L$ and $(\xi_L,\Phi_L)$, as these quantities are instrumental in the determination of the top of the spectrum of the Anderson Hamiltonian on $Q_L$, and its relation to the (location of the) maxima of $\xi_L$.

Our analysis will rely on the splitting scheme introduced in Section \ref{sec:Splitting}. 
We will restrict ourselves to the study of the maxima of the fields on $U_L$, 
since, on $Q_L \backslash U_L$, they remain ``small''  with large probability. 
As already mentioned in Section \ref{sec:Splitting}, $U_L$ is a union of mesoscopic boxes 
which lie at a distance at least $\sqrt{R_L}$ from one another. 
A crucial step of our analysis will be to establish suitable {\it decorrelation estimates}, which, roughly speaking, 
allow to regard the restrictions of $\xi_L$ to the mesoscopic boxes $Q_{R_L, z_{j,L}}$, $j\in\{1,\ldots,n_L\}$, as independent. 
As these are technically challenging, we will postpone their statement and proof to 
the end of this section, in Section~\ref{sec:LongRange}. 

\subsection{Convergence of the maxima and fluctuations on $U_L$}\label{sec:OrderStatsXi}

We will deal with the rescaled fields $\Theta_L^{\xi}$, $\Theta_L^{\Xi}$ (see~\eqref{e:ScaledFields}) and $\Theta_L^{(\xi,\Phi)}$, where
\begin{equ}
\glsadd{Thetaxi} \glsadd{ThetaXi}
\Theta_L^{\xi}(x)= a_L\big(\xi_L(x)-a_L\big)\,,\quad \Theta_L^{\Xi}(x)= a_L\big(\Xi_L(x)-a_L^{\Xi}\big)\,,\qquad x\in\Z^d\,,
\end{equ}
for $\Xi_L$ as in~\eqref{e:XiL} and $a_L^\Xi$ in~\eqref{e:aLXiL}, and
\begin{equ}[e:ScaledFieldsJoint]
\glsadd{ThetaxiPhi}
\Theta_L^{(\xi,\Phi)}(x)\eqdef \big(a_L\big(\xi_L(x)-a_L\big),a_L\Phi_L(x)\big)\,,\qquad x\in\Z^d\,.
\end{equ}
for $\Phi_L$ as in~\eqref{e:FL} and $\tau_L$ in~\eqref{e:tauL}. For any $j\in \{1,\ldots,n_L\}$, denote by $w_{j,L}$ the point in $Q_{R_L, z_{j,L}}$ where $\xi_L$ achieves its maximum (recall the definition of $z_{j,L}$ and $n_L$ in~\eqref{e:QLnL}). 
%

For $\chi=\xi,\, \Xi$ or $(\xi,\Phi)$, the random (point) measures of interest are 
\begin{equ}[e:RPM]
	\cP^{\chi}_L \eqdef \sum_{j=1}^{n_L} \delta_{\big(\frac{z_{j,L}}{L} , \Theta^{\chi}_{L}(w_{j,L})\big)}\,,
\end{equ}
and our first goal is to establish their vague convergence.

\begin{proposition}\label{P:PPPxiPrelim}
	As $L\to\infty$, each random measure $\cP^{\chi}_L$ in~\eqref{e:RPM} converges in law to a Poisson random measure $\cP_\infty^{\chi}$ where
	\begin{itemize}
		\item for $\chi=\xi$ or $\Xi$, the convergence holds for the topology of vague convergence on $[-1,1]^d\times (-\infty,\plusinfty]$, and, in both cases, the limiting Poisson measure has intensity $\dd x \otimes e^{-u} \dd u$,
		\item for $\chi=(\xi, \Phi)$, we further assume that $a_L\tau_L\sim \sqrt{b}$ for some $b\geq 0$. Then the convergence holds for the topology of vague convergence on $[-1,1]^d\times (-\infty,\plusinfty]\times[-\infty,\infty]$, and the limiting Poisson measure has intensity $\dd x \otimes e^{-u} \dd u\otimes \frac1{\sqrt{2\pi b}} e^{-\frac{v^2}{2b}}\dd v$.
	\end{itemize}
\end{proposition}

\begin{remark}\label{rem:b2}
The restriction $a_L\tau_L\sim \sqrt{b}$ in case $\chi=(\xi,\Phi)$ can be lifted following 
the same changes discussed in Remark~\ref{rem:b}. 
\end{remark}

To define the notion of convergence stated in the above theorem, set $I^\chi \eqdef [-1,1]^d\times (-\infty,\plusinfty]$ 
for $\chi=\xi$ or $\Xi$, and $I^\chi \eqdef  [-1,1]^d\times (-\infty,\plusinfty]\times[-\infty,\plusinfty]$ 
for $\chi = (\xi, \Phi)$. Then, $\cP^{\chi}_L$ is said to converge in law in the topology of vague convergence 
to $\cP_\infty^\chi$ provided that for any continuous function $g \colon I^\chi \to \R$ with compact support, 
the real-valued random variable $\cP^{\chi}_L(g)$ converges in law to $\cP_\infty^\chi(g)$. 
We refer to~\cite[Theorem 16.16 and Theorem A2.3]{Ka02} for further details on this topology. 
\begin{proof}
Let $\chi=\xi,\Xi$ or $(\xi,\Phi)$. Let $g\colon I^\chi \to \R$ with compact support, 
non-negative and of class $\cC^2$. 
Thanks to~\cite[Theorem 16.16]{Ka02}\footnote{Actually, in the above-mentioned reference $g$ is not assumed to be of class $\cC^2$ but merely continuous. However, a straightforward approximation procedure guarantees that one can restrict to $\cC^2$ functions.}, if we show that for all $\lambda \ge 0$
	\begin{equ}[e:PLPinfty]
		\E[\exp(-\lambda \cP_L^{\chi}(g))] \to  \E[\exp(-\lambda \cP_\infty^{\chi}(g))]\;,\quad \text{as $L\to\infty$,}
	\end{equ}
then we conclude that $\cP^{\chi}_L(g)$ converges in law to $\cP_\infty^{\chi}(g)$. Since $g$ was arbitrary, the very definition 
of the vague convergence of point measures ensures that this is enough to establish the statement. 

To prove~\eqref{e:PLPinfty}, let us fix $g$ and $\lambda$ as above. First of all, observe that
	\begin{equ}
		\E[\exp(-\lambda \cP^{\chi}_L(g))] = \E\Big[\prod_{j=1}^{n_L} \exp(-\lambda g(\tfrac{z_{j,L}}{L},\Theta^{\chi}_{L}(w_{j,L}))\Big]\;.
	\end{equ}
	We claim that 
	\begin{equ}[e:Eprod]
		\E\Big[\prod_{j=1}^{n_L} \exp(-\lambda g(\tfrac{z_{j,L}}{L},\Theta^{\chi}_{L}(w_{j,L}))\Big] = \prod_{j=1}^{n_L} \E\Big[\exp(-\lambda g(\tfrac{z_{j,L}}{L},\Theta^{\chi}_{L}(w_{j,L}))\Big] + o(1)\;,
	\end{equ}
	where $o(1)$ is a quantity that vanishes as $L\to\infty$. Given~\eqref{e:Eprod}, we have 
	\begin{equs}
		\E[\exp(-\lambda \cP^{\chi}_L(g))]=\exp\bigg(\sum_{j=1}^{n_L} \ln \Big(1 - \E \big[1-\exp \big(-\lambda g(\tfrac{z_{j,L}}{L},\Theta^{\chi}_{L}(w_{j,L})) \big) \big]\Big)\bigg)+o(1)\;.
	\end{equs}
	By Theorem \ref{Th:TailJointLaw}, uniformly over all $j$ the expectation on the r.h.s.~is of order $R_L^d/L^d$. Since $n_L$ is of order $(L/R_L)^d$, we deduce that the last term equals
	\begin{equs}
		\ &\exp\Big(-\sum_{j=1}^{n_L} \E[1-\exp(-\lambda g(\tfrac{z_{j,L}}{L},\Theta^{\chi}_{L}(w_{j,L}))])\Big) (1+o(1))\\
		&=\exp\Big(-\sum_{j=1}^{n_L} \int \big(1-\exp(-\lambda g(\tfrac{z_{j,L}}{L},q))\big) \P(\Theta^{\chi}_{L}(w_{j,L}) \in \dd q)\Big)(1+o(1))\;.
	\end{equs}
	Invoking Theorem \ref{Th:TailJointLaw} once again, the latter converges to
	\begin{equ}
		\exp\Big(- \int \int \big(1-\exp(-\lambda g(x,u))\big) \dd x \otimes e^{-u} \dd u \Big)\;,
	\end{equ}
	if $\chi = \xi$ or $\Xi$, and to
	\begin{equ}
	\exp\Big(- \int \int \big(1-\exp(-\lambda g(x,u,v))\big) \dd x \otimes e^{-u} \dd u \otimes \frac1{\sqrt{2\pi b}} e^{-\frac{v^2}{2b}} \dd v \Big)\;,
	\end{equ}	
	if $\chi = (\xi, \Phi)$. In all cases, this equals $\E[\exp(-\lambda \cP_\infty^{\chi}(g))]$, 
	and thus~\eqref{e:PLPinfty} follows.\\
	
	We are left with proving~\eqref{e:Eprod}. What we will show is that~\eqref{e:Eprod} holds 
	provided the following decorrelation estimate does 
	\begin{equs}[e:Eprodbis]
		\E&\Big[\prod_{j=1}^{n_L}\prod_{x_0 \in Q_{R_L,z_{j,L}}} \exp(-\lambda g(\tfrac{z_{j,L}}{L},\Theta^{\chi}_{L}(x_0)))\Big]\\
		&= \prod_{j=1}^{n_L}\E\Big[\prod_{x_0 \in Q_{R_L,z_{j,L}}} \exp(-\lambda g(\tfrac{z_{j,L}}{L},\Theta^{\chi}_{L}(x_0)))\Big] + o(1)\;, 
	\end{equs}
	which in turn will be proved in Proposition \ref{L:Leadbetter}.
	
	To see the relation between~\eqref{e:Eprod} and~\eqref{e:Eprodbis}, 
	let us begin by considering $\chi =(\xi, \Phi)$. 
	Let $c>0$ be such that $g$ vanishes on $[-1,1]^d \times (-\infty,-c]\times[-\infty,\infty]$. 
	On the complement of the event $\cA^\xi_{z_{j,L}, L}(c)$ in~\eqref{e:AjLXi}, 
	there is at most one point $x_0 \in Q_{R_L,z_{j,L}}$ where $\xi_L(x_0) \ge a_L - \frac{c}{a_L}$ 
	(thus $\Theta^{(\xi,\Phi)}(x_0)\in[-c,\infty]\times[-\infty,\infty]$), 
	and necessarily, if such a $x_0$ exists then $x_0 = w_{j,L}$. 
	Since $g$ is a non-negative function supported on $[-1,1]^d \times [-c,\plusinfty] \times [-\infty,\plusinfty]$, we deduce that
	\begin{equ}
		\Big|\exp(-\lambda g(\tfrac{z_{j,L}}{L},\Theta^{(\xi, \Phi)}_{L}(w_{j,L})) - \prod_{x_0 \in Q_{R_L,z_{j,L}}} \exp(-\lambda g(\tfrac{z_{j,L}}{L},\Theta^{(\xi, \Phi)}_{L}(x_0))) \Big| \le \1_{\cA^\xi_{L,z_{j,L}}(c)}\,,
	\end{equ}
	and
	\begin{equs}
	\Big|\E\Big[&\exp(-\lambda g(\tfrac{z_{j,L}}{L},\Theta^{(\xi, \Phi)}_{L}(w_{j,L}))\Big]\\
	&- \E\Big[\prod_{x_0 \in Q_{R_L,z_{j,L}}} \exp(-\lambda g(\tfrac{z_{j,L}}{L},\Theta^{(\xi, \Phi)}_{L}(x_0)))\Big] \Big| \le \P\big(\cA^\xi_{L,z_{j,L}}(c)\big)\,,
	\end{equs}
	Using the identity $\prod_{j=1}^{n_L} a_j - \prod_{j=1}^{n_L} b_j =\sum_{k=1}^{n_L} a_1 \cdots a_{k-1} (a_k-b_k) b_{k+1}\cdots b_{n_L}$, and the fact that each factor is bounded by $1$, we get
	\begin{equs}
		\Big| \prod_{j=1}^{n_L}\E\Big[ &\exp(-\lambda g(\tfrac{z_{j,L}}{L},\Theta^{(\xi, \Phi)}_{L}(w_{j,L}))\Big]
		\\
		&- \prod_{j=1}^{n_L}\E\Big[\prod_{x_0 \in Q_{R_L,z_{j,L}}} \exp(-\lambda g(\tfrac{z_{j,L}}{L},\Theta^{(\xi, \Phi)}_{L}(x_0)))\Big] \Big| 
	\le \sum_{j=1}^{n_L} \P(\cA^\xi_{L,z_{j,L}}(c))\;,
	\end{equs}
	and
	\begin{equs}
		\Big| \E\Big[\prod_{j=1}^{n_L} &\exp(-\lambda g(\tfrac{z_{j,L}}{L},\Theta^{(\xi, \Phi)}_{L}(w_{j,L})))\Big]\\
		&- \E\Big[\prod_{j=1}^{n_L}\prod_{x_0 \in Q_{R_L,z_{j,L}}} \exp(-\lambda g(\tfrac{z_{j,L}}{L},\Theta^{(\xi, \Phi)}_{L}(x_0)))\Big] \Big| \le \sum_{j=1}^{n_L} \P(\cA^\xi_{L,z_{j,L}}(c))\;.
	\end{equs}
	By~\eqref{e:UniqueMax}, the sums at the r.h.s.~of the last two inequalities go to $0$ as $L\to\infty$, 
	and thus, the triangle inequality immediately implies that the proof of~\eqref{e:Eprod} can be reduced to 
	that of~\eqref{e:Eprodbis}. 
	
	For $\chi = \xi$ or $\Xi$, the argument is virtually identical, the only difference in the case $\chi=\Xi$ is that 
	$\cA^\xi_{L,z_{j,L}}(c)$ has to be replaced by
	$$ \cA^\Xi_{L,z_{j,L}}(c) \cup \{\max_{x\in Q_{R_L,z_{j,L}}}\Xi_L\geq a^{\Xi}_L-\frac{c}{a_L}\,;\, w_{L,z_{j,L}}\neq w_{L,z_{j,L}}^{\Xi}\}\;.$$
	To bound the sum over $j$ of the probability of their union, it suffices to use~\eqref{e:UniqueMax} and~\eqref{e:SameMax}, which once again implies that~\eqref{e:Eprod} holds provided~\eqref{e:Eprodbis} does. 
	\end{proof}

\subsection{Proof of Theorem \ref{Th:Potential}}\label{sec:ThPot}
Thanks to the results in the previous section and in particular Proposition~\ref{P:PPPxiPrelim} 
for $\chi=\xi$, we can complete the proof of the first theorem stated in the introduction. 
 
\begin{proof}[Proof of Theorem \ref{Th:Potential}]
Define the random measure
\begin{equ}
\cM_L \eqdef \sum_{k=1}^{\#Q_L} \delta_{\big(\frac{y_{k,L}}{L},\Theta_L^{\xi}(y_{k,L})\big)}\;,
\end{equ}
for $\Theta_L^{\xi}(\cdot)=a_L(\xi_L(\cdot)-a_L)$ as in~\eqref{e:ScaledFields} and $y_{k,L}$ the point on $Q_L$ at 
which $\xi_L$ reaches its $k$-th largest maximum. 
Theorem~\ref{Th:Potential} states that $\cM_L$ converges in law towards a Poisson random measure 
$\cM_\infty$ of intensity $\dd x \otimes e^{-u}\dd u$, 
for the topology of vague convergence of Radon measures on $[-1,1]^d \times (-\infty,\plusinfty]$. 
By~\cite[Theorem 16.16]{Ka02}, it suffices to prove that 
for any continuous function $g\colon [-1,1]^d \times (-\infty,\plusinfty] \to \R_+$ with compact support, 
$\cM_L(g) \to \cM_\infty(g)$ in law as $L\to\infty$. 
Fix such a function $g$ and let $c>0$ be such that $g$ vanishes outside $[-1,1]^d\times (-c,\plusinfty]$. 
Set
\begin{equ}
\cB_{L}(c) \eqdef \big\{ \exists x \in Q_L \backslash U_L\colon \Theta_L^{\xi}(x) \ge -c \big\}\;,
\end{equ}
that is, $\cB_L(c)$ is the event on which $\xi_L$ is ``large'' in $Q_L \backslash U_L$. 
By definition of $U_L$, $|Q_L \backslash U_L| \lesssim n_L \sqrt{R_L} R_L^{d-1}$ and thus by~\eqref{e:preBound} and~\eqref{e:QLnL}, we get 
\begin{equ}[e:ProbBLC]
		\P(\cB_{L}(c)) \le |Q_L \backslash U_L| \,\P(\Theta_L^{\xi}(0) \ge -c) \lesssim n_L \sqrt{R_L} R_L^{d-1} \frac{e^{c}}{L^d} \lesssim \frac{1}{\sqrt{R_L}}\;,
	\end{equ}
and the r.h.s. vanishes as $L\to\infty$.
	
Recall the definition of $z_{j,L}$ and $\cA^{\xi}_{L,\cdot}(c)$ in~\eqref{e:QLnL} and~\eqref{e:AjLXi}, respectively, 
and set $\cA^\xi_L(c) = \cup_{j=1}^{n_L}\cA^{\xi}_{L,z_{j,L}}(c)$.
On the event $\cA^\xi_L(c)^\complement$, 
for every $j\in\{1,\ldots,n_L\}$, 
the box $Q_{R_L, z_{j,L}}$ contains at most one point where $\xi_L$ lies above $a_L - c/a_L$, 
and if such a point exists it must be $w_{j,L}$.

This implies that on the event $\cA^\xi_L(c)^\complement \cap \cB_{L}(c)^\complement$
the set of points $\{y_{k,L}\colon 1\le k \le \# Q_L\,,\,\xi_L(y_{k,L}) \ge a_L - c / a_L \}$ 
coincides with the set of points $\{w_{j,L}\colon 1\le j \le n_L, \xi_L(w_{j,L}) \ge a_L - c / a_L\}$. 
Therefore, using the notation of~\eqref{e:RPM}, on $\cA^\xi_{L}(c)^\complement \cap \cB_{L}(c)^\complement$, we get 
\begin{equs}
		\cP_L^{\xi}(g) - \cM_L(g) = \sum_{j=1}^{n_L} \Big( g\big(\tfrac{z_{j,L}}{L} , \Theta^{\xi}_{L}(w_{j,L})\big) - g\big(\tfrac{w_{j,L}}{L} , \Theta^{\xi}_{L}(w_{j,L})\big)\Big)\;. 
\end{equs}
Now, $g$ is uniformly continuous in its first coordinate, so that, since $|z_{j,L} - w_{j,L}| \le \sqrt{d} \,R_L$, 
we deduce
\begin{equ}
 |\cP_L^{\xi}(g) - \cM_L(g)| \lesssim \omega(\tfrac{\sqrt{d} \,R_L}{L}) \cP_L^{\xi}([-1,1]^d \times (-c,\infty])\;.
 \end{equ}
where $\omega(\cdot)$ is the modulus of continuity of $g$ in its first coordinate. 
 The prefactor $\omega(\sqrt{d} \,R_L/L)$ vanishes as $L\to\infty$, while $\cP_L^{\xi}([-1,1]^d \times (-c,\infty])$ 
converges in law to a finite limit thanks to Proposition \ref{P:PPPxiPrelim}. 
As the probability of $\cA^\xi_{L}(c)^\complement \cap \cB_{L}(c)^\complement$ 
converges to $1$ by~\eqref{e:UniqueMax} and~\eqref{e:ProbBLC}, we deduce that 
$|\cP_L^{\xi}(g) - \cM_L(g)|$ goes to $0$ in (probability and thus in) law, and therefore, invoking once again 
Proposition \ref{P:PPPxiPrelim}, the statement follows at once.
\end{proof}

\subsection{Decorrelation estimates}\label{sec:LongRange}

In order to deal with the long-range correlations of our field, we now prove the decorrelation estimates 
which were exploited in the proof of Proposition \ref{P:PPPxiPrelim}.

\begin{proposition}\label{L:Leadbetter}
For $\chi = \xi, \,\Xi$, or $(\xi,\Phi)$, let $g\colon I^\chi \to \R$ (where $I^\chi$ is defined after 
Proposition~\ref{P:PPPxiPrelim}) be a compactly supported non-negative function of class $\cC^2$. 
If $\chi=(\xi,\Phi)$, further assume that $a_L\tau_L\sim \sqrt{b}$ for some $b\geq 0$. 
Then, as $L\to\infty$ 
	\begin{equs}[e:Eprodbis2]
		\E\Big[\prod_{j=1}^{n_L}&\prod_{x_0 \in Q_{R_L,z_{j,L}}} \exp\Big(-\lambda g(\tfrac{z_{j,L}}{L},\Theta^{\chi}_{L}(x_0))\Big)\Big]\\
		&- \prod_{j=1}^{n_L}\E\Big[\prod_{x_0 \in Q_{R_L,z_{j,L}}} \exp\Big(-\lambda g(\tfrac{z_{j,L}}{L},\Theta^{\chi}_{L}(x_0))\Big)\Big] = o(1)\;.
	\end{equs}
\end{proposition}


The proof is inspired by~\cite[Theorem 4.2.1]{Leadbetter}. In the reference, the estimate is established for a function $g$ which only depends on the second coordinate and which is the indicator of a semi-infinite interval. It turns out that dealing with regular functions $g$ makes the proof somewhat easier. What we establish is a ``long-range'' decorrelation estimate: in the second term on the l.h.s.~of \eqref{e:Eprodbis2} the product over all $x_0 \in Q_{R_L,z_{j,L}}$ remains inside the expectation (and so the decorrelation is proved for disjoint boxes) as, in our setting, at small distances the r.v.'s~at stake may have a complicated correlation structure that we do not try to disentangle.

\begin{proof}
We will present the proof in detail for $\chi = (\xi,\Phi)$, and, since that for $\chi = \xi,\, \Xi$ is similar (and actually 
simpler), we will limit ourselves to outline the main (minor) differences at the very end. 

Let us introduce some notation.  
We rename the Gaussian vector of interest as 
\begin{equ}
(\eta^{1,0}_L(x_0))_{x_0\in U_L}\eqdef ((\eta^{1,0}_{1,L}(x_0), \eta^{1,0}_{2,L}(x_0))_{x_0\in U_L}\eqdef (\xi_L(x_0),a_L\Phi_L(x_0))_{x_0\in U_L}
\end{equ}
(the reason for the double superscript 
will be clarified soon). Let  $\Sigma^{1,0}$ be its covariance matrix 
and index its entries by $(x_0,i)$ with $x_0 \in U_L$ and $i\in\{1,2\}$. 
For instance, $\Sigma^{1,0}_{(x_0,2), (y_0,1)}$ is the covariance of 
$a_L\Phi_L(x_0)$ and $\xi_L(y_0)$. 
\medskip

As mentioned at the beginning of the section, the statement boils down 
to show that the error made by replacing  
$\eta^{1,0}_L$ with a Gaussian vector $\eta^{0,0}_L$ 
such that, for every $j=1,\dots, n_L$, $(\eta^{1,0}_L(x))_{x\in Q_{R_L,z_{j,L}}}\eqlaw(\eta^{0,0}_L(x))_{x\in Q_{R_L,z_{j,L}}}$, and $(\eta^{0,0}_L(x))_{x\in Q_{R_L,z_{j_1,L}}}$ is independent of 
$(\eta^{0,0}_L(x))_{x\in Q_{R_L,z_{j_2,L}}}$ for every $j_1\neq j_2$, is negligible as $L\to\infty$. 
Let $\Sigma^{0,0}$ be the covariance matrix of $\eta^{0,0}_L$ and notice that 
it is given by 
\begin{equ}[e:Sigma0] 
\Sigma^{0,0}_{(x_0,i),(y_0,i')} \eqdef \begin{cases} \Sigma^{1,0}_{(x_0,i),(y_0,i')} &\mbox{ if there exists $j$ s.t. $x_0, y_0\in Q_{R_L, z_{j,L}}$, }\\
		0 &\mbox{ else. }
\end{cases}
\end{equ}
For the reader's convenience, let us split the (quite involved) proof into four steps: 
{\it degeneracy}, {\it interpolation}, {\it density estimates} and {\it decay}, 
whose names will be justified along the way.
\medskip

\noindent{\it Step 1: Degeneracy.} The problem with the Gaussian vectors 
$\eta^{1,0}_L$ and $\eta^{0,0}_L$ is that they may be degenerate and thus might not 
admit a density with respect to the Lebesgue measure. To overcome the issue, 
we will slightly perturb them: let $\eps>0$ and $(\gamma_1(x_0), \gamma_2(x_0))_{x_0 \in U_L}$ be an independent 
centred Gaussian vector of i.i.d.~$\cN(0,\eps)$ r.v.'s and, for $\alpha=0,1$,  
define $\eta^{\alpha,\eps}_L$ according to 
\begin{equ}
(\eta^{\alpha,\eps}_{1,L}(x_0),\eta^{\alpha,\eps}_{2,L}(x_0))\eqdef  (\eta^{\alpha,0}_{1,L}(x_0)+\gamma_1(x_0),\eta^{\alpha,0}_{2,L}(x_0)+\gamma_2(x_0))\,,\qquad x_0\in U_L\,.
\end{equ}
Notice that $\eta^{\alpha,\eps}_L$ is again a Gaussian vector but, this time, has 
a full rank covariance matrix, which we denote by $\Sigma^{\alpha,\eps}$. 

We now claim that provided we show 
\begin{equ}[e:claimLongRange]
\limsup_{L\to\infty}\limsup_{\eps\to 0} |I^\eps_{L}|=0
\end{equ}
where $I^\eps_L$ is defined according to 
\begin{equs}[e:Eprodter]
	I_{L}^\eps \eqdef \E&\Big[\prod_{j=1}^{n_L}\prod_{x_0 \in Q_{R_L,z_{j,L}}} \exp\Big(-\lambda g(\tfrac{z_{j,L}}{L},a_L(\eta^{1,\eps}_{1,L}(x_0) - a_L), \eta^{1,\eps}_{2,L}(x_0))\Big)\Big]\\
	&- \prod_{j=1}^{n_L}\E\Big[\prod_{x_0 \in Q_{R_L,z_{j,L}}} \exp\Big(-\lambda g(\tfrac{z_{j,L}}{L},a_L(\eta^{0,\eps}_{1,L}(x_0) - a_L), \eta^{0,\eps}_{2,L}(x_0))\Big)\Big]\,,
\end{equs}
then~\eqref{e:Eprodbis2} follows. 
Indeed, since $I^0_L$ coincides with the l.h.s. of~\eqref{e:Eprodbis2} and 
the vectors $\eta^{1,\eps}_L$ and $\eta^{0,\eps}_L$ converge in law as 
$\eps\downarrow 0$ respectively to $\eta^{1,0}_L$ and $\eta^{0,0}_L$, 
for any fixed $L$ we have $\lim_{\eps \downarrow 0} I_L^\eps=I^0_L$, 
so that~\eqref{e:claimLongRange} implies the statement. 
We are left to prove~\eqref{e:claimLongRange} to which the next steps are devoted. 
\medskip

\noindent{\it Step 2: Interpolation.} We now introduce an interpolation 
between the covariance matrices $\Sigma^{1,\eps}$ and $\Sigma^{0,\eps}$, i.e. 
for $h\in[0,1]$ we define 
\begin{equ}[e:Interpol]
\Sigma^{h,\eps} \eqdef h \Sigma^{1,\eps} + (1-h) \Sigma^{0,\eps}\;.
\end{equ}
This is still a positive definite matrix, and therefore it is 
the covariance matrix of a non-degenerate Gaussian vector $\eta^{h,\eps}_L$. 
Let $f_h(s)$ for $s=(s_{(x_0,i)})_{(x_0,i)\in U_L\times\{1,2\}}\in \R^{2|U_L|}$ be the associated Gaussian density 
at $s$ and set 
\begin{equ}[e:Ast]
A(s) \eqdef \prod_{j=1}^{n_L} \prod_{x_0 \in Q_{R_L,z_{j,L}}} \exp(-\lambda g(\tfrac{z_{j,L}}{L},a_L(s_{(x_0,1)} - a_L), s_{(x_0,2)}))\;,
\end{equ}
and, for any $h\in[0,1]$, 
\begin{equs}
F(h)&\eqdef \E\Big[\prod_{j=1}^{n_L}\prod_{x_0 \in Q_{R_L,z_{j,L}}} \exp\Big(-\lambda g(\tfrac{z_{j,L}}{L},a_L(\eta^{h,\eps}_{1,L}(x_0) - a_L), \eta^{h,\eps}_{2,L}(x_0))\Big)\Big]\\
&=\int A(s) f_h(s)\dd s\,.
\end{equs}
Notice in particular that this means 
\begin{equ}[e:Newclaim]
I^\eps_L=F(1)-F(0)=\int_0^1 F'(h)\dd h =\int_0^1\int A(s) \partial_h f_h(s) \dd s \,\dd h\,.
\end{equ}
Hence, to obtain~\eqref{e:claimLongRange} we need to bound the r.h.s.  
by a quantity independent of $\eps$ and that vanishes 
as $L\to\infty$.
\medskip

\noindent{\it Step 3: Density Estimates.} Let $\le$ be an arbitrary total order on $\R^{|U_L|}$, and with a slight abuse of notation let us extend it into a total order on $\R^{|U_L|} \times \{1,2\}$ by setting
$$ (x_0,i) \le (y_0,i') \Leftrightarrow (x_0 < y_0) \mbox{ or } (x_0=y_0, i \le i')\;.$$
The dependence of $f_h$ on $h$ only goes through $\Sigma^{h, \eps}$. 
Since this matrix is symmetric, we will only consider its entries ``above the diagonal'', that is $(\Sigma_{(x_0,i),(y_0,i')}^{h,\eps}: (x_0,i) \le (y_0,i'))$.
By~\eqref{e:Interpol} and the definition of $\Sigma^{0,\eps}$ in~\eqref{e:Sigma0}, the derivative of 
$\Sigma^{h,\eps}$ in $h$ reads
\begin{equs}
\partial_h \Sigma_{(x_0,i),(y_0,i')}^{h,\eps} &= \Sigma_{(x_0,i),(y_0,i')}^{1,\eps} - \Sigma_{(x_0,i),(y_0,i')}^{0,\eps}\\
&=\begin{cases} 0 &\mbox{ if there exists $j$ s.t. $x_0, y_0\in Q_{R_L, z_{j,L}}$, }\\
		\Sigma^{1,\eps}_{(x_0,i),(y_0,i')} &\mbox{ else. }
\end{cases}
\end{equs}
On the other hand, using the identities
$$ \frac{\partial \det \Sigma}{\partial \Sigma_{(x_0,i),(y_0,i')}} = 2(\det \Sigma)\, \Sigma^{-1}_{(x_0,i),(y_0,i')}\;,\quad \frac{\partial \Sigma^{-1}}{\partial \Sigma_{(x_0,i),(y_0,i')}} = -\Sigma^{-1} \frac{\partial \Sigma}{\partial \Sigma_{(x_0,i),(y_0,i')}} \Sigma^{-1}\;,$$
where $\Sigma^{-1}$ denotes the inverse of $\Sigma$ and $\Sigma^{-1}_{(x_0,i),(y_0,i')}$ its $((x_0,i),(y_0,i'))$-entry, 
a straightforward computation yields
\begin{equ}
\frac{\partial f_h}{\partial \Sigma_{(x_0,i),(y_0,i')}^{h,\eps}} = \frac{\partial^2 f_h}{\partial s_{(x_0,i)} \partial s_{(y_0,i')}}\;. 
\end{equ}
Therefore,
\begin{equs}
F'(h) &= \int A(s) \partial_h f_h(s) \dd s=\sum_{(x_0,i)\leq (y_0, i')}\int A(s)\partial_h \Sigma_{(x_0,i),(y_0,i')}^{h,\eps} \frac{\partial f_h(s)}{\partial \Sigma_{(x_0,i),(y_0,i')}^{h,\eps}}\dd s\\
&=\sum_{(x_0,i) \le (y_0,i')} \Sigma_{(x_0,i) , (y_0,i')}^{1,\eps}  \int A(s) \frac{\partial^2 f_h(s)}{\partial s_{(x_0,i)} \partial s_{(y_0,i')}} \dd s \\
&=\sum_{(x_0,i) \le (y_0,i')} \Sigma_{(x_0,i) , (y_0,i')}^{1,\eps} \int \frac{\partial^2 A(s)}{\partial s_{(x_0,i)} \partial s_{(y_0,i')}} f_h(s) \dd s \;,
\end{equs}
where the sum is only over $x_0$ and $y_0$ that do not fall within the same box. 

We now need to estimate both $\Sigma^{1,\eps}$ and $A$. 
For the former, since $\gamma_1$, $\gamma_2$ and $(\xi_L, \Phi_L)$ 
are independent of each others, it is immediate to see that for any $(x_0,i), (y_0,i')\in U_L\times \{1,2\}$, 
$\Sigma_{(x_0,i),(y_0,i')}^{1,\eps} =\cov(\xi_L(x_0), \xi_L(y_0))$ if $i=i'=1$, $a_L \cov(\xi_L(x_0), \Phi_L(y_0))$ 
if $i=1, i'=2$ and $a_L^2 \cov(\Phi_L(x_0), \Phi_L(y_0))$ if $i=i'=2$. 
Therefore, provided $x_0,y_0$ do not belong to the same mesoscopic box (so that in particular $|x_0-y_0|> \sqrt{R_L}$),~\eqref{e:CovxiPhi} implies 
\begin{equ}
|\Sigma_{(x_0,i),(y_0,i')}^{1,\eps}|
\lesssim a_L^{i+i'-2}\sup_{|x|\geq |x_0-y_0|-\sqrt{d}\,r_L} v_L(x)\,.
\end{equ}
For $A$, by assumption $g\in\cC^2$ is compactly supported, so let $c>0$ be such that 
the support of $g$ is contained in $[-1,1]^d \times [-c,\plusinfty]\times[-\infty,\plusinfty]$.  
Then, evaluating the second derivative of $A$ in~\eqref{e:Ast} gives  
\begin{equ}
\Big| \frac{\partial^2 A(s)}{\partial s_{(x_0,i)} \partial s_{(y_0,i')}} \Big| \le C a_L^{4-i-i'} \1_{\{s_{(x_0,1)}\wedge s_{(y_0,1)}\geq a_L-\frac{c}{a_L}\}} \;,
\end{equ}
for some constant $C>0$ independent of $L$ (the variables in the indicator are both with $i=1$!). 
Putting everything together, we obtain 
\begin{equs}
|F'(h)|&\lesssim \sum_{(x_0,i) \le (y_0,i')}\sup_{|x|\geq |x_0-y_0|-\sqrt{d}\,r_L} v_L(x)\; a_L^2 \int  \1_{\{s_{(x_0,1)}\wedge s_{(y_0,1)}\geq a_L-\frac{c}{a_L}\}} f_h(s)\dd s\\
&\lesssim \sum_{x_0\leq y_0}\sup_{|x|\geq |x_0-y_0|-\sqrt{d}\,r_L} v_L(x)\; a_L^2 \int  \1_{\{s_{x_0}\wedge s_{y_0}\geq a_L-\frac{c}{a_L}\}} \tilde f_h(s_{x_0}, s_{y_0})\dd s_{x_0} \dd s_{y_0}
\end{equs}
where in the last step we used that the summand only depends on $x_0,y_0$ and not on $i,i'$ 
so that, with a slight abuse, we suppressed the latter from the notation, and we denoted 
by $\tilde f_h$ the marginal of $f_h$ restricted to the two coordinates $s_{x_0}=s_{x_0,1},s_{y_0}=s_{y_0,1}$. 
In other words, $\tilde f_h$ is the density of the Gaussian pair 
$(\xi_L(x_0)+ \gamma^1(x_0),\xi_L(y_0)+ \gamma^2(y_0))$ and thus 
is given by 
\begin{equs}
\tilde f_h(s_{x_0}, s_{y_0})&= \frac{\exp\Big(-\frac{(1+\eps) s_{x_0}^2 - 2 v_L(x_0-y_0) s_{x_0} s_{y_0} + (1+\eps) s_{y_0}^2}{2((1+\eps)^2-v_L(x_0-y_0)^2)} \Big)}{2\pi \big((1+\eps)^2- v_L(x_0-y_0)^2\big)^{1/2}}\\
&\leq \frac{\exp\Big(-\frac{s_{x_0}^2 +s_{y_0}^2}{2(1+\eps+v_L(x_0-y_0))} \Big)}{2\pi \big((1+\eps)^2- v_L(x_0-y_0)^2\big)^{1/2}}
\end{equs}
as follows by applying $a^2+b^2\geq 2ab$. Therefore, using the above and 
the basic Gaussian estimate~\eqref{e:TailGauss}, we deduce 
\begin{equs}
a_L^2& \int  \1_{\{s_{x_0}\wedge s_{y_0}\geq a_L-\frac{c}{a_L}\}} \tilde f_h(s_{x_0}, s_{y_0})\dd s_{x_0} \dd s_{y_0}\\
 &\leq  \frac{a_L^2}{\big((1+\eps)^2- v_L(x_0-y_0)^2\big)^{1/2}}
 \Big(\int_{a_L-\frac{c}{a_L}}^\infty \frac{\exp\Big(-\frac{t^2}{2(1+\eps+v_L(x_0-y_0))} \Big)}{\sqrt{2\pi}}\dd t\Big)^2\\
 &\lesssim \frac{1}{\big((1+\eps)^2- v_L(x_0-y_0)^2\big)^{1/2}}\exp\Big(-\frac{a_L^2}{1+\eps+v_L(x_0-y_0)}\Big)\,.
\end{equs}
Plugging all the previous estimates into~\eqref{e:Newclaim}, since the r.h.s.~of 
the bounds obtained so far are independent of $h$, 
we finally obtain
\begin{equs}[e:kappaL]
\limsup_{\eps\to 0} |I^\eps_L|&\lesssim \sum_{x_0\leq y_0}\frac{\sup_{|x|\geq |x_0-y_0|-\sqrt{d}\,r_L} v_L(x)}{\big(1- v_L(x_0-y_0)^2\big)^{1/2}}\exp\Big(-\frac{a_L^2}{1+v_L(x_0-y_0)}\Big)\\
&\lesssim L^d\sum_{z\in Q_L\setminus Q_{2\exp(\sqrt{\ln L})}}\frac{\sup_{|x|\geq |z|-\sqrt{d}\,r_L} v_L(x)}{\big(1- v_L(z)^2\big)^{1/2}}\exp\Big(-\frac{a_L^2}{1+v_L(z)}\Big)
\end{equs}
where the last step holds as $x_0$ and $y_0$ belong to distinct boxes, so that 
$|x_0 - y_0| \ge \sqrt{R_L} \ge 2\exp(\sqrt{\ln L})$. 
The last step consists of proving that the sum vanishes, which in turn is a consequence of the decay of $v_L$. 
\medskip

\noindent{\it Step 4: Decay of Correlations. } We split the sum at the r.h.s. of~\eqref{e:kappaL} into two parts. 
First, we consider the sum over $z\in Q_{L^{1/4}}\setminus Q_{2\exp(\sqrt{\ln L})}$, 
on which for all $L$ large enough, \eqref{e:LongRange}
ensures that $v_L(z) \le 1/2$ and $\sup_{|x| \ge |z| -\sqrt{d}\,r_L} v_L(x) \le 1/2$. Thus, 
\begin{align*}
L^d&\sum_{z\in Q_{L^{1/4}}\setminus Q_{2\exp(\sqrt{\ln L})}}\frac{\sup_{|x|\geq |z|-\sqrt{d}\,r_L} v_L(x)}{\big(1- v_L(z)^2\big)^{1/2}}\exp\Big(-\frac{a_L^2}{1+v_L(z)}\Big)\\
&\lesssim L^d \; |Q_{L^{1/4}}\setminus Q_{2\exp(\sqrt{\ln L})}| \;\exp\Big(-\frac{a_L^2}{1+1/2}\Big)\lesssim L^{d(1+\frac14)} \exp\Big(-\tfrac23a_L^2\Big)\\
&=L^{\frac54 d} \Big(e^{-\frac{a_L^2}{2}}\Big)^{\frac43}\lesssim L^{\frac54 d} \Big(\frac{a_L}{L^d}\Big)^{\frac43} = L^{-\frac{d}{12}} a_L^{\frac43}\;,
	\end{align*}
	where we further used~\eqref{e:Asympa_L}, and the r.h.s. vanishes as $L\to\infty$. 
	
On the other hand, to control the sum over $x\in Q_L\setminus Q_{L^{1/4}}$, 
set $\eps_L \eqdef \sup_{|x|\ge L^{1/4}-\sqrt{d}\,r_L} v_L(x)$ and write
\begin{align*}
L^d&\sum_{z\in Q_L\setminus Q_{L^{1/4}}}\frac{\sup_{|x|\geq |z|-\sqrt{d}\,r_L} v_L(x)}{\big(1- v_L(z)^2\big)^{1/2}}\exp\Big(-\frac{a_L^2}{1+v_L(z)}\Big)\\	
&\lesssim  L^{2d} \eps_L e^{-\frac{a_L^2}{1+\eps_L}}\lesssim  L^{2d} \eps_L e^{-a_L^2} e^{a_L^2\frac{\eps_L}{1+\eps_L}}\lesssim a_L^2 \eps_L  e^{a_L^2 \eps_L}\;,
\end{align*}
which goes to $0$ since $a_L^2 \eps_L$ goes to $0$ by \eqref{e:LongRange}.
\newline

The proof of the statement is thus complete for $\chi=(\xi,\Phi)$. For $\chi=\xi$ or $\Xi$, one can follow the 
exact same steps. Note that in the latter case, one has to replace the properties of $v_L$ with those 
of $v_L^\Xi$ in Lemma~\ref{l:XiL}, in particular, that $v_L^\Xi(0)=1+\tau_L^2$ and~\eqref{e:LongRangeXi}. 
\end{proof}

\section{The mesoscopic eigenproblem}\label{Sec:Local}

\noindent While in the previous section we completely characterised the asymptotics of the maxima 
of the potential, we now turn to the analysis of the Anderson Hamiltonian associated to it 
and, as in Section~\ref{Sec:Field}, we begin by studying the eigenproblem {\it locally} 
on a mesoscopic box of side-length $R_L$, with $R_L$ as in~\eqref{e:RL}. {Using the stationarity in law of our potential under spatial shifts, this analysis will then be applied to any box $Q_{R_L,z_{j,L}}$ in the next section.}

More specifically, we aim at understanding the behaviour of the principal eigenvalue $\lambda_1(Q_{R_L}, \xi_L)$ 
and eigenfunction $\phi_{R_L}$ of $\cH_{Q_{R_L}, \xi_L}$. 
As mentioned in the introduction, their behaviour is intimately related 
to that of the deterministic eigenproblem associated to the Hamiltonian
\begin{equ}
\bar\cH_{L} \eqdef \cH_{Q_{r_L}, -\cS_L} = \Delta - \cS_L\,,\qquad\text{ on $Q_{r_L}$,}
\end{equ}  
for $r_L$ as in~\eqref{e:rL} and $\cS_L$ the shape defined in~\eqref{e:Shape}, 
whose principal eigenfunction and eigenvalue 
are denoted by $\bar\phi_{L}$ and $\bar\lambda_L$. 

To state the main theorem of this section 
we need to introduce a few quantities. 
Recall that we denote by $w_L$ the point in $Q_{R_L}$ where $\xi_L$ attains its maximum, 
by $\Xi_L$ and $\Phi_L$ the fields in~\eqref{e:XiL} and~\eqref{e:FL} respectively,  
by $\tau_L^2$ the variance of $\Phi_L(y)$ for any given $y\in\Z^d$ and $a_L^\Xi = a_L\sqrt{1+\tau_L^2}$.
At last, the event(s) whose probability we want to determine is  
\begin{equ}[e:LambdaL]
\glsadd{LambdaLs}
\Lambda_L(s) \eqdef \Big\{\lambda_1(Q_{R_L},\xi_L) \ge a_L^\Xi + \bar{\lambda}_L + \frac{s}{a_L} \Big\}\;,\qquad s\in\R\,.
\end{equ}

\begin{theorem}\label{Th:Tail}
There exists a sequence of positive constants $(\eta_L)_{L\ge 1}$ 
which vanishes in the limit $L\to\infty$ such that the following statements hold for any given $s\in \R$
\begin{enumerate}[label=\normalfont{(\arabic*)}]
		\item\label{i:1} (Tail distribution of the main eigenvalue)
		\begin{equ}[e:Tail]
		\lim_{L\to\infty} \Big(\frac{L}{R_L}\Big)^d \P(\Lambda_L(s)) = e^{-s}\;.
		\end{equ}
		\item\label{i:2} (Approximation of the main eigenvalue)
		\begin{equ}[e:Approx]
		\lim_{L\to\infty} \Big(\frac{L}{R_L}\Big)^d \P\Big(\Lambda_L(s) ; \big| \lambda_1(Q_{R_L},\xi_L) - \big( \Xi_L(w_L) + \bar{\lambda}_L\big) \big| > \frac{\eta_L}{a_L} \Big) = 0\;,
		\end{equ}
		and
		\begin{equ}[e:Approxbis]
			\lim_{L\to\infty} \Big(\frac{L}{R_L}\Big)^d \P\Big(\Xi_L(w_L) \geq a_L^\Xi + \frac{s}{a_L} ; \big| \lambda_1(Q_{R_L},\xi_L) - \big( \Xi_L(w_L) + \bar{\lambda}_L\big) \big| > \frac{\eta_L}{a_L} \Big) = 0\;.
		\end{equ}
		\item\label{i:3} (Magnitude of the maximum)
		\begin{equ}[e:Max]
		 \lim _{C\to\infty}\limsup_{L\to\infty} \Big(\frac{L}{R_L}\Big)^d \P\Big(\Lambda_L(s) ; \xi_L(w_L) \notin I_{L}(C) \Big)=0\;,
		\end{equ}
		where, for $L,C>0$, $I_{L}(C)$ is the interval defined in~\eqref{e:ILC}. 
		\item\label{i:4} (Large spectral gap) there exists a $C'>0$ independent of $s$ such that 
		\begin{equ}[e:SpectralGap]
		\lim_{L\to\infty} \Big(\frac{L}{R_L}\Big)^d \P\Big(\Lambda_L(s) ; \lambda_2(Q_{R_L},\xi_L) > a_L^\Xi + \bar\lambda_L - C' \frac{a_L}{d_L}\Big) = 0\;.\end{equ}
		\item\label{i:5} (Behaviour of the eigenfunction)
		\begin{equ}[e:Eigen]
		\lim_{L\to\infty} \Big(\frac{L}{R_L}\Big)^d \P\Big(\Lambda_L(s) ; \|\phi_{R_L} - \bar{\varphi}_{L}(\cdot - w_L)\|_{\ell^2(Q_{R_L})} > \frac{d_L}{a_L}\eta_L \Big) = 0\;.
		\end{equ}
	\end{enumerate}
\end{theorem}

The rest of the section is devoted to the proof of this theorem. 
The crucial step in our analysis is the identification of the expansion for the eigenvalue 
in point~\ref{i:2} above. As it is one of the major technical novelties of  
our work, we dedicate to it the next section.

\subsection{Approximating the eigenproblem}\label{sec:local}

In Section~\ref{Subsec:Shape}, we have seen that whenever the potential is larger than $a_L-\theta$ at some point $x_0$,  
it induces a local (deterministic) shape in a neighbourhood of $x_0$ and the fluctuations around 
such shape are encoded via $\zeta_{L,x_0}$. 
In this section, we want to understand how this 
influences
the behaviour of the main eigenvalue and eigenfunction of the Anderson Hamiltonian on $Q_{R_L}$. 
To this purpose, notice first that 
for any $x_0\in Q_{R_L}$, it holds that
\begin{equ}[e:lambdaV]
\lambda_1(Q_{R_L}, \xi_L) =\xi_L(x_0) +  \lambda_1( Q_{R_L}, V_{L,x_0})
\end{equ}
where we set $\glsadd{VLx} V_{L,x_0}\eqdef\xi_L-\xi_L(x_0)$. For $x\in  Q_{R_L}$, by~\eqref{e:FluctField}, $V_{L,x_0}$ satisfies
\begin{equ}[e:VLDecomp]
V_{L,x_0}(x)=\xi_L(x_0)(v_L(x-x_0)-1) +\zeta_{L,x_0}(x)\;.
\end{equ}
Our goal now is twofold. On the one hand we want to prove that, since on the event $E_{L,x_0}$, 
$\xi_L(x_0)$ is the unique maximum and is of order $a_L$, in  the right-hand side of \eqref{e:VLDecomp}  we can replace 
the first summand by $-\cS_L$. On the other hand, we will show that the first non-trivial contribution of the fluctuation field 
$\zeta_{L,x_0}$ to the main eigenvalue on $Q_{R_L}$ is given by the r.v.~$\Phi_L(x_0)$. Let us state the theorem which rigorously details what we just explained.

\begin{theorem}\label{thm:MainEstimate}
There exists a constant $C>0$ and an integer $L_0\geq 1$ such that for all $L\geq L_0$ and $x_0\in Q_{R_L-r_L}$, 
on the event $E_{L,x_0}=E_{L,x_0}(\kappa)$ (for $\kappa\in(0,1/3)$) from Definition~\ref{d:Event}, 
we have
\begin{equ}
\Big\vert \lambda_1(Q_{R_L},\xi_L) - \bar\lambda_L-\Xi_L(x_0) 
\Big\vert
\le C\frac{d_L}{a_L}\Big(\Big(\frac{d_L}{a_L}\Big)^{1-2\kappa}|\xi_L(x_0)-a_L| +\frac{1}{a_L}\Big),\label{e:EigenExpMain}
\end{equ}
where we recall that $\Xi_L(x_0)=\xi_L(x_0)+\Phi_L(x_0)$ and the latter is as in~\eqref{e:FL}, and
\begin{equ}[Eq:Bdpsi]
\| \phi_{R_L} - \bar\phi_{L}(\cdot-x_0) \|_{\ell^2(Q_{R_L})} \le  C\frac{d_L}{a_L}\sqrt{\Big(\frac{d_L}{a_L}\Big)^{1-2\kappa}|\xi_L(x_0)-a_L|+\frac{1}{a_L}}\;,
\end{equ}
where we extended $\bar\phi_{L}$ by setting it to 
be zero outside $Q_{r_L}$. 
\end{theorem}

In order to prove the above statement, we need to show that (1) 
we can  localise the eigenproblem to a ball of size $r_L$ centred around the maximum of $\xi$ on $Q_{R_L}$, 
which, since we will be working on $E_{L,x_0}$, 
is at $x_0\in Q_{R_L-r_L}$, 
and (2) the local eigenproblem on such ball is close to the deterministic 
one associated to the operator $\bar\cH_{L}$. 
For these, two main ingredients are required, namely,  
suitable a-priori estimates on the decay of the main eigenfunctions, 
and a basic (but very useful) technical lemma on convex functionals, 
which in particular applies to the Dirichlet forms associated to  $\cH_{Q_{R_L}, \xi_L}$ and $\bar\cH_{L}$. 
For the former, we will use~\cite[Lemma 4.2]{BiKo16} whose statement is recalled below 
in a slightly different formulation which better suits our purposes.

\begin{lemma}{\cite[Lemma 4.2]{BiKo16}}\label{lemma:BKL2Bound}
Let $V\colon\Z^d \rightarrow \R$ and $D\subset \Z^d$. 
Let $\lambda,\phi$ be an eigenvalue and eigenfunction (normalised in $\ell^2(D)$) of $\cH_{D,V}$ with Dirichlet boundary conditions. 
Assume $D'\subset D$, $A>0$ and $R\ge 1$ is an integer, such that
for all $x\in D$ such that $\min_{y\in D'} |x-y|_1 <R$ one has $V(x)< \lambda -A$,
where $|x|_1 \eqdef \sum_{i=1}^d |x_i|$ denotes the $\ell^1$-norm. Then,
\begin{equ}[e:BKL2Bound]
\sum_{x\in D'} |\phi(x)|^2 \le \Big(1 + \frac{A}{2d}\Big)^{-2R} \;.
\end{equ}
\end{lemma}

Let us see what type of information the previous lemma provides in our context.

\begin{lemma}\label{l:DecayPsi}
	In the setting of Theorem~\ref{thm:MainEstimate}, there exists a constant $c_d>0$ such that 
	for $L$ large enough, on the event 
	$E^1_{L,x_0}\cap E^2_{L,x_0}$ (see \eqref{e:E1} and \eqref{e:E2}), 
	we have
	\begin{equ}[e:Decay]
		\phi_{R_L}(x)^2\leq  \Big(1+ c_d\frac{a_L}{d_L}\Big)^{-2|x-x_0|}\qquad \forall\, x\in Q_{R_L}\,,
	\end{equ}
	and
	\begin{equ}[e:x0]
		\phi_{R_L}(x_0)^2\geq 1 -\Big(1+ c_d\frac{a_L}{d_L}\Big)^{-2}\,.
	\end{equ}
	Furthermore, both~\eqref{e:Decay} and~\eqref{e:x0} hold with $\phi_{R_L}$ and $x_0$ replaced by 
	$\bar{\phi}_L$ and $0$ (the restriction to the event $E^1_{L,x_0}\cap E^2_{L,x_0}$ clearly being unnecessary  in this case).
\end{lemma}
\begin{proof}
	Without loss of generality, we take $x_0=0$ throughout this proof, and 
	omit the corresponding subscript from the notation (so that e.g. $V_L=V_{L,0}$ and so on). 
	Moreover, note that $\phi_{R_L}$ is also the principal eigenfunction of $\cH_{Q_{R_L}, V_L}$, 
	for $V_L$ as in~\eqref{e:VLDecomp}, associated to the eigenvalue 
	$\lambda_1(Q_{R_L},V_L)=\lambda_1(Q_{R_L},\xi_L)-\xi_L(0)$.
	
	If we establish the bound
	\begin{equ}[e:Decaybis]
	\sum_{\substack{y\in Q_{R_L}:\\|y|_1 \ge |x|_1}} \phi_{R_L}(y)^2 \le \Big(1+ c_d\frac{a_L}{d_L}\Big)^{-2|x|}\qquad \forall\, x\in Q_{R_L}\,,
	\end{equ}
	then both~\eqref{e:Decay} and~\eqref{e:x0} follow (for the latter recall that $\phi_{R_L}$ and $\bar\phi_{L}$ are normalised in $\ell^2(Q_{R_L})$). We thus prove \eqref{e:Decaybis}.
	On the event $E^1_{L}\cap E_L^2$, by Proposition \ref{l:Event} provided $L$ is large enough
	\begin{equ}[e:PotBound]
		V_L(y)\leq -\frac{\fc}{2}\frac{a_L}{d_L}\,,\qquad \forall y\in Q_{R_L}^{\neq0}\,.
	\end{equ}
	By~\eqref{e:AlmostDecay}, the same holds for $-\cS_L$ with $\fc/2$ replaced by $\fc$. 
	Since $V_L(0) =0=\cS_L(0)$, we deduce from the min-max formula that $\lambda_1(Q_{R_L},V_L)\wedge\bar\lambda_L\geq -2d$. 
	Hence, upon setting $R= |x|_{1}$, $D=Q_{R_L}^{\neq0}$, $D'= \{y\in Q_{R_L}^{\neq0}: |y|_1 \ge |x|_1\}$ and 
	$A= \fc d_L/(4a_L)$ (or $\fc d_L/(2a_L)$ if we deal with $\bar{\phi}_L$), the hypothesis 
	in Lemma~\ref{lemma:BKL2Bound} holds provided $L$ is large enough. 
	Thus,~\eqref{e:BKL2Bound} yields the bound in~\eqref{e:Decaybis} but with the exponent at the r.h.s. given by  
	$-2|x|_1$. However since $|y|_1 \ge |y|$ for any $y\in \Z^d$, the desired bound immediately follows 
	and the proof is complete. 
\end{proof}

Before stating the next lemma detailing the second tool we need, let us briefly motivate it. 
Let $\cH$ be either of the operators $\cH_{Q_{R_L},\xi_L}=-\Delta + \xi_L$ on $Q_{R_L}$ or $\bar\cH_L=-\Delta -\cS_L$ on $Q_{r_L}$, 
$\lambda$ and $\phi$ be its respective principal eigenvalue and eigenfunction (which we 
take normalised and non-negative to ensure uniqueness), $r$ be either $R_L$ or $r_L$. 
By the min-max theorem, we know that $\lambda= \max \cD(\psi) = \cD(\phi)$ where $\cD$ is either $\cD_{R_L}$ or $\bar{\cD}_L$ 
and the latter are given by 
\begin{equ}[e:I]
	\cD_{R_L}(\psi)=\langle \psi, \cH_{Q_{R_L},\xi_L} \psi\rangle_{\ell^2(Q_{R_L})}\;,\quad \bar{\cD}_L(\psi)=\langle \psi, \bar{\cH}_L \psi\rangle_{\ell^2(Q_{r_L})}\;,
\end{equ}
the maximum carrying over all functions $\psi\colon Q_r \to \R$, normalised in $\ell^2(Q_r)$. Actually, we do not need to consider all such functions $\psi$, but only those that share the decay properties of $\phi$ as detailed in Lemma~\ref{l:DecayPsi}. Thus, we will view $\cD$ as a functional of $(r+1)^d-1$ variables 
(as the value at $0$ of the normalised, non-negative functions can be recovered from those elsewhere) 
defined on $\cZ_{r}\subset \ell^2(Q_r^{\neq 0})$ whose elements $\psi$ satisfy 
\begin{equs}[e:Z]
|\psi(x)|^2\leq  \Big(c_d\frac{a_L}{d_L}\Big)^{-2|x|}\,,\qquad \forall x\in Q^{\neq 0}_r\,.
\end{equs}
$\cZ_{r}$ is closed and convex. Note that, compared to \eqref{e:Decay}, we imposed a slightly larger upper bound: 
this is to ensure that $\phi_{R_L}$ and $\bar\phi_L$ lie in the interiors of $\cZ_{R_L}$ and $\cZ_{r_L}$ respectively.

The next lemma provides a general statement that suitably exploits 
convexity to derive estimates on the increments of functionals as above, near their maximisers.

\begin{lemma}\label{l:Convex}
Let $S\subset\Z^d$ be finite and $\cC\subset\ell^2(S)$ be closed and convex. 
Assume that $G\colon\CC\to\R$ is a strictly concave, 
twice continuously differentiable (on $\mathring{\cC}$) functional  for which 
there exists a constant $H>0$, such that 
for all $z\in\mathring{\cC}$, its Hessian $\Hess G$ at $z$ satisfies
\begin{equ}[e:Convex]
\langle y, \Hess G(z) y\rangle_{\ell^2(S)}\leq -H \|y\|^2_{\ell^2(S)}\,,\qquad \forall y\in\ell^2(S).
\end{equ}
Let $x$ be the maximiser of $G$ in $\cC$ (which exists and is unique by~\eqref{e:Convex}), and assume it lies in $\mathring{\cC}$.
Then, for any $\bar x\in\mathring{\cC}$, we have
\begin{equs}
|G(x)-G(\bar x)|&\leq \frac{1}{H}\|\nabla G(\bar x)\|^2_{\ell^2(S)}\,,\label{e:DiffGN}\\
\|x-\bar x\|_{\ell^2(S)}&\leq \frac{1}{H}\|\nabla G(\bar x)\|_{\ell^2(S)}\,.\label{e:Diffx}
\end{equs}
\end{lemma}
\begin{proof}
Throughout the proof, the scalar product and the norm used are those on $\ell^2(S)$ 
thus, to lighten the notation, we omit the corresponding subscript, i.e. 
we write $\langle\cdot,\cdot\rangle$ and $\|\cdot\|$ 
in place of $\langle\cdot,\cdot\rangle_{\ell^2(S)}$ and $\|\cdot\|_{\ell^2(S)}$. 

We first establish~\eqref{e:Diffx} and then use it to show~\eqref{e:DiffGN}. 
Let $x$ be the maximiser of $G$ in $\cC$ and $\bar x\neq x$ be another element of $\CC$. 
Since $x$ is a maximiser, $\nabla G(x)\equiv 0$ and therefore
\begin{equs}[e:NablavsHess]
\langle x-\bar x, \nabla G(\bar x)\rangle&=-\langle x-\bar x, \nabla G(x)-\nabla G(\bar x)\rangle\\
&=-\int_0^1\langle x-\bar x, \Hess G(\bar x+t(x-\bar x))(x-\bar x)\rangle\dd t\,.
\end{equs}
By assumption $\cC$ is convex, so that $\bar x+t(x-\bar x)\in\mathring{\CC}$ for any $t\in[0,1]$, 
and we can use~\eqref{e:Convex} to bound the r.h.s. of~\eqref{e:NablavsHess} from below by 
$H\|x-\bar x\|^2$. As a consequence, we deduce 
\begin{equ}
\|x-\bar x\|^2\leq \frac{1}{H}\langle x-\bar x, \nabla G(\bar x)\rangle\leq \frac{1}{H}\|x-\bar x\|\|\nabla G(\bar x)\|
\end{equ} 
from which~\eqref{e:Diffx} follows at once. 

For~\eqref{e:DiffGN}, consider the map $f:[0,1]\ni t \mapsto G(\bar x+t(x-\bar x))$ which is concave and achieves its maximum at $t=1$. Necessarily the maximum of its derivative is attained at $t=0$ and therefore
\begin{equ}
	|G(x)-G(\bar x)| = |f(1)-f(0))| \le f'(0) = \langle x-\bar x, \nabla G(\bar x)\rangle \le \|x-\bar x\|\|\nabla G(\bar x)\|
\end{equ}
and thus~\eqref{e:DiffGN} follows by plugging~\eqref{e:Diffx} at the r.h.s.
\end{proof}

With Lemmas~\ref{l:DecayPsi} and~\ref{l:Convex} at our disposal, and anticipating some properties of $\cD_{R_L}$ and $\bar \cD_L$ stated and shown in Appendix~\ref{a:IL}, we are ready to prove Theorem~\ref{thm:MainEstimate}.  

\begin{proof}[Proof of Theorem~\ref{thm:MainEstimate}]
As soon as $x_0\in Q_{R_L-r_L}$, the inclusion $Q_{r_L,x_0} \subset Q_{R_L}$ holds. As the arguments presented in this proof only rely on such inclusion, w.l.o.g., we can take $x_0 = 0$ and omit the corresponding index from the notation. 

Let us first consider the l.h.s. of~\eqref{e:EigenExpMain}. 
Let $\cH_{Q_{R_L},V_{L}}$ be the operator on $Q_{R_L}$ given by $\Delta+V_L$ for $V_{L}=\xi_L-\xi_L(0)$. 
By~\eqref{e:lambdaV} and~\eqref{e:FL}, we have
\begin{equs}
\lambda_1(Q_{R_L},\xi_L) - \bar\lambda_L -\Xi_L(0)&=\lambda_1(Q_{R_L},\xi_L) - \xi_L(0) - \bar\lambda_L-\Phi_L(0)\\
&= \lambda_1(Q_{R_L},V_{L}) - \bar\lambda_L-\langle\bar\phi_{L}, \zeta_{L}\bar\phi_{L}\rangle_{\ell^2(Q_{r_L})}\\
&= \cD_{R_L}(\phi_{R_L})-\bar \cD_{L}(\bar\phi_{L})-\langle\bar\phi_{L}, \zeta_{L}\bar\phi_{L}\rangle_{\ell^2(Q_{r_L})}\,,
\end{equs}
where $\cD_{R_L}$ and $\bar \cD_{L}$ are defined according to~\eqref{e:I}. Then, 
the r.h.s.~coincides with the sum of two terms
\begin{equs}
\one&\eqdef \cD_{R_L}(\phi_{R_L})-\cD_{R_L}(\bar\phi_{L})\,,\label{e:A}\\
\two&\eqdef \cD_{R_L}(\bar\phi_{L})-\bar \cD_{L}(\bar\phi_{L})-\langle\bar\phi_{L}, \zeta_{L}\bar\phi_{L}\rangle_{\ell^2(Q_{r_L})}\,,\label{e:B}
\end{equs}
which we will separately control. 

Let us begin with $\two$. Since $\supp(\bar\phi_{L})\subset Q_{r_L}$, in the first summand 
 the scalar product in the definition of $\cD_{R_L}$ in~\eqref{e:I} 
can be taken in $\ell^2(Q_{r_L})$ instead of $\ell^2(Q_{R_L})$. Since all the scalar products appearing 
in this term are in $\ell^2(Q_{r_L})$, we lighten the presentation by omitting the corresponding subscript from the notation. 
Then,~\eqref{e:VLDecomp} and the definition of $\cS_L$ in~\eqref{e:Shape} give 
\begin{equs}
\two&=\langle \bar\phi_{L},\cH_{Q_{R_L},V_{L}}\bar\phi_{L}\rangle-\langle \bar\phi_{L},\bar\cH_{L}\bar\phi_{L}\rangle-\langle\bar\phi_{L}, \zeta_{L}\bar\phi_{L}\rangle\\
&=(\xi_L(0)-a_L)\langle \bar\phi_{L}, [v_L(\cdot)-1]\bar\phi_{L}\rangle\,.
\end{equs}
On $E_{L,x_0}^1$, $|\xi_L(0)-a_L|\leq \theta$ which, together with~\eqref{e:AlmostDecay}, implies 
\begin{equs}
\vert\two\vert&=\vert\xi_L(0)-a_L\vert\Big\vert \sum_{x\in Q_{r_L}^{\neq 0}}[v_L(x)-1]\bar\phi_{L}(x)^2\Big\vert \leq\frac{\theta}{d_L}\sum_{x\in Q_{r_L}^{\neq 0}} e^{\fc'|x|} \bar\phi_{L}(x)^2\;,
\end{equs}
the exclusion of $0$ in the first sum is a consequence of $v_L(0)=1$. Using the bound~\eqref{e:Decay} on the decay 
of $\bar\phi_{L}$, we easily deduce that
\begin{equ}[e:II]
	\vert\two\vert \leq  \frac{\theta}{d_L}\sum_{x\in Q_{r_L}^{\neq 0}} e^{\fc'|x|}\Big(1+ c_d\frac{a_L}{d_L}\Big)^{-2|x|} \le \frac{C}{a_L}\Big(\frac{d_L}{a_L}\Big)
\end{equ}
for some constant $C>0$ independent of $L$. 

We now turn to $\one$, for which we apply Lemma~\ref{l:Convex} with $S=Q_{R_L}^{\neq 0}$. More specifically, 
by Lemma~\ref{l:IL}, the Hessian of $\cD_{R_L}$ satisfies~\eqref{e:Convex} 
with $H=c_0\,a_L/d_L$ and the $\ell^2(Q_{R_L}^{\neq0})$-norm of its gradient can be bounded by~\eqref{e:GradientL2}. 
As a consequence,~\eqref{e:DiffGN} gives
\begin{equ}
|\one|\lesssim \frac{d_L}{a_L}\Big(\frac{1}{a_L}+\|\bar\phi_{L}\zeta_L\|_{\ell^2(Q_{r_L})}^2\Big) 
\end{equ}  
which, together with~\eqref{e:II}, implies for some constant $C>0$ and for all $L$ large enough
\begin{equ}
\Big\vert \lambda_1(Q_{R_L},\xi_L) - \bar\lambda_L - \Xi_L(0) 
\Big\vert
\le C\frac{d_L}{a_L}\Big(\frac{1}{a_L}+\|\bar\phi_{L}\zeta_{L}\|_{\ell^2(Q_{r_L})}^2\Big).\label{e:EigenExpMain2}
\end{equ}

Before completing the proof, let us consider the l.h.s. of~\eqref{Eq:Bdpsi}, 
for which we argue as for $\one$ above invoking~\eqref{e:Diffx} 
instead of~\eqref{e:DiffGN}. Thus, we deduce
\begin{equ}
\| \phi_{R_L} - \bar\phi_{L} \|_{\ell^2(Q^{\neq0}_{R_L})} \le  C\frac{d_L}{a_L}\sqrt{\frac{1}{a_L}+\|\bar\phi_{L}\zeta_L\|_{\ell^2(Q_{r_L})}^2}\,.
\end{equ}
To control the difference of $\phi_{R_L}$ and $\bar\phi_{L}$ at $0$, 
we use the fact that $x\mapsto \sqrt{1-x^2}$ is Lipschitz on $(-1/2,1/2)$ to get
\begin{equs}[e:0]
|\phi_{R_L}(0)-\bar\phi_{L}(0)|&=\Big\vert\sqrt{1-\|\phi_{R_L}\|_{\ell^2(Q_{R_L}^{\neq 0})}^2}-\sqrt{1-\|\bar\phi_{L}\|_{\ell^2(Q_{R_L}^{\neq 0})}^2}\Big\vert\\
&\lesssim \Big|\|\phi_{R_L}\|_{\ell^2(Q_{R_L}^{\neq 0})}^2-\|\bar\phi_{L}\|_{\ell^2(Q_{R_L}^{\neq 0})}^2\Big|\lesssim \|\phi_{R_L}-\bar\phi_{L}\|_{\ell^2(Q_{R_L}^{\neq 0})}\,,
\end{equs}
which ultimately gives (possibly for a different constant $C>0$)
\begin{equ}[Eq:Bdpsi2]
\| \phi_{R_L} - \bar\phi_{L} \|_{\ell^2(Q_{R_L})} \le  C\frac{d_L}{a_L}\sqrt{\frac{1}{a_L}+\|\bar\phi_{L}\zeta_{L}\|_{\ell^2(Q_{r_L})}^2}\;. 
\end{equ}
\medskip

Thanks to~\eqref{e:EigenExpMain2} and~\eqref{Eq:Bdpsi2},~\eqref{e:EigenExpMain} and~\eqref{Eq:Bdpsi} 
follow provided we suitably estimate the $\ell^2(Q_{r_L})$-norm of 
$\bar\phi_{L}\zeta_{L}$. Notice that so far, we never used the bound provided by 
$E_{L}^3$ and this is the point at which it becomes essential. 
Indeed, on $E^3_L$ for $x\in Q_{r_L}$ the fluctuation field $\zeta_{L}$ 
is bounded above by 
\begin{equs}
|\zeta_{L}(x)|&\leq \sqrt{{\rm Var}[\zeta_{L}(x)]}\Big(\frac{a_L}{d_L}\Big)^{\kappa|x|}\sqrt{1\vee(|\xi(0)-a_L|a_L)}\\
&\lesssim \frac{e^{\frac{c'' |x|}{2}}}{\sqrt{d_L}}\Big(\frac{a_L}{d_L}\Big)^{\kappa|x|}\sqrt{1+|\xi(0)-a_L|a_L} 
\end{equs}
where in the last step we also used~\eqref{e:VarZeta}. Note that $\zeta_{L}(0)=0$. Using the exponential decay of $\bar{\phi}_L$ stated in~\eqref{e:Decay}, we thus deduce
\begin{equs}
\|\bar\phi_{L}\zeta_{L}\|_{\ell^2(Q_{r_L})}^2&\lesssim \sum_{x\in Q_{r_L}^{\neq0}}\bar\phi_{L}(x)^2\frac{e^{c'' |x|}}{d_L} \Big(\frac{a_L}{d_L}\Big)^{2\kappa|x|}(1+|\xi(0)-a_L|a_L)\\
&\lesssim\frac{1+|\xi(0)-a_L|a_L}{d_L}\sum_{x\in Q_{r_L}^{\neq0}} e^{c'' |x|} \Big(c_d\frac{d_L}{a_L}\Big)^{(2-2\kappa)|x|}\\
&\lesssim \frac{1+|\xi(0)-a_L|a_L}{d_L}\Big(\frac{d_{L}}{a_L}\Big)^{2-2\kappa}=\Big(\frac{d_{L}}{a_L}\Big)^{1-2\kappa}\Big[\frac{1}{a_L}+|\xi_L(0)-a_L|\Big]\;.
\end{equs}
Plugging this estimate into~\eqref{e:EigenExpMain2} and~\eqref{Eq:Bdpsi2}, the 
statement follows at once. 
\end{proof}

Before concluding this section we state and prove the next proposition 
which, together with Theorem~\ref{thm:MainEstimate}, will be shown to imply 
a (diverging) spectral gap for the operator $\cH_{Q_{R_L}, \xi_L}$ on  
the event $E_{L,x_0}$ in Definition~\ref{d:Event}. 

\begin{proposition}\label{p:SpectralGap}
There exist a constant $C_\gap>0$ and an integer $L_0>1$ such that for all $L\ge L_0$ and all $x_0\in Q_{R_L-r_L}$,
on the event ${E}^1_{L,x_0}\cap E^2_{L,x_0}$ as in Definition~\ref{d:Event}, we have
\begin{equ}[e:gap]
\lambda_2(Q_{R_L},\xi_L) \leq \xi_L(x_0) - C_\gap \frac{a_L}{d_L}\,.
\end{equ}
\end{proposition}

\begin{remark}
The proof below shows that the constant $C_\gap$ in the previous statement depends on the 
decay of the correlation function of $\xi_L$ near the origin as it satisfies $C_\gap\geq 3\fc/8$ for $\fc$ as in~\eqref{e:AlmostDecay}. 
\end{remark}

\begin{proof}
By the min-max formula, the second eigenvalue of $\cH_{Q_{R_L}, \xi_L}$ satisfies
\begin{equ}[e:minmax2]
\lambda_2(Q_{R_L},\xi_L)=\sup\{\cD_{R_L}(\psi)\colon \|\psi\|_{\ell^2(Q_{R_L})}=1\,,\, \langle \psi,\phi_{R_L}\rangle=0\}\,.
\end{equ}
Notice first that, for any $\psi\in \ell^2(Q_{R_L})$ normalised to $1$ such that 
$0=\langle \psi,\phi_{R_L}\rangle=\sum_{x\in Q_{R_L}} \psi(x)\phi_{R_L}(x)$, 
we have 
\begin{equ}
\psi(x_0)=-\frac{1}{\phi_{R_L}(x_0)}\sum_{x\in Q_{R_L}^{\ne x_0}}\psi(x)\phi_{R_L}(x)
\end{equ}
the expression above being meaningful as~\eqref{e:x0} implies that, for $L$ 
large enough on the event $E_{L,x_0}$, $\phi_{R_L}(x_0)^2\geq 1 - C (d_L/a_L)^2 \ge 1/4$. 
As a consequence, Cauchy-Schwarz gives
\begin{equ}[e:Boundphi0]
|\psi(x_0)|\leq 2 \|\phi_{R_L}\|_{\ell^2(Q_{R_L}^{\ne x_0})}\|\psi\|_{\ell^2(Q_{R_L}^{\ne x_0})}\leq 2\sqrt{1-\phi_{R_L}(x_0)^2}\leq 2 \sqrt{C} \frac{d_L}{a_L} \le \frac12\,.
\end{equ}
Now, let $\psi$ be as above and consider the quadratic form at $\psi$, which is 
\begin{equs}
\cD_{R_L}(\psi)&\leq \langle \psi, \xi_L\psi\rangle\leq \xi_L(x_0)\psi(x_0)^2 +\Big(\max_{x\in Q_{R_L}^{\ne x_0}} \xi_L(x)\Big)\|\psi\|_{\ell^2(Q_{R_L}^{\ne x_0})}^2\\
&=\xi_L(x_0)\psi(x_0)^2+\Big(\max_{x\in Q_{R_L}^{\ne x_0}} \xi_L(x)\Big)(1-\psi(x_0)^2)\,.\label{e:Step1SG}
\end{equs}
By Proposition \ref{l:Event}, on the event $E_{L,x_0}$ we have
$$\max_{x\in Q_{R_L}^{\ne x_0}} \xi_L(x) \le \xi_L(x_0) - \frac{\fc}{2} \frac{a_L}{d_L}\;.$$
Using \eqref{e:Boundphi0} we find
\begin{equs}
	\cD_{R_L}(\psi) &\leq \xi_L(x_0) - \frac{\fc}{2} \frac{a_L}{d_L}(1-\psi(x_0)^2)\leq \xi_L(x_0)- \frac{3\fc}{8} \frac{a_L}{d_L}\;,
\end{equs}
thus concluding the proof.
\end{proof}

\subsection{Proof of Theorem \ref{Th:Tail}}

For Theorem~\ref{thm:MainEstimate} to provide a useful description of the principal eigenvalue of $\cH_{Q_{R_L},\xi_L}$, 
we need to ensure that the random variable at the r.h.s. of~\eqref{e:EigenExpMain}, i.e. 
\begin{equ}[e:error]
	\frac{d_L}{a_L}\Big(\Big(\frac{d_L}{a_L}\Big)^{1-2\kappa}|\xi_L(x_0)-a_L| +\frac{1}{a_L}\Big)\;, 
\end{equ}
is negligible compared to the putative fluctuation scale of $\lambda_1(Q_{R_L},\xi_L)$ itself, that is, $a_L^{-1}$. 
For this, in the next proposition, we show that we can restrict ourselves to the event $\{|\xi_L(x_0)-a_L|<\theta_L\}$ 
for some sequence $(\theta_L)_{L\ge 1}$ satisfying
\begin{equ}\label{e:etaL}
	\Big(\frac{d_L}{a_L}\Big)^{2-2\kappa} \theta_L\ll \frac1{a_L}\,.
\end{equ}
As we will apply Lemma~\ref{l:GaussianNew}, in particular~\eqref{e:GausSum2}, we are not allowed 
to take $\theta_L$ arbitrarily small, but we need it to satisfy $\theta_L\gg \max\{a_L^{-1},a_L\tau_L^2\}$. 
For concreteness, let us make a specific choice and from here on set  
\begin{equ}\label{e:ThetaL}
	\theta_L\eqdef \Big(\frac{a_L}{d_L}\Big)^\kappa\max\{a_L^{-1},a_L\tau_L^2\}\,,
\end{equ}
where $\kappa$ is the small parameter appearing in the definition of the event $E_{L,x_0}$. 
As $\kappa<1/3$ and recalling Assumption \ref{a:TechnicalAssum}, ~\eqref{e:etaL} can be immediately checked to be satisfied. 

Let us now state the above-mentioned proposition, whose proof is postponed to the end of the section.

\begin{proposition}\label{p:Shrink}
	Let $\theta_L$ be defined according to~\eqref{e:ThetaL}. For any event $G$ 
	\begin{equs}\label{e:Shrink}
		\P&\Big(G\cap \Lambda_L(s)\Big)\\
		&=\sum_{x_0\in Q_{R_L-r_L}} \P\Big(G\cap \Lambda_L(s) \cap E_{L,x_0}\cap \{|\xi_L(x_0)-a_L|<\theta_L\}\Big)+o\Big(\frac{R_L^d}{L^d}\Big)\,.
	\end{equs}
\end{proposition}

Before turning to the proof of Theorem~\ref{Th:Tail}, let us appreciate the advantage 
of the previous statement. What it guarantees is that, when studying the asymptotic 
behaviour of the probability of $\Lambda_L(s)\cap G$, for $s\in\R$ and some event $G$, 
it suffices to analyse that of $\Lambda_L(s)\cap G\cap E_{L,x_0}\cap \{|\xi_L(x_0)-a_L|<\theta_L\}$ 
for $x_0\in Q_{R_L-r_L}$. On $E_{L,x_0}\cap \{|\xi_L(x_0)-a_L|<\theta_L\}$, 
Theorem~\ref{thm:MainEstimate} says that $\lambda_1(Q_{R_L},\xi_L)$ satisfies 
\begin{equ}
	\lambda_1(Q_{R_L},\xi_L)\approx \bar\lambda_L+\Xi_L(x_0)=\bar\lambda_L+\xi_L(x_0)+\Phi_L(x_0)
\end{equ}
up to an error {\it strictly smaller} than the size of the fluctuations $a_L^{-1}$ (see~\eqref{e:error} and~\eqref{e:etaL}). 
Among the terms at the r.h.s., $\bar\lambda_L$ is deterministic while, 
for fixed $x_0$, $\xi_L(x_0)$ and $\Phi_L(x_0)$ are {\it independent Gaussian 
	random variables} of variances $1$ and $\tau_L^2$ respectively (see Lemma~\ref{l:FluctField} and \eqref{e:FL}).

In other words, we managed to reduce the analysis of the fluctuations of 
the complicated object $\lambda_1(Q_{R_L},\xi_L)$ to that  
of the sum of two independent Gaussian random variables which in turn 
was studied in detail in Lemma~\ref{l:GaussianNew}. 

\begin{proof}[Proof of Theorem~\ref{Th:Tail}]
	We take $\theta_L$ as in~\eqref{e:ThetaL}. Let us begin by identifying a suitable sequence $(\eta_L)_{L\ge 1}$. 
	For any $x_0\in Q_{R_L-r_L}$, on the event $E_{L,x_0}$, Theorem~\ref{thm:MainEstimate} gives 
	\begin{equs}\label{e:GreatApprox}
		\Big\vert \lambda_1(Q_{R_L},\xi_L)  - \bar\lambda_L-\Xi_L(x_0)\Big\vert\leq \frac{C}{a_L}\Big(\Big(\frac{d_L}{a_L}\Big)^{2-2\kappa}\theta_L a_L+\frac{d_L}{a_L}\Big)&=:\frac{\tilde\eta_L}{a_L}\,,\\
		\|\phi_{R_L}-\bar\phi_{L}(\cdot-x_0)\|_{\ell^2(Q_{R_L})} 
		\leq C\frac{d_L}{a_L}\sqrt{\Big(\frac{d_L}{a_L}\Big)^{1-2\kappa}\theta_L+\frac{1}{a_L}}&=:\frac{d_L}{a_L}\tilde\eta'_L\,,\label{e:GreatApprox2}
	\end{equs}
	and set $\eta_L\eqdef \tilde\eta_L\vee \tilde\eta_L'$. 
	By \eqref{e:etaL}, we see that $\eta_L$ goes to $0$ as $L\to\infty$. 
	We now turn to the proof of each of the five points in which the statement is divided. 
	\medskip
	
	\noindent\textbf{Point~\ref{i:1}.} Proposition~\ref{p:Shrink} with 
	$G$ being the whole probability space, implies that the statement follows 
	if we show that uniformly over all $x_0\in Q_{R_L-r_L}$, we have 
	\begin{equ}[e:1Final]
		P_{L,x_0}\eqdef\P\Big(\Lambda_L(s) \cap E_{L,x_0}\cap\{|\xi_L(x_0)-a_L|<\theta_L\}\Big)\sim\frac{e^{-s}}{L^d}\,.
	\end{equ}
	Thanks to~\eqref{e:GreatApprox}, we immediately get
	\begin{equs}
		P_{L,x_0}&\leq \P\Big(\Xi_L(x_0) \ge a^\Xi_L + \frac{s-\eta_L}{a_L}\Big)\,,\label{e:UB}\\
		P_{L,x_0}&\geq \P\Big(\Big\{\Xi_L(x_0) \ge a^\Xi_L + \frac{s+\eta_L}{a_L} \Big\}
		\cap E_{L,x_0} \cap\{ |\xi_L(x_0) - a_L| \le \theta_L\}\Big)\,,
	\end{equs}
	and recall that $\Xi_L(x_0)=\xi_L(x_0)+\Phi_L(x_0)$, that is $\Xi_L(x_0)$ is the sum of independent mean-zero Gaussian 
	random variables of variance $1$ and $\tau_L^2$, respectively. 
	Now, for the upper bound, we apply Lemma~\ref{l:GaussianNew} and in particular~\eqref{e:GausSum0}. 
	For the lower bound, we first remove the event $E_{L,x_0}$ at a price negligible with respect to $L^{-d}$, which is 
	allowed since $\theta_L\ll \theta$ and 
	\begin{equ}
		\P(|\xi_L(x_0) - a_L| \leq \theta_L ; E_{L,x_0}^\complement) \le \P(\xi_L(x_0) \ge a_L - \theta ; E_{L,x_0}^\complement) = o\Big(\frac1{L^d}\Big)\;,
	\end{equ}
	as implied by~\eqref{e:nonE}. Thus, additionally using~\eqref{e:GausSum2} we deduce
	\begin{equs}
		P_{L,x_0}\geq& \P\Big(\Xi_L(x_0)  \ge a^\Xi_L + \frac{s+\eta_L}{a_L} ; |\xi_L(x_0) - a_L| \le \theta_L\Big) + o\Big(\frac1{L^d}\Big)\\
		=&\P\Big(\Xi_L(x_0) \ge a^\Xi_L + \frac{s+\eta_L}{a_L}\Big)\\
		&-\P\Big(\Xi_L(x_0) \ge a^\Xi_L + \frac{s+\eta_L}{a_L} ; |\xi_L(x_0) - a_L| > \theta_L\Big) + o\Big(\frac1{L^d}\Big)\\
		=&\P\Big(\Xi_L(x_0) \ge a^\Xi_L + \frac{s+\eta_L}{a_L}\Big) + o\Big(\frac1{L^d}\Big)
	\end{equs}
	and, to the latter, we apply once again~\eqref{e:GausSum0}. 
	Putting upper and lower bounds together,~\eqref{e:1Final} follows. 
	\medskip
	
	\noindent\textbf{Point~\ref{i:2}.} Take $G$ in Proposition~\ref{p:Shrink} to be the event 
	$\{\vert \lambda_1(Q_{R_L},\xi_L) - \bar\lambda_L-\Xi_L(w_L) \vert>\eta_L/a_L\}$. 
	Note that, for every $x_0\in Q_{R_L-r_L}$, on the event $E_{L,x_0}$ the maximum of $\xi_L$ 
	is achieved at $x_0$ so that $w_L=x_0$ by Proposition~\ref{l:Event}. Thus,~\eqref{e:GreatApprox} implies that  
	each of the summands at the r.h.s. of~\eqref{e:Shrink} is $0$ 
	and~\eqref{e:Approx} follows at once. Regarding~\eqref{e:Approxbis}, since
	$$ \Big\{\Xi_L(w_L)  > a_L^\Xi + \frac{s}{a_L}\Big\} = \{\Theta^\Xi_L \ge s\}\;,$$
	the same argument as in~\eqref{e:QuasiQuasi} ensures that for any event $H$
	\begin{equs}
		\ &\P\Big(\Xi_L(w_L)  > a_L^\Xi + \frac{s}{a_L} ; H\Big)=\sum_{x_0\in Q_{R_L}} \P\Big(\Xi_L(x_0) > a_L^\Xi + \frac{s}{a_L} ; H ; E_{L,x_0}\Big) + o\Big(\frac{R_L^d}{L^d}\Big)\;.
	\end{equs}
	Now taking $H \eqdef \{| \lambda_1(Q_{R_L},\xi_L) - ( \Xi_L(w_L) + \bar{\lambda}_L) | > \eta_L/a_L\}$, 
	and recalling that on $E_{L,x_0}$ we have $w_L = x_0$, Theorem \ref{thm:MainEstimate} ensures that each term 
	in the sum over $x_0$ vanishes, thus completing the proof of \eqref{e:Approxbis}.

	\medskip
	
	\noindent\textbf{Point~\ref{i:3}.} We choose the event $G$ in~\eqref{e:Shrink} to be $G=\{\xi_L(w_L)\notin I_L(C)\}$. 
	As before, it suffices to control the probability of the event 
	$G\cap \Lambda_L(s) \cap E_{L,x_0}\cap \{|\xi_L(x_0)-a_L|<\theta_L\}$ uniformly in $x_0$. 
	Arguing as in~\eqref{e:UB}, we get 
	\begin{equs}
		\P\Big(&G\cap \Lambda_L(s) \cap E_{L,x_0}\cap \{|\xi_L(x_0)-a_L|<\theta_L\}\Big)\\
		&\leq \P\Big(\Xi_L(x_0)  \ge a^\Xi_L + \frac{s-\eta_L}{a_L} ; \xi_L(x_0) \notin I_L(C)\Big)
	\end{equs}
	and the quantity at the r.h.s. is independent of $x_0$. 
	The limit in~\eqref{e:GausSum} implies that for any given $\eps > 0$, 
	provided $C$ is large enough, its $\limsup_{L\to\infty}$ passes below $\epsilon / L^d$ and one can conclude.
	\medskip
	
	\noindent\textbf{Point~\ref{i:4}.} We choose $G$ in~\eqref{e:Shrink} to be 
	\begin{equ}
		G\eqdef \Big\{\lambda_2(Q_{R_L},\xi_L) > a_L^\Xi + \bar\lambda_L -C'\frac{a_L}{d_L}\Big\}\,,
	\end{equ}
	and, as argued before we only need to show that uniformly over all $x_0\in Q_{R_L-r_L}$
	\begin{equ}[e:Goal2]
		\P\Big(G\cap \Lambda_L(s)\cap E_{L,x_0}\Big)=o\Big(\frac{1}{L^d}\Big)\,.
	\end{equ}
	Actually, an even stronger statement is true, namely, there exists a constant $C'>0$ such that 
	for $L$ sufficiently large, the probability at the l.h.s. of~\eqref{e:Goal2} is simply equal to $0$. 
	Indeed, using Proposition~\ref{p:SpectralGap}, the definition of $E_{L,x_0}^1$, 
	the fact that $\bar \lambda_L\ge -2d$ (by the minmax formula) and that $\theta= 2d+1$ 
	as in Definition~\ref{d:Event}, we know that on 
	$E_{L,x_0}\cap \Lambda_L(s)$, 
	\begin{equs}
		\lambda_2(Q_{R_L},\xi_L)&\leq \xi_L(x_0)-C_\gap \frac{a_L}{d_L}\leq a_L + \theta -C_\gap \frac{a_L}{d_L}\\
		&\leq a_L \sqrt{1+\tau_L^2} + \bar{\lambda}_L + 2d +\theta-C_\gap \frac{a_L}{d_L}\leq a_L^\Xi + \bar\lambda_L+2\theta-C_\gap \frac{a_L}{d_L}\\
		&< a_L^\Xi + \bar\lambda_L-C'\frac{a_L}{d_L}\;,
	\end{equs}
	provided $C'<C_\gap$ and $L$ is large enough. Hence $G$ cannot hold on $E_{L,x_0}\cap \Lambda_L(s)$, and thus the l.h.s. of~\eqref{e:Goal2} is $0$. 
	\medskip
	
	\noindent\textbf{Point~\ref{i:5}.} Let $G$ be the event 
	$\{\|\phi_{R_L}-\bar\phi_{L}(\cdot-x_0)\|_{\ell^2(Q_{R_L})}  > \eta_Ld_L/a_L\}$, where $\eta_L$ was defined 
	right below~\eqref{e:GreatApprox2}. But then, by~\eqref{e:GreatApprox2}, for any $x_0$, 
	$\Lambda_L(s)\cap G\cap E_{L,x_0}=\emptyset$, so that each summand in the sum at the r.h.s. of~\eqref{e:Shrink}  
	is $0$ and~\eqref{e:Eigen} follows at once. Therefore, the proof of point~\ref{i:5} and of the statement are complete. 
\end{proof}

We now turn to the proof of Proposition \ref{p:Shrink}. It is performed in two steps, summarised by the following two lemmas.

\begin{lemma}\label{l:RedE}
	For any event $G$, as $L\to\infty$, we have 
	\begin{equ}[e:RedE]
		\P\Big(G\cap \Lambda_L(s)\Big)=\sum_{x_0\in Q_{R_L-r_L}} \P\Big(G\cap \Lambda_L(s) \cap E_{L,x_0}\Big)+o\Big(\frac{R_L^d}{L^d}\Big) \,.
	\end{equ}
\end{lemma}
\begin{proof}
	As an initial step, we want to localise $\Lambda_L(s)$ to an event in which 
	the maximum of $\xi_L$ over $Q_{R_L}$ is of order $a_L$. 
	To do so, we begin with two remarks. First, by the minmax formula, we have that 
	\begin{equ}
		\lambda_1(Q_{R_L},\xi_L) \le \max_{x\in Q_{R_L}} \xi_L(x)\;.
	\end{equ}
	Second, since $\bar{\lambda}_L\geq -2d$, we immediately deduce that for $L$ large enough
	\begin{equ}
		a_L-\theta=a_L-2d-1<a_L\sqrt{1+\tau_L^2}+\bar{\lambda}_L - 1\leq a_L^\Xi+ \frac{s}{a_L}+\bar\lambda_L\;.
	\end{equ}
	with $\theta=2d+1$ as in Definition~\ref{d:Event}. 
	Consequently, if $\max \xi_L < a_L - \theta$ 
	then $\lambda_1(Q_{R_L},\xi_L) < a_L - \theta\le a_L^\Xi + \frac{s}{a_L}+\bar\lambda_L$, 
	which means that $\Lambda_L(s)\subset \{\max \xi_L \geq a_L - \theta\}$. 
	Hence, by~\eqref{e:MaxNegl}, we deduce that
	\begin{equs}
	\P\Big(G\cap \Lambda_L(s)\cap \Big(\bigcup_{x_0\in Q_{R_L}} E_{L,x_0}\Big)^{\complement}\Big)\leq \P\Big(\{\max_{Q_{R_L}} \xi_L \geq a_L - \theta\}\cap \Big(\bigcup_{x_0\in Q_{R_L}} E_{L,x_0}\Big)^{\complement}\Big)
	\end{equs}
	is negligible compared to $R_L^d/L^d$, which 
	implies 
	\begin{equ}
	\P(G\cap \Lambda_L(s))=\sum_{x_0\in Q_{R_L}}\P(G\cap \Lambda_L(s)\cap E_{L,x_0}) +o\Big(\frac{R_L^d}{L^d}\Big)\,.
	\end{equ}
	As a consequence, we are left to neglect the sum over $Q_{R_L}\setminus Q_{R_L-r_L}$, which in turn can 
	be controlled by~\eqref{e:preBound} as 
	\begin{equs}
	\sum_{x_0 \in Q_{R_L}\backslash Q_{R_L-r_L}}\P(G\cap \Lambda_L(s)\cap E_{L,x_0})&\leq \sum_{x_0 \in Q_{R_L}\backslash Q_{R_L-r_L}} \P(\xi_L(x_0)\geq a_L-\theta)\\
	&\lesssim r_L R_L^{d-1}   \frac1{L^d} e^{\theta a_L}
	\end{equs}
	which is also negligible compared to $R_L^d/L^d$ since $\ln r_L \ll a_L \ll \ln R_L$ by~\eqref{e:RL} and~\eqref{e:rL}. 
\end{proof}

\begin{lemma}\label{l:Shrink}
	Let $\theta_L$ be defined according to~\eqref{e:ThetaL}. Then, for any event $G$ and any $x_0\in Q_{R_L-r_L}$, 
	as $L\to\infty$, we have 
	\begin{equ}
		\P\Big(G\cap \Lambda_L(s) \cap E_{L,x_0}\Big)=\P\Big(G\cap \Lambda_L(s) \cap E_{L,x_0}\cap \{|\xi_L(x_0)-a_L|<\theta_L\}\Big)+o\Big(\frac{1}{L^d}\Big)
	\end{equ}
	which, together with Lemma~\ref{l:RedE}, ultimately gives \eqref{e:Shrink}.
\end{lemma}
\begin{proof}
	Our goal is to show that uniformly over all $x_0\in Q_{R_L-r_L}$
	\begin{equ}[e:Actual]
		\P\Big(G\cap \Lambda_L(s) \cap E_{L,x_0}\cap \{|\xi_L(x_0)-a_L|\geq\theta_L\}\Big)=o\Big(\frac{1}{L^d}\Big)\,.
	\end{equ}
	By a union bound, we can separately estimate the probability of the events 
	$G\cap \Lambda_L(s) \cap E_{L,x_0}\cap\{\xi_L(x_0)\geq a_L+\theta_L\}$ and 
	$G\cap \Lambda_L(s) \cap E_{L,x_0}\cap\{\xi_L(x_0)\leq a_L-\theta_L\}$. 
	For the former,~\eqref{e:Actual} holds as can be seen by applying~\eqref{e:preBound} 
	to $\P(\xi_L(x_0)\geq a_L+\theta_L)$ 
	and using that, by definition of $\theta_L$ in~\eqref{e:ThetaL}, we have $a_L\theta_L\gg 1$. 
	
	For the latter, notice that by~\eqref{e:EigenExpMain}, on $E_{L,x_0}\cap\{\xi_L(x_0)\leq a_L-\theta_L\}$ we have 
	\begin{equ}
		\lambda_1(Q_{R_L},\xi_L) - \bar\lambda_L-\Xi_L(x_0)
		\le C\Big(\Big(\frac{d_L}{a_L}\Big)^{2-2\kappa}(a_L-\xi_L(x_0)) +\frac{d_L}{a_L^2}\Big)\,,
	\end{equ}
	so that on $\Lambda_L(s) \cap E_{L,x_0}\cap\{\xi_L(x_0)\leq a_L-\theta_L\}$ we have
	\begin{equ}
		\xi_L(x_0)+\Phi_L(x_0)\geq a_L\sqrt{1+\tau_L^2} -\frac{s+C d_L/a_L}{a_L}-C\Big(\frac{d_L}{a_L}\Big)^{2-2\kappa}(a_L-\xi_L(x_0))\;,
	\end{equ}
	which implies
	\begin{equs}
		\Phi_L(x_0)&\geq a_L(\sqrt{1+\tau_L^2}-1) -\frac{s+C d_L/a_L}{a_L}+(a_L-\xi_L(x_0))(1-C(d_L/a_L)^{2-2\kappa})\\
		&\geq \frac{1}{2}(a_L-\xi_L(x_0))\,,
	\end{equs}
	where we used that, by definition of $\theta_L$, on $\{\xi_L(x_0)\leq a_L-\theta_L\}$, 
	$|s+C d_L/a_L|/a_L\ll \theta_L\leq a_L-\xi_L(x_0)$. 
	We can now exploit the independence of $\xi_L(x_0)$ and $\Phi_L(x_0)$ 
	(and the fact that the variance of the latter is $\tau_L^2$) to deduce 
	\begin{equs}
		\P\Big(G&\cap \Lambda_L(s) \cap E_{L,x_0}\cap \{\xi_L(x_0)\leq a_L-\theta_L\}\Big)\\
		&\leq \P\Big(\Phi_L(x_0)\geq \frac{1}{2}(a_L-\xi_L(x_0))\,;\,\xi_L(x_0)\leq a_L-\theta_L\Big)\\
		&=\int_{\theta_L}^\infty\frac{e^{-\frac{(y-a_L)^2}{2}}}{\sqrt{2\pi}}\P\Big(\Phi_L(x_0)\geq \frac{y}{2}\Big)\dd y\lesssim \frac{a_L}{L^d}\int_{\theta_L}^\infty e^{y a_L} e^{-\frac{y^2}{8\tau_L^2}}\dd y\\
		&\leq \frac{a_L}{L^d}\int_{\theta_L}^\infty e^{-\frac{y^2}{16\tau_L^2}}\dd y = 4\sqrt{\pi} \frac{a_L \tau_L }{L^d}\P\Big(\cN(0,1) \ge \frac{\theta_L}{\sqrt{8} \tau_L}\Big) \lesssim \frac{1}{L^d}\frac{a_L\tau_L^2}{\theta_L}\;.
	\end{equs}
	In the third step, we neglected the term $e^{-y^2/2}$, and we used~\eqref{e:Asympa_L} and \eqref{e:TailGauss}, since $\theta_L \gg \tau_L$. The fourth step relies on~\eqref{e:ThetaL}, as $y>\theta_L\gg a_L\tau_L^2$ implies $a_L y\leq y^2/(16 \tau_L^2)$ for 
	$L$ large enough. The last step uses~\eqref{e:TailGaussAsymp} (as $\theta_L \gg \tau_L$) and a basic exponential bound. By \eqref{e:ThetaL}, the last quantity is negligible compared to $L^{-d}$, and thus the proof of the statement is complete.
\end{proof}

\section{The macroscopic eigenproblem}\label{sec:LocToGlob}

\noindent The goal of this section is to prove Theorems~\ref{Th:MainEigenvalue}, 
\ref{Th:MainLocalisation} and \ref{Th:MainRelationship}. Recall the splitting scheme introduced in~\eqref{e:UL} and the definition of 
$U_L$ in~\eqref{e:UL}. Similarly to what was done in Section~\ref{sec:LoctoGlobPot} for the potential, 
we will first (Section~\ref{subsec:spectrumUL}) establish the above mentioned theorems for the Hamiltonian restricted to $U_L$, 
and then, in Section~\ref{subsec:ProofMainTh}, show that the difference in behaviour on $U_L$ and $Q_L$ is negligible. To carry out the first task, we will patch together the spectral information on the operator $\Delta + \xi_L$ on each mesoscopic box $Q_{R_L, z_{j,L}}$ in order to deduce the spectral behaviour of the same operator but on $U_L$.

\subsection{Convergence of the top of the spectrum on $U_L$}\label{subsec:spectrumUL}

Consider the operator $\cH_{U_L,\xi_L}$ and let $(\hat{\lambda}_{k,L}, \hat{\varphi}_{k,L})_{k\ge 1}$ 
be the sequence of its eigenvalues and normalised eigenfunctions in the non-increasing order of their first coordinates: this is nothing but the collection of all the eigenvalues and eigenfunctions of $\Delta +\xi_L$ on every mesoscopic box $Q_{R_L,z_{j,L}}$. We will argue below that only the principal eigenvalue / eigenfunction on each mesoscopic box may contribute to the top of the spectrum on $U_L$ with large probability.\\
Let $\hat{y}_{k,L} \in U_L$ be the point where $\xi_L$ reaches its $k$-th largest value on $U_L$. Denote by $\hat{x}_{k,L}$ the point where $|\hat{\varphi}_{k,L}|$ reaches its maximum, 
and assume w.l.o.g.~that $\hat{\varphi}_{k,L}$ is positive at this point. Finally, let $\hat{\ell}_L(k)$ be defined through $\hat{x}_{k,L} = \hat{y}_{\hat{\ell}_L(k),L}$.

The main result of this section is the translation of Theorems \ref{Th:MainEigenvalue}, \ref{Th:MainLocalisation} and \ref{Th:MainRelationship} for the operator $\cH_{U_L,\xi_L}$.

\begin{theorem}\label{Th:MainRevisited}
	The following statements are satisfied. 
	\begin{enumerate}[label=\normalfont{(\arabic*)}]
		\item\label{i:main1} The point process
		$$ \Big(\frac{\hat{x}_{k,L}}{L}, a_L\big(\hat{\lambda}_{k,L}-a_L\sqrt{1+\tau_L^2}-\bar{\lambda}_L\big)\Big)_{1\le k \le \# U_L}\;,$$
		converges in law as $L\to\infty$ towards a Poisson point process on $[-1,1]^d\times\R$ of intensity $\dd x \otimes e^{-u} \dd u$,
		\item\label{i:main2} For any $k\ge 1$, the r.v.
		$$ \frac{a_L}{d_L} \Big\| \hat{\varphi}_{k,L}(\cdot)-\bar{\varphi}_L(\cdot-\hat{x}_{k,L}) \Big\|_{\ell^2(Q_{L})}\;,$$
		converges to $0$ in probability.
		\item\label{i:main3} It holds:
		\begin{enumerate}[label=\normalfont{(\alph*)}]
			\item\label{i:suba} if $\tau_L \ll \frac1{a_L}$, then for any given $k\ge 1$, $\P(\hat\ell_L(k) = k) \to 1$ as $L\to\infty$,
			\item\label{i:subb} if $\tau_L \sim \sqrt{b} \frac1{a_L}$ for some constant $b>0$, then $(\hat\ell_L(k))_{k\ge 1}$ converges in law to $(\ell_{\infty,b}(k))_{k\ge 1}$, the latter being defined according to~\eqref{e:Permutation}, 
			\item\label{i:subc} if $\tau_L \gg \frac1{a_L}$, then for any given $k\ge 1$, $\hat\ell_L(k)\to \plusinfty$ in probability.
		\end{enumerate}
	\end{enumerate}
\end{theorem}

For every $j\in\{1,\ldots,n_L\}$, let $w_{j,L}$ be the location of the maximum of $\xi_L$ on 
$Q_{R_L, z_{j,L}}$, $\Phi_{L}(\cdot)$ 
be given as in~\eqref{e:FL}, $a_L^\Xi=a_L\sqrt{1+\tau_L^2}$ as in~\eqref{e:aLXiL}, 
${\varphi}_{R_L,j}$ be the eigenfunction of the operator $\cH_{Q_{R_L, z_{j,L}}, \xi_L}$ 
associated to $\lambda_1(Q_{R_L, z_{j,L}},\xi_L)$ and 
$(\eta_L)_{L\ge 1}$ be the vanishing sequence as in Theorem~\ref{Th:Tail}. 

Let $\Xi_{j,L} = \xi_L(w_{j,L}) + \Phi_{L}(w_{j,L})$ and $(j_k)_{1 \le k \le n_L}$ be the permutation of 
$(1,\ldots,n_L)$ corresponding to the order statistics of $(\Xi_{j,L})_{j}$, that is, 
$\Xi_{j_1,L} \ge \Xi_{j_2, L} \ge \ldots$. 
The next lemma shows, among other things, that these order statistics provide the ordering of 
the eigenvalues with large probability.

\begin{lemma}\label{Lemma:GoodEvent}
	Let $(C_L)_{L\ge 1}$ be an arbitrary sequence of non-negative numbers going to $\plusinfty$ as $L\to\infty$. For any integer $k\ge 1$, let $\cV_L(k)$ be the event on which 
	\begin{equ}[e:condD]
		\hat{\lambda}_{k,L} = \lambda_1(Q_{R_L, z_{j_k,L}},\xi_L) \;,\quad \hat\varphi_{k,L} = \varphi_{R_L,j_k} \;,\quad\hat{x}_{k,L} = w_{j_k,L}\;,
	\end{equ}
	and
	\begin{equs}[e:condD2]
				\big| \lambda_1(Q_{R_L, z_{j_k,L}},\xi_L) - (\Xi_{j_k, L} + \bar\lambda_L) \big| \le \frac{\eta_L}{a_L}\;,\\
				\frac{a_L}{d_L} \Big\| \varphi_{R_L,j_k}(\cdot)-\bar{\varphi}_L(\cdot-w_{j_k,L}) \Big\|_{\ell^2(Q_{L})} \le \eta_L\;,\\
				\xi_L(w_{j_k,L})\leq \frac{a_L}{\sqrt{1+\tau_L^2}}+C_L\max\{\tfrac{1}{a_L},\tau_L\}\;. 
	\end{equs}
Then, the probability of $\cV_L(k)$ goes to $1$ as $L\to\infty$. 
\end{lemma}
We postpone the proof of this crucial lemma, and proceed with that of Theorem~\ref{Th:MainRevisited}.

\begin{proof}[Proof of Theorem~\ref{Th:MainRevisited}]
Fix an integer $k_0\ge 1$. Throughout the proof, 
we will use that, by Lemma~\ref{Lemma:GoodEvent}, we have 
\begin{equ}[e:PGoodEvent]
\P\Big(\bigcap_{k=1}^{k_0} \cV_L(k)\Big)\longrightarrow 1\,,\qquad \text{as $L\to\infty$. }
\end{equ}
In particular (see~\eqref{e:condD},~\eqref{e:condD2}), 
the $k_0$ pairs $(\hat{x}_{k,L},\hat{\lambda}_{k,L})_{1\le k \le k_0}$ match with the $k_0$ pairs 
$(z_{j_k,L},\Xi_{j_k,L} + \bar\lambda_L)_{1 \le k \le k_0}$, up to an error of 
at most $R_L$ for the first coordinate and $\eta_L/a_L$ for the second. 
Thus, the convergence of $\cP^\Xi_L$ stated in Proposition \ref{P:PPPxiPrelim} 
ensures that $(\hat{x}_{k,L}/L,a_L(\hat{\lambda}_{k,L} - a_L^\Xi - \bar\lambda_L))_{1\le k \le k_0}$ 
converges in law to the $k_0$ largest points (in the non-increasing order of their second coordinate) 
of $\cP^\Xi_\infty$ and this completes the proof of~\ref{i:main1}. 

Furthermore, for any integer $k\le k_0$
	$$ \frac{a_L}{d_L} \Big\| \hat\varphi_{k,L}(\cdot)-\bar{\varphi}_L(\cdot-\hat{x}_{k,L}) \Big\|_{\ell^2(Q_{L})} = \frac{a_L}{d_L} \Big\| \varphi_{R_L,j_k}(\cdot)-\bar{\varphi}_L(\cdot-w_{j_k,L}) \Big\|_{\ell^2(Q_{L})} \le \eta_L\;,$$
	so that also the conclusion of~\ref{i:main2} follows.
	
	We turn to~\ref{i:main3}. For any $k\ge 1$, on $\cV_L(k)$ the r.v.~$\hat\ell_L(k)$ is the rank of the r.v.~$\xi_L(w_{j_k,L})$ among the values taken by $\xi_L$ on $U_L$ in non-increasing order, that is
	\begin{equ}[e:hatell]
		w_{j_k,L} = \hat{y}_{\hat{\ell}_L(k),L}\;.
	\end{equ}
	Moreover, on $\cV_L(k)$ we have
	\begin{equ}[e:BoundxiLw]
		\xi_L(w_{j_k,L})\leq \frac{a_L}{\sqrt{1+\tau_L^2}}+C_L\max\{\tfrac{1}{a_L},\tau_L\} =: \rho^+_L\;.
	\end{equ}
If $\tau_L\gg a_L^{-1}$, upon taking $C_L =\sqrt{a_L \tau_L}$ and 
applying a Taylor expansion of the r.h.s.~of the last inequality, we see that $a_L(\rho^+_L - a_L) \to -\infty$. 
The convergence of $\cP^\xi_L$ stated in Proposition \ref{P:PPPxiPrelim} 
gives that $\cP^\xi_L([-1,1]^d \times [a_L(\rho^+_L - a_L), \plusinfty])$ goes to $\plusinfty$ in probability as $L\to\infty$, 
which means that the number of points in $U_L$ where $\xi_L$ lies above $\rho^+_L$ diverges in probability. 
As a consequence the rank $\hat{\ell}_L(k)$ goes to $\plusinfty$ in probability 
and this yields~\ref{i:main3}-\ref{i:subc} of the theorem.

We now assume that $\tau_L = \cO(a_L^{-1})$. 
To cover jointly~\ref{i:main3}-\ref{i:suba} and~\ref{i:main3}-\ref{i:subb}, 
when $\tau_L \ll a_L^{-1}$, set $b=0$ and $\ell_{\infty,0}(k)=k$ for any $k\ge 1$, 
while for $a_L\tau_L\sim \sqrt{b} $ and $b>0$
recall that $\ell_{\infty,b}$ was defined in \eqref{e:Permutation}. 

For any $k\ge 1$ let $j'_k$ be the index of the mesoscopic box 
where the $k$-th largest value of $\xi_L$ on $U_L$ lies, that is, $\hat{y}_{k,L} \in Q_{R_L,z_{j'_k,L}}$. 
Then, define 
	\begin{equ}[e:uv]
		u_{k,L} \eqdef a_L( \xi_L(\hat{y}_{k,L}) - a_L)\;,\quad v_{k,L} \eqdef a_L\Phi_L(w_{j'_k,L})\;.
	\end{equ}
The convergence of $\cP_L^{(\xi,\Phi)}$ stated in Proposition \ref{P:PPPxiPrelim} 
implies that $(u_{k,L},v_{k,L})_{k\ge 1}$ converges in law to $(u_k,v_k)_{k\ge 1}$ where $u_1 > u_2 > \ldots $ 
follow a Poisson point process of intensity $e^{-u} \dd u$, and $(v_k)_{k\ge 1}$ are independent $\cN(0,b)$ r.v.'s.

Recall that $(j_k)_{k\ge1}$ is the permutation corresponding to the order statistics of $(\Xi_{j,L})_{j\ge1}$ where $\Xi_{j,L} = \xi_L(w_{j,L}) + \Phi_{L}(w_{j,L})$. Now set 
	\begin{equs}
		p_{k,L} &\eqdef 
		a_L\big(\xi_L(w_{j_k,L})-a_L\big)+a_L \Phi_{L}(w_{j_k,L})\,. 
	\end{equs}
Combining~\eqref{e:hatell} and~\eqref{e:uv}, the integer $\hat\ell_L(k)$ is such that $j_k=j'_{\hat\ell(k)}$ and thus 
	\begin{equ}[e:pkL]
		p_{k,L} = u_{\hat\ell_L(k),L} + v_{\hat\ell_L(k),L}\,.
	\end{equ}
Hence, the r.h.s. of~\eqref{e:pkL} converges in law to the r.h.s. of~\eqref{e:Permutation} and therefore  
the r.v.'s $(\hat{\ell}_L(k))_{k\ge 1}$ converge in law to $(\ell_{\infty,b}(k))_{k\ge 1}$, 
which gives both~\ref{i:main3}-\ref{i:suba} and~\ref{i:main3}-\ref{i:subb}. The proof 
of the theorem is complete. 
\end{proof}

It remains to show Lemma~\ref{Lemma:GoodEvent} whose proof relies extensively on Theorem \ref{Th:Tail}. 
Let us first introduce some additional notation. 
For any $C>0$, let $F_{j,L}(C)=F_{j,L} \eqdef \cup_{i=1}^4 F_{j,L}^{i}$, for $j\in\{1,\dots,n_L\}$, 
be the union of the events whose probability 
is estimated in~\eqref{e:Approx}-\eqref{e:Eigen} but on boxes of side-length $R_L$ 
centred at $z_{j,L}$, i.e. 
\begin{equs}
	F_{j,L}^{1} &\eqdef \Big\{ \big| \lambda_1(Q_{R_L, z_{j,L}},\xi_L) - \big( \Xi_L(w_{j,L})  + \bar{\lambda}_L\big) \big| > \frac{\eta_L}{a_L} \Big\}\;,\\
	F_{j,L}^{2} &\eqdef \Big\{ \xi_L(w_{j,L}) \notin \Big[\tfrac{a_L}{\sqrt{1+\tau_L^2}} - C\max(\tfrac{1}{a_L},\tau_L), \tfrac{a_L}{\sqrt{1+\tau_L^2}} + C\max(\tfrac{1}{a_L},\tau_L)\Big] \Big\}\;,\\
	F_{j,L}^{3} &\eqdef \Big\{ \lambda_2(Q_{R_L, z_{j,L}},\xi_L) > a_L^\Xi + \bar\lambda_L - C' \frac{a_L}{d_L} \Big\}\;,\\
	F_{j,L}^{4} &\eqdef \Big\{ \|{\varphi}_{R_L,j}(\cdot) - \bar{\varphi}_{L}(\cdot-w_{j,L})\|_{\ell^2(Q_{R_L,z_{j,L}})} > \frac{d_L}{a_L}\eta_L \Big\}\;, 
\end{equs}
where $C'>0$ is fixed and chosen so that~\eqref{e:SpectralGap} holds. 
Set also
\begin{equs}
	G_{j,L}(s) \eqdef \Big\{ \Xi_L&(w_{j,L}) \ge a_L^\Xi + \frac{s}{a_L} ;\big| \lambda_1(Q_{R_L, z_{j,L}},\xi_L) - \big( \Xi_L(w_{j,L}) + \bar{\lambda}_L\big) \big| > \frac{\eta_L}{a_L} \Big\}\;.
\end{equs}
Given $C>0$, the event of interest is (the complement of) 
\begin{equ}[e:BjL]
	B_L(s,C)\eqdef\bigcup_{j=1}^{n_L}B_{j,L}(s,C) \;,
\end{equ}
where
\begin{equ}
	B_{j,L}(s,C) \eqdef \Big(\Big\{\lambda_1(Q_{R_L, z_{j,L}},\xi_L) \ge a_L^\Xi + \bar\lambda_L + \frac{s}{a_L}\Big\} \cap F_{j,L}(C)  \Big) \cup G_{j,L}(s)\;.
\end{equ}
The probability of $B_{j,L}(s,C)$ is independent of $j$, thus 
\begin{equ}
	\P(B_L(s,C)) \le n_L\, \P(B_{1,L}(s,C)) \lesssim \Big(\frac{L}{R_L}\Big)^d \P(B_{1,L}(s,C))\;,
\end{equ}
and, thanks to~\eqref{e:Approx}-\eqref{e:Eigen} we obtain
\begin{equ}[e:AsympBjL]
	 \lim  _{C\to\infty} \limsup_{L\to\infty} \P(B_L(s,C)) = 0\;.
\end{equ}

\begin{proof}[Proof of Lemma \ref{Lemma:GoodEvent}]
	Fix $k_0\ge 1$ and $\epsilon > 0$. We will show that
	$$ \liminf_{L\to\infty} \P\big(\cV_L(1) \cap\ldots \cap \cV_L(k_0)\big) \ge 1- \epsilon\;.$$
	Consider the event
	$$ D_L(c) \eqdef \Big\{\Xi_{L,j_{k_0}} < a_L^\Xi - \frac{c}{a_L}\Big\}  \cup \bigcup_{k=1}^{k_0} \Big\{|\Xi_{L,j_{k+1}} - \Xi_{L,j_{k}}| \le 10\frac{\eta_L}{a_L}\Big\}\;.$$
	From the convergence of $\cP_L^\Xi$ stated in Proposition \ref{P:PPPxiPrelim}, we deduce that, provided $c>0$ is sufficiently big, for all $L$ large enough $\P(D_L(c)) < \epsilon /2$. Furthermore, choosing also $C>0$ sufficiently big, 
	we deduce from~\eqref{e:AsympBjL} that for all $L$ large enough $\P(B_L(-c-1,C)) < \epsilon /2$.\\
	We now work on the event $B_L(-c-1,C)^\complement \cap D_L(c)^\complement$ whose probability is at least $1-\epsilon$ for all $L$ large enough, and will show that this event is contained in $\cV_L(1) \cap \cdots \cap \cV_L(k_0)$.  
	Using the fact that we are on the complements of the events $F_{j_k,L}^{1}$, $F_{j_k,L}^{3}$ and $G_{j_k,L}(-c-1)$, $1\le k \le k_0$, as well as the complement of $D_L(c)$, we deduce that there is a one-to-one correspondence between the $k_0$ largest eigenvalues / eigenfunctions of $\cH_{U_L,\xi_L}$, and the $k_0$ largest eigenvalues / eigenfunctions over the mesoscopic boxes, namely for every $1 \le k \le k_0$
	\begin{equ}
		\hat{\lambda}_{k,L} = \lambda_1(Q_{R_L, z_{j_k,L}},\xi_L) \;,\quad \hat\varphi_{k,L} = \varphi_{R_L,j_k} \;.
	\end{equ}
	The three bounds of~\eqref{e:condD2} follow from the complements of the events $F_{j_k,L}^{1}$, $F_{j_k,L}^{4}$ and $F_{j_k,L}^2$ (note that $C_L$ lies above $C$ for all $L$ large enough). 
	Since $\bar{\varphi}_L$ is almost a Dirac mass at the origin, the second bound in~\eqref{e:condD2} 
	also implies that $\hat\varphi_{k,L}$ (which is equal to $\varphi_{R_L,j_k}$) admits its maximum at $w_{j_k,L}$ and therefore $\hat{x}_{k,L} = w_{j_k,L}$.
\end{proof}

\subsection{Proof of the main results}\label{subsec:ProofMainTh}

This last section is devoted to the proof of Theorems \ref{Th:MainEigenvalue}, \ref{Th:MainLocalisation} 
and \ref{Th:MainRelationship}. Thanks to Theorem \ref{Th:MainRevisited}, what remains to show is 
that the eigenvalues and eigenfunctions of $\cH_L$ are sufficiently close to those of $\cH_{U_L,\xi_L}$ 
and that the localisation centres are the same. More precisely, 
we need to check that for any $k\in\N$, the random variables
\begin{equ}[e:ToProvehat]
a_L(\hat{\lambda}_{k,L} - {\lambda}_{k,L})\;,\quad \frac{a_L}{d_L}\| \hat{\varphi}_{k,L} - \varphi_{k,L}\|_{\ell^2(Q_{L})}\;,
\end{equ}
converge to $0$ in probability as $L\to\infty$, and that the probability of the event 
$\{\hat{x}_{k,L} = x_{k,L}\}$ goes to $1$ as $L\to\infty$. 
\medskip

\noindent Recall that $R_L$ and $r_L$ satisfy~\eqref{e:RL} and~\eqref{e:rL}, respectively. 
As a preliminary step, fix $k_0\ge 1$ and define the event $G_L\eqdef G_L^{(1)}\cup G_L^{(2)}$ as 
\begin{enumerate}[label=\normalfont{(\arabic*)}]
	\item 
	\label{i:on_G_L1}
	on $G_L^{(1)}$ the first $k_0+1$ eigenvalues of $\cH_{U_L,\xi_L}$ are larger than $a_L^\Xi + \bar\lambda_L - \frac{1}{\sqrt{a_L}}$, and all their spacings are at least $a_L^{-3/2}$, i.e. 
	\begin{equ}
	\hat{\lambda}_{k_0+1,L} \ge a_L^\Xi + \bar\lambda_L - \frac1{\sqrt{a_L}}\;\qquad \text{and}\qquad \hat{\lambda}_{i,L}-\hat{\lambda}_{i+1,L} > a_L^{-3/2} \;,
	\end{equ}
	for all $i\in\{1,\dots, k_0\}$, 
	\item on $G_L^{(2)}$, for all $x\in Q_L$ we have (recall that $\theta = 2d+1$)
	\begin{equ}
	x\notin  \bigcup_{j=1}^{n_L} Q_{R_L-r_L, z_{j,L}} \quad \Longrightarrow\quad \xi_L(x) < a_L - \theta\;. 
	\end{equ}
\end{enumerate}
The lower bounds $a_L^\Xi + \bar\lambda_L - \frac1{\sqrt{a_L}}$ and $a_L^{-3/2}$ in \ref{i:on_G_L1} are relatively arbitrary and 
chosen so that the next lemma holds. 

\begin{lemma}\label{l:GL}
	We have $\P(G_L)\to 1$ as $L\to \infty$.
\end{lemma}
\begin{proof}
By the first item of Theorem \ref{Th:MainRevisited}, we know that 
$(a_L(\hat{\lambda}_{k,L}-a_L^\Xi - \bar\lambda_L))_{k\ge 1}$ converges to a Poisson point process of intensity $e^{-u}du$. 
Therefore, the probability of $G^{(1)}_L$ goes to $1$. 
Regarding $G^{(2)}_L$, note that the cardinality of $Q_L \setminus \bigcup_{j=1}^{n_L} Q_{R_L-r_L, z_{j,L}} = \bigcup_{j=1}^{n_L} Q_{R_L + \sqrt{R_L},  z_{j,L}} \backslash Q_{R_L-r_L, z_{j,L}}$ (recall~\eqref{e:QLnL})
is of order 
\begin{equ}
n_L (R_L-r_L)^{d-1} (\sqrt{R_L} + r_L) \lesssim L^d \frac1{\sqrt{R_L}}\,.
\end{equ} 
Hence, a union bound and~\eqref{e:preBound} imply that for $L$ large enough
\begin{equs}
\P\big((G^{(2)}_L)^\complement\big)\leq \sum_{x\in Q_L \setminus \bigcup_{j=1}^{n_L} Q_{R_L-r_L, z_{j,L}}}\P(\xi_L(x)\geq a_L-\theta)\lesssim \frac{1}{\sqrt{R_L}} e^{\theta a_L}
\end{equs}
and, since $\ln R_L\gg a_L$ by~\eqref{e:RL} and~\eqref{e:rL}, 
the right-hand side vanishes as $L\to\infty$. Thus, the statement follows.
\end{proof}

We are now ready to complete the proof of the main statements.

\begin{proof}[Proof of Theorems \ref{Th:MainEigenvalue}, \ref{Th:MainLocalisation} and \ref{Th:MainRelationship}]
In view of Lemma \ref{l:GL}, we can and will work on $G_L$ throughout the proof. 
Our goal is to  show that for every $k\in\{1,\ldots,k_0\}$, the r.v.'s in~\eqref{e:ToProvehat} 
vanish as $L\to\infty$. 
Once this is established, we can easily deduce by Theorem~\ref{Th:MainRevisited} that the localisation centres are the same 
with probability converging to $1$. 
Indeed, recall that $x_{k,L}$, resp.~$\hat{x}_{k,L}$, is the point at which $|\varphi_{k,L}|$, 
resp.~$|\hat{\varphi}_{k,L}|$, achieves its maximum. By item~\ref{i:main2} 
of Theorem \ref{Th:MainRevisited}, combined with the fact that $\bar{\varphi}_L(0) = 1-\cO(d_L/a_L)$ 
and $\bar{\varphi}_L(x)= \cO(d_L/a_L)$ for $x\neq 0$
as shown in~\eqref{e:x0}, we deduce that $\hat\varphi_{k,L}(\hat{x}_{k,L})$ 
converges to $1$ in probability as $L\to\infty$, while $\hat\varphi_{k,L}(y)$ for $y\neq\hat{x}_{k,L}$ vanishes 
at rate $d_L/a_L$. Now, if $\frac{a_L}{d_L}\| \hat{\varphi}_{k,L} - \varphi_{k,L}\|_{L^2(Q_L)}$ goes to $0$ in probability, 
then $\varphi_{k,L}$ behaves as $\hat\varphi_{k,L}$, which means that it converges to $1$ at $\hat{x}_{k,L}$ 
and vanishes elsewhere, so that $\hat{x}_{k,L}$ is the unique maximum of $\varphi_{k,L}$ 
and thus the probability of $x_{k,L} = \hat{x}_{k,L}$ goes to $1$. 
\medskip

Let us now turn to the convergence of the r.v.'s in~\eqref{e:ToProvehat}. 
What we will prove is that these quantities are bounded above by a deterministic 
constant that goes to $0$ as $L\to\infty$. 
\medskip

Denote by $\partial U_L$ the inner boundary of $U_L$, that is, the set of points of $U_L$ that admit at least one neighbour outside $U_L$. Similarly denote by $\partial (Q_L\backslash U_L)$ the set of points of $Q_L \backslash U_L$ that admit at least one neighbour in $U_L$. 

Let $(\lambda,\varphi)$ be an eigenvalue/eigenfunction of $\cH_L$ on $Q_L$ which we assume to 
be such that $\lambda \ge  \hat{\lambda}_{k_0,L}$. 
Recall that $\theta = 2d+1$ and $\bar\lambda_L \ge -2d$. Note that, as we are on $G_L$, for all $x\in Q_L\setminus \cup_{j=1}^{n_L} Q_{R_L-r_L, z_{j,L}}$, we have for large $L$ 
\begin{equ}
\xi_L(x)<a_L-\theta \leq a_L^\Xi + \bar\lambda_L-1 \leq \hat{\lambda}_{k_0,L}-1+\frac{1}{\sqrt{a_L}}\leq \lambda-\frac{1}{2}\,.
\end{equ}
Therefore, Lemma \ref{lemma:BKL2Bound} applied with $D'=(Q_L\setminus U_L)\cup\partial U_L$, 
$A= 1/2$ and $R=r_L-1$ yields
\begin{equ}[e:deltaL]
\sum_{x\in (Q_L \backslash {U_L}) \cup \partial U_L } |\varphi(x)|^2 \le \delta_L \eqdef \Big(1 + \frac{1}{4d}\Big)^{-2(r_L-1)}\;.
\end{equ}
and, since $r_L\ge a_L$ by~\eqref{e:rL}, the r.h.s. is negligible compared to $a_L^{-n}$ for any given $n\ge 1$. 
In particular, $\phi$ puts negligible mass on the complement of $U_L$. 
What we want to do now is (a) use the above to show that the $\ell^2$-distance between $\phi$ 
and its normalised restriction to $U_L$ is small, and (b) prove that there exists a unique $k$ 
such that the latter is close to $\hat\phi_{k,L}$. 

Set 
\begin{equ}
\psi\eqdef \frac{\varphi \1_{U_L}}{\|\varphi \1_{U_L}\|_2}\;,
\end{equ}
where, here and below, we write $\|\cdot\|_2$ for the $\ell^2(Q_L)$-norm and 
$\|\cdot\|_{\ell^2(U_L)}$ for the $\ell^2(U_L)$-norm. 
For (a), it suffices to note that, for all $L$ large enough
\begin{equation}\label{Eq:psiphi}
	\| \psi - \varphi\|_2 = \| \psi  - \varphi\1_{U_L} + \varphi \1_{U_L} - \varphi\|_2 \le (1- \|\varphi \1_{U_L}\|_{2}) + \|\varphi \1_{U_L^c}\|_{2} \le 2 \sqrt{\delta_L}\;.
\end{equation}
For (b) instead, the argument exploits the equation satisfied by $\phi$ and $\hat\phi_{k,L}$, and 
the fact that $(\hat\phi_{k,L})_{k\ge 1}$ forms an orthonormal basis of $\ell^2(U_L)$. 
By the former, we get  
\begin{align*}
	(\cH_L - \lambda)\psi &= \frac1{\|\varphi \1_{U_L}\|_{2}} \Big( \Delta(\varphi \1_{U_L}) + (\xi_L-\lambda)\varphi\1_{U_L}\Big)\\
	&= \frac1{\|\varphi \1_{U_L}\|_{2}} \Big( \Delta(\varphi \1_{U_L}) - \1_{U_L} \Delta \varphi + \big(\Delta \varphi + (\xi_L-\lambda)\varphi \big)\1_{U_L}\Big)\\
	&= \frac1{\|\varphi \1_{U_L}\|_{2}} \Big( \Delta(\varphi \1_{U_L}) - \1_{U_L} \Delta \varphi\Big) 
\end{align*}
and the r.h.s. is $0$ outside $\partial U_L \cup \partial (Q_L\backslash U_L)$. 
It is then easy to check that there exists a constant $C > 0$, independent of $L$, such that
\begin{equation}\label{Eq:KatoTemple}\begin{split}
		\| (\cH_L - \lambda)\psi \|_{\ell^2(U_L)}^2 &\le \frac{C}{\|\varphi \1_{U_L}\|_{2}^2} \sum_{x\in (Q_L \backslash {U_L}) \cup \partial U_L} |\varphi(x)|^2\le C \frac{\delta_L}{1-\delta_L}\;.
\end{split}\end{equation}
On the other hand, we can expand $\psi$ on the $\ell^2(U_L)$ basis provided 
by the eigenfunctions of $\cH_{U_L,\xi_L}$ thus yielding
$$ \psi = \sum_{k\ge 1}  \hat{\varphi}_{k,L} \langle \hat{\varphi}_{k,L}, \psi \rangle\;.$$
Since further $\cH_L \hat{\varphi}_{k,L} =\cH_{U_L,\xi_L}\hat{\varphi}_{k,L}= \hat{\lambda}_{k,L} \hat{\varphi}_{k,L}$ on $U_L$, we can write
\begin{align*}
	\| (\cH_L - \lambda)\psi \|_{\ell^2(U_L)}^2 = \sum_{k\ge 1} \langle \hat{\varphi}_{k,L}, \psi \rangle^2 |\hat{\lambda}_{k,L} - \lambda|^2\;.
\end{align*}
Now, by construction, $1=\|\psi\|_{\ell^2(U_L)}=\sum_{k\ge 1} \langle \hat{\varphi}_{k,L}, \psi \rangle^2$, 
so that the sum at the r.h.s. is a convex combination of the $(|\hat{\lambda}_{k,L} - \lambda|^2)_{k\ge 1}$. 
Then, \eqref{Eq:KatoTemple} implies that necessarily there exists a $k\ge 1$ such that
\begin{equation}\label{Eq:LambdaDist}
	|\hat{\lambda}_{k,L} - \lambda|^2 \le C \frac{\delta_L}{1-\delta_L}\;.
\end{equation}
As, by~\eqref{e:deltaL}, $\delta_L\ll a_L^{-3}$, $\lambda \ge \hat{\lambda}_{k_0,L}$ by assumption 
and, on $G_L$, the spacings between the $k_0+1$ first eigenvalues $\hat{\lambda}_{k,L}$ 
are all larger than $a_L^{-3/2}$, 
the integer $k$ satisfying \eqref{Eq:LambdaDist} belongs to $\{1,\ldots,k_0\}$ and is unique.
Moreover, for any $\ell\neq k$, we must have 
\begin{equ}
|\hat{\lambda}_{\ell,L} - \lambda|^2 \ge \frac1{2a_L^3}\,,
\end{equ}
\begin{calc}
because $(a+b)^2 \le 2(a^2+b^2)$ and $a^2 \ge \frac12 ((a+b)^2 - 2b^2)$, we have 
\begin{align*}
|\hat{\lambda}_{\ell,L} - \lambda|^2 \ge \frac12 (|\lambda_{k,L} - \lambda_{\ell,L}|^2 - 2|\hat{\lambda}_{k,L} - \lambda|^2) \ge \frac12 \Big(\frac{1}{a_L^3} - 2C \frac{\delta_L}{1-\delta_L} \Big),
\end{align*}
\end{calc}
or equivalently $2a_L^3|\hat{\lambda}_{\ell,L} - \lambda|^2 \ge 1$. As a consequence,  
\begin{equ}
q_{k,L}^2 \eqdef \sum_{\ell \ne k} \langle \hat{\varphi}_{\ell,L}, \psi \rangle^2 \le 2a_L^3 \sum_{\ell \ne k} \langle \hat{\varphi}_{\ell,L}, \psi \rangle^2 |\hat{\lambda}_{\ell,L} - \lambda|^2 \le C \frac{\delta_L}{1-\delta_L} 2a_L^3\;,
\end{equ}
that vanishes as $L\to\infty$ and thus gives 
\begin{equs}
	\| \psi - \hat{\varphi}_{k,L} \|_{\ell^2(U_L)}^2 &= (\langle \hat{\varphi}_{k,L}, \psi \rangle-1)^2 + \sum_{\ell \ne k} \langle \hat{\varphi}_{\ell,L}, \psi \rangle^2\\
	&= (\sqrt{1-q_{k,L}^2} - 1)^2 + q_{k,L}^2\lesssim  \delta_L a_L^3\,.
\end{equs}
Combining the previous with~\eqref{Eq:psiphi}, we finally obtain 
\begin{equ}
\| \varphi - \hat{\varphi}_{k,L} \|_{2} \le \| \psi - \varphi\|_2 + \| \psi - \hat{\varphi}_{k,L} \|_{\ell^2(U_L)}\lesssim \sqrt{ \delta_L a_L^3}\,.
\end{equ}
Summarising, we have constructed, on the event $G_L$, 
a map that associates to any eigenvalue/eigenfunction $(\lambda,\phi)$ of $\cH_L$ such that 
$\lambda \ge \hat{\lambda}_{k_0,L}$, some $(\hat{\lambda}_{k,L}, \hat\phi_{k,L})$ with $k\in\{1,\ldots,k_0\}$ such that we simultaneously have 
\begin{equation}\label{Eq:BoundKato} 
|a_L(\hat{\lambda}_{k,L} - \lambda)|^2 \lesssim \delta_L a_L^2 \;,\quad \frac{a_L}{d_L}\| \varphi - \hat{\varphi}_{k,L} \|_{2} \lesssim \sqrt{ \delta_L a_L^3} \frac{a_L}{d_L} \;.
\end{equation}
Note that this map is necessarily injective. Indeed, otherwise there would exist two orthonormal functions $\varphi$ and $\tilde\varphi$ in $\ell^2(Q_L)$ such that for some $k$
$$ \| \varphi - \hat{\varphi}_{k,L} \|_{2} \lesssim \sqrt{ \delta_L a_L^3}\;,\quad \| \tilde\varphi - \hat{\varphi}_{k,L} \|_{2} \lesssim\sqrt{ \delta_L a_L^3}\;,$$
thus raising a contradiction.

By the variational formula, we know that there are at least $k_0$ eigenvalues of $\cH_L$ 
that lie above $\hat{\lambda}_{k_0,L}$, which means that the above map is also surjective, and thus bijective. 
From the ordering of the eigenvalues, this map necessarily sends $\lambda_{k,L}$ to $\hat{\lambda}_{k,L}$ 
for every $k\in\{1,\ldots,k_0\}$. Since $\delta_L$ is negligible compared to any negative power of $a_L$, 
the r.h.s.'s of \eqref{Eq:BoundKato} go to $0$ as $L\to\infty$, 
and this ensures the convergence in probability to $0$ of \eqref{e:ToProvehat} and completes the proof.
\end{proof}

\appendix

\section{Gaussian estimates}\label{a:Gauss}

\begin{proof}[Proof of Lemma~\ref{l:GaussianNew}]
	We set $u_L \eqdef a_L\sqrt{1+\tau_L^2} + \frac{s}{a_L}$. At several places in the proof, we will use the inequality $1 - (2 \tau_L^2/3) \le (1+\tau_L^2)^{-1/2} \le 1 - (\tau_L^2/4)$ which holds true provided $L$ is large enough. We start by proving \eqref{e:GausSum0}. Since $X+\YYL$ is a standard Gaussian random variable with variance $1+\tau_L^2$, 
	a simple scaling argument applied to~\eqref{e:preBound}, combined with the fact that $\tau_L$ converges to $0$ as $L\to\infty$, implies  
	\begin{equ}
		\P\Big(X+\YYL \ge u_L \Big)= \P\Big(X \ge \tfrac{u_L}{\sqrt{1+\tau_L^2}} \Big) 
		\begin{calc}
		\sim \frac{1}{L^d} e^{-\frac{s}{\sqrt{1+\tau_L^2}} - \frac{s^2}{a_L^2 (1+\tau_L^2)}} 
		\end{calc}				
		\sim \frac1{L^d} e^{-s}\;.
	\end{equ}
	Concerning~\eqref{e:GausSum2}, it follows from \eqref{e:GausSum} and 
	the fact that, for $\theta_L$ as in the statement, $I_L(C)\subset[a_L-\theta_L, a_L+\theta_L]$. Indeed,  observe that $\theta_L \gg \tau_L$ and that 
	for the  lower bound (the upper bound being analogous) we have 
	\begin{equs}
		\tfrac{a_L}{\sqrt{1+\tau_L^2}}- C\max\{a_L^{-1}, \tau_L\}
		\ge a_L - \frac23 a_L \tau_L^2 - C\max\{a_L^{-1}, \tau_L\}
		\gg a_L - \theta_L, 
	\end{equs}

	\begin{calc}
	where for the upper bound, 
		\begin{equs}
	\tfrac{a_L}{\sqrt{1+\tau_L^2}}+C\max\{a_L^{-1}, \tau_L\}
	&\leq a_L- \frac{1}{4} a_L\tau_L^2+C\max\{a_L^{-1}, \tau_L\}\ll a_L+\theta_L\,, 
	\end{equs}
	\end{calc}
	from which the result follows. 
	
	We are thus left with proving \eqref{e:GausSum}. 
	As $X$ and $\YYL$ are independent, we have
	\begin{equs}
		\P\Big(X+\YYL\geq u_L\;;\; X\notin I_L(C)\Big)&= \int_{I_L(C)^c} \frac{e^{-\frac{x^2}{2}}}{\sqrt{2\pi}} \P(\YYL> u_L-x) \dd x\\
		&=\int_{I_L(C)^c} \frac{e^{-\frac{x^2}{2}}}{\sqrt{2\pi}} \P\Big(\cN(0,1)> \tfrac{u_L-x}{\tau_L}\Big) \dd x=: J^1_L+J^2_L
	\end{equs}
	where the former is the integral over $x< a_L(1+\tau_L^2)^{-1/2}-C\max(\frac{1}{a_L},\tau_L)=:\rho^-_L$ while the latter that 
	over $x> a_L(1+\tau_L^2)^{-1/2}+C\max(\frac{1}{a_L},\tau_L)=:\rho^+_L$. We are going to show that $ \lim _{C\to\infty}\limsup_{L\to\infty} L^d J^i_L = 0$, $i=1,2$. 
	
	Let us begin with $J^1_L$. Note that, provided $C>|s|$ and $L$ is large enough, for all $x< \rho^-_L$
	\begin{equ}
		\frac{ u_L-x}{\tau_L} \ge \frac12 a_L\tau_L\;.
	\end{equ}
	We can apply~\eqref{e:TailGauss} to deduce 
	\begin{equs}
		J^1_L &\le \int_{-\infty}^{\rho^-_L} \frac{e^{-\frac{x^2}{2}}}{2\pi} \frac{\tau_L}{u_L-x} e^{-\frac{(u_L-x)^2}{2\tau_L^2}} \dd x \lesssim \frac1{a_L\tau_L} \int_{-\infty}^{\rho^-_L} \frac{1}{2\pi} e^{-\frac{x^2}{2}-\frac{(u_L-x)^2}{2\tau_L^2}} \dd x
	\end{equs}
	Now, 
	$$\frac{x^2}{2}+\frac{(u_L-x)^2}{2\tau_L^2} = \frac{u_L^2}{2(1+\tau_L^2)} + \frac{1+\tau_L^2}{2\tau_L^2}\Big(x- \frac{u_L}{1+\tau_L^2}\Big)^2\;.$$
	The first summand is independent of $x$ and a simple computation combined with \eqref{e:preBound} yields
	\begin{equ}[e:J1a]
		\frac{1}{\sqrt{2\pi}}e^{-\frac{u_L^2}{2(1+\tau_L^2)}} \lesssim \frac{a_L}{L^d}\;.
	\end{equ}
	On the other hand, the change of variable $y = -\sqrt{(1+\tau_L^2)/\tau_L^2}(x- \frac{u_L}{1+\tau_L^2})$ and the fact that, provided $C > 2\max\{|s|,1\}$, $x<\rho^-_L$ implies that $y> C/2$ yield
	\begin{equs}
		\int_{-\infty}^{\rho^-_L} \frac{1}{\sqrt{2\pi}} e^{-\frac{1+\tau_L^2}{2\tau_L^2}(x- \frac{u_L}{1+\tau_L^2})^2} \dd x &\le \frac{\tau_L}{\sqrt{1+\tau_L^2}} \,\P\Big(\cN(0,1) > \frac{C}{2} \Big)\le \frac{\tau_L}{\sqrt{1+\tau_L^2}} e^{-\frac{C^2}{8}}\;,
	\end{equs}
	where we used \eqref{e:TailGauss} at the last line. Putting everything together, we have shown that
	$$ J^1_L \lesssim \frac1{L^d} \frac{1}{\sqrt{1+\tau_L^2}} e^{-\frac{C^2}{8}}\;,$$
	so that $ \lim _{C\to\infty} \limsup_{L\to\infty} L^d J^1_L = 0$.
	
	We turn to $J^2_L$. First of all, we note that
	\begin{equs}
		V_L\eqdef\int_{a_L + \frac{C}{3 a_L}}^{\infty} \frac{e^{-\frac{x^2}{2}}}{\sqrt{2\pi}} \P\Big(\cN(0,1)> \tfrac{u_L-x}{\tau_L}\Big) \dd x &\le \int_{a_L + \frac{C}{3 a_L}}^\infty \frac{e^{-\frac{x^2}{2}}}{\sqrt{2\pi}} \dd x\lesssim \frac1{L^d} e^{-\frac{C}{3}}\;,
	\end{equs}
	where we used \eqref{e:preBound} at the last line. We thus deduce that the 
	$\limsup$ first in $L\to\infty$ and then in $C\to\infty$ of $L^d V_L$ vanishes.
	Coming back to $J^2_L$, we distinguish two cases. If $\tau_L\le \sqrt{C}/a_L$ then
	$$ \rho^+_L \ge a_L - \frac23 a_L \tau_L^2  + \frac{C}{a_L} \ge a_L + \frac{C}{3a_L}\;,$$
	and therefore $J^2_L \le V_L$ and we can conclude. If $\tau_L > \sqrt{C}/a_L$ then, provided $C$ is large enough compared to $|s|$,
	$$ u_L - \frac18 a_L \tau_L^2 \ge a_L + \frac14 a_L \tau_L^2 + \frac{s}{a_L} - \frac18 a_L \tau_L^2 \ge a_L +  \frac{s + (C/8)}{a_L} \ge a_L + \frac{C}{10 a_L}\;.$$
	Consequently $J^2_L = W_L + V_L$ where
	$$ W_L \eqdef \int_{\rho^+_L}^{u_L - \frac18 a_L \tau_L^2} \frac{e^{-\frac{x^2}{2}}}{\sqrt{2\pi}} \P\Big(\cN(0,1)> \tfrac{u_L-x}{\tau_L}\Big) \dd x\;.$$
	We now follow the same steps as for the bound on $J^1_L$: provided $C$ is large enough compared to $|s|$ it holds
	\begin{equs}
		W_L &\leq \int_{\rho^+_L}^{u_L - \frac18 a_L \tau_L^2} \frac{e^{-\frac{x^2}{2}}}{2\pi} \frac{\tau_L}{u_L-x} e^{-\frac{(u_L-x)^2}{2\tau_L^2}} \dd x\\
		&\lesssim \frac1{a_L\tau_L} \int_{\rho^+_L}^{u_L - \frac18 a_L \tau_L^2 } \frac{1}{2\pi} e^{-\frac{x^2}{2}-\frac{(u_L-x)^2}{2\tau_L^2}} \dd x\\
		&\lesssim \frac1{L^d \tau_L} \int_{\rho^+_L}^{\infty} \frac{1}{2\pi} e^{-\frac{1+\tau_L^2}{2\tau_L^2}(x- \frac{u_L}{1+\tau_L^2})^2} \dd x\\
		&\lesssim \frac1{L^d} \P\big(\cN(0,1) > \frac{C}{2}\big)\;.
	\end{equs}
	We can apply \eqref{e:TailGauss} and get $W_L \lesssim \frac1{L^d} e^{-\frac{C}{8}}$. Hence, 
	$ \lim  _{C\to\infty}$ $\limsup_{L\to\infty}$ $ L^d J^2_L= 0$, which, together with the same limit for $J^1_L$, 
	completes  the proof. 
\end{proof}

\section{Basic properties and estimates on the quadratic forms}\label{a:IL}

In this appendix, we state and prove some basic results concerning the quadratic forms 
$\cD_{R_L}$ and $\bar \cD_{L}$ defined in~\eqref{e:I}, with 
$\cH$ and $r$ given by $\cH_{Q_{R_L}, V_L}$ and $R_L$, and, $\bar\cH_{L}$ and $r_L$, respectively. 
Recall that these operators are defined on $\cZ_{R_L}\subset\ell^2(Q_{R_L}^{\neq0})$ and 
$\cZ_{r_L}\subset\ell^2(Q_{r_L}^{\neq0})$, 
which are closed and convex (see~\eqref{e:Z}). With a slight abuse of notation, for $\psi$ in 
either of the two sets, we write
\begin{equ}[e:psio0IL]
\psi(0) \eqdef \sqrt{1 - \sum_{x\in Q_{r}^{\neq 0}} |\psi (x)|^2}\;.
\end{equ}
In the following lemma, we collect the properties of $\cD_{R_L}$ and $\bar \cD_{L}$ we will need.

\begin{lemma}\label{l:IL}
The maps $\cD_{R_L}$ and $\bar \cD_{L}$ are twice continuously differentiable and there exists $c_0>0$ such that 
for every $\psi\in\cZ_{R_L}$ (resp. $\psi\in\cZ_{r_L}$), the Hessian $\Hess \cD_{R_L}(\psi)$ (resp. $\Hess \bar \cD_{L}(\psi)$) 
satisfies on $E^1_{L}\cap E^2_{L}$ (as in Definition~\ref{d:Event})
\begin{equ}[e:LaplaBound]
\langle \varphi, \Hess  \cD_{R_L}(\psi) \varphi\rangle_{\ell^2(Q_{R_L}^{\neq 0})}\leq -c_0\frac{a_L}{d_L}\|\varphi\|_{\ell^2(Q_{R_L}^{\neq 0})}^2\;,\qquad\forall\,\varphi\in\ell^2(Q_{R_L})\,
\end{equ}
(resp.~$\langle \varphi, \Hess  \bar \cD_{L}(\psi) \varphi\rangle_{\ell^2(Q_{r_L}^{\neq 0})}$). 
In particular, $\cD_{R_L}$ (resp. $\bar \cD_{L}$) has 
a unique maximiser and such maximiser is $\phi_{R_L}$ (resp. $\bar\phi_{L}$). 
Furthermore, on $E^1_{L}\cap E^2_{L}$, 
there exists a constant $C>0$ such that for any $L$ large enough
\begin{equs}
\|\nabla \cD_{R_L}(\bar\phi_{L})\|_{\ell^2(Q_{R_L}^{\neq 0})}&\leq C\sqrt{ \frac{1}{a_L}+\|\bar\phi_{L}\zeta_L\|_{\ell^2(Q_{r_L})}^2}\,.\label{e:GradientL2}
\end{equs}
At last, for any $x\in Q_1^{\neq 0}$ we have
\begin{equ}[e:barphiL]
\bar\phi_L(x)=\frac{1}{\cS_L(x)}(1+o(1))=\cO\Big(\frac{d_L}{a_L}\Big)\,,
\end{equ} 
where $\cS_L$ is the shape in~\eqref{e:Shape}, and consequently
\begin{equ}[e:barlambdafinal]
	\bar{\lambda}_L = -2d + \sum_{x\in Q_1^{\neq 0}} \frac{1}{\cS_L(x)} + o\Big(\frac{d_L}{a_L}\Big)\;.
\end{equ}
\end{lemma}
\begin{proof}
We start with the differentiability and convexity of $\cD_{R_L}$. Observe that for $\psi \in \cZ_{R_L}$, 
since $V_{L}(0)=0$, the map $\cD_{R_L}$ is given by 
\begin{equ}
\cD_{R_L}(\psi) = -2d +2\psi(0) \sum_{\substack{x\in Q_{R_L}^{\neq 0}\\ x \sim 0}} \psi(x)
+ \sum_{\substack{x,y \in Q_{R_L}^{\neq 0}\\ x\sim y }} \psi(x) \psi(y) + \sum_{x\in Q_{R_L}^{\neq 0}} \psi(x)^2 V_{L}(x)\,.
\end{equ}
A direct computation shows that the first derivative of $\cD_{R_L}$ in the direction $\psi(x)$ for $x\in Q_{R_L}^{\neq 0}$ 
is given by 
\begin{equ}[e:DerIL]
\frac{\partial \cD_{R_L}}{\partial \psi(x)}(\psi)
= 2\psi(0) \1_{\{x\sim 0\}} +  2\frac{\partial \psi(0)}{\partial \psi(x)}  \sum_{\substack{y\in Q_{R_L}^{\neq 0}\\y\sim 0}} \psi(y) 
+ 2\sum_{\substack{y\in Q_{R_L}^{\neq 0}\\y\sim x}} \psi(y) + 2 \psi(x) V_{L}(x)\;,
\end{equ}
and the second derivative reads
\begin{equs}[e:Der2IL]
\frac{\partial^2 \cD_{R_L}}{\partial \psi(x)^2}(\psi)&=4 \frac{\partial \psi(0)}{\partial \psi(x)} \1_{\{x\sim 0\}} +2\frac{\partial^2 \psi(0)}{\partial \psi(x)^2}  \sum_{\substack{y\in Q_{R_L}^{\neq 0}\\y\sim 0}} \psi(y)+2 V_{L}(x)\,,\\
\frac{\partial^2 \cD_{R_L}}{\partial \psi(x')\partial \psi(x)}(\psi)&=2 \Big(\frac{\partial \psi(0)}{\partial \psi(x')} \1_{\{x\sim 0\}}+\frac{\partial \psi(0)}{\partial \psi(x)} \1_{\{x'\sim 0\}}
+\frac{\partial^2 \psi(0)}{\partial \psi(x)\partial \psi(x')}  \sum_{\substack{y\in Q_{R_L}^{\neq 0}\\y\sim 0}} \psi(y)+\1_{\{x\sim x'\}}\Big)\,,
\end{equs}
where the latter holds for $x\neq x'$. In the above expressions, the derivatives involving $\psi(0)$ equal
\begin{equ}[e:Derpsi0]
\frac{\partial\psi(0)}{\partial\psi(x)}=-\frac{\psi(x)}{\psi(0)}\;,\qquad\frac{\partial^2\psi(0)}{\partial\psi(x)\partial\psi(x')}=-\frac{\psi(x)\psi(x')}{\psi(0)^3}-\frac{\1_{\{x= x'\}}}{\psi(0)}\;.
\end{equ}
Since any $\psi\in\cZ_{R_L}$ satisfies the upper bound in~\eqref{e:Decay} 
and $\psi(0)$ is given according to~\eqref{e:psio0IL}, it follows that $\psi(0)>0$ so that $1/\psi(0)$ is well-defined. 
Hence, $\cD_{R_L}$ is twice continuously differentiable. It can now be checked that all terms appearing in the second order derivatives of $\cD_{R_L}$ are of order at most $1$ except $V_L(x)$ which, on the event $E^1_{L}\cap E^2_{L}$, is smaller than $-(\fc/2) a_L/d_L$ thanks to \eqref{e:LowerShape}. Consequently, it is straightforward to check the existence of $c_0>0$ such that, 
uniformly in $\psi\in\cZ_{R_L}$, 
$\Hess \cD_{R_L}(\psi)$ satisfies~\eqref{e:LaplaBound} for any $\varphi\in\ell^2(Q_{R_L})$. 
This in particular means that $\cD_{R_L}$ is strictly concave on $\cZ_{R_L}$ and its unique maximiser coincides with $\phi_{R_L}$.

The arguments apply almost verbatim to $\bar \cD_{L}$, the only specific input comes from the bound $-\cS_L(x) \le -\fc a_L/d_L$ which is a consequence of \eqref{e:AlmostDecay}.
\medskip

For~\eqref{e:GradientL2}, note that, since $\bar\phi_{L}$ is a maximiser for $\bar \cD_{L}$, 
$\nabla \bar \cD_{L}(\bar\phi_{L})\equiv0$. Viewing $\bar\cD_L$ as a function from $Q_{R_L}^{\neq 0}$ into $\R$ (which does not depend on $\psi(x)$ whenever $x\notin Q_{r_L}$), we further have $\frac{\partial \bar \cD_{L}}{\partial \psi(x)} (\bar\phi_{L}) = 0$ for all $x$. We now write
\begin{equs}
\frac{\partial \cD_{R_L}}{\partial \psi(x)} (\bar\phi_{L})=&\frac{\partial \bar \cD_{L}}{\partial \psi(x)} (\bar\phi_{L})+2\1_{ \tilde Q_{r_L}}(x)\sum_{\substack{y\in Q_{r_L}\\ y\sim x}} \bar\phi_{L}(y)+2 \bar\phi_{L}(x)(V_L(x)+\cS_L(x))\\
=&2\1_{\tilde Q_{r_L}}(x)\sum_{\substack{y\in Q_{r_L}\\ y\sim x}} \bar\phi_{L}(y) + 2\bar\phi_{L}(x)(\xi_L(0)-a_L) [v_L(x)-1]+2\bar\phi_{L}(x)\zeta_L(x)\,,
\end{equs}
where $\tilde Q_{r_L}\eqdef\{x\notin Q_{r_L}\colon \exists y\in Q_{r_L}\text{ s.t. }|x-y|= 1\}$. Thus, 
\begin{equs}
\|\nabla \cD_{R_L}&(\bar\phi_{L})\|_{\ell^2(Q_{R_L}^{\neq 0})}^2
 \lesssim 
\sum_{x\in \tilde Q_{r_L}}\Big|\sum_{\substack{y\in  Q_{r_L}\\ y\sim x}} \bar\phi_{L}(y)\Big|^2  \\
& +|\xi_L(0)-a_L|^2\sum_{x\in Q_{r_L}^{\neq 0}}\bar\phi_{L}(x)^2|v_L(x)-1|^2+\sum_{x\in Q_{r_L}^{\neq0}}\bar\phi_{L}(x)^2|\zeta_L(x)|^2\;.
\end{equs}
Let us bound the first two terms on the r.h.s. From the exponential decay~\eqref{e:Decay} of $\bar \phi_L$ we get
\begin{equ}
\sum_{x\in \tilde Q_{r_L}}\Big|\sum_{\substack{y\in  Q_{r_L}\\ y\sim x}} \bar\phi_{L}(y)\Big|^2 \lesssim r_L^{d-1}  \Big(c_d\frac{d_L}{a_L}\Big)^{2r_L}\leq r_L^{d-1}\Big(\frac12\Big)^{2r_L}\lesssim\frac1{r_L}\leq \frac1{a_L}
\end{equ}
for $L$ large enough, where in the last step we used that, by~\eqref{e:rL}, $r_L\geq a_L$. 

We turn to the second term. Using the exponential decay \eqref{e:Decay} of $\bar\phi_L$, \eqref{e:AlmostDecay}, and the content of event $E_{L}^1$, we deduce that
\begin{equs}
	|\xi(0)-a_L|^2\sum_{x\in Q_{r_L}^{\neq 0}}\bar\phi_{L}(x)^2|v_L(x)-1|^2 &\lesssim \sum_{x\in Q_{r_L}^{\neq 0}} \frac{e^{2c' |x|}}{d_L^2} \Big(\frac{d_L}{a_L}\Big)^{2|x|}\lesssim \frac1{d_L^2}\Big(\frac{d_L}{a_L}\Big)^2 
\end{equs}
which is negligible compared to $a_L^{-1}$. Putting everything together~\eqref{e:GradientL2} follows at once.  

\medskip

In order to prove~\eqref{e:barphiL}, note that $\nabla\bar\cD_L(\bar\phi_L)=0$. Hence, similar to~\eqref{e:DerIL}, for any $x\in Q_1^{\neq 0}$ 
we have 
\begin{equs}
0&=\frac{\partial\bar\cD_L}{\partial \psi(x)}(\bar\phi_L)= 2\bar\phi_L(0) -  2\frac{ \bar\phi_L(x)}{\bar\phi_L(0)}  \sum_{\substack{y\in Q_{R_L}^{\neq 0}\\y\sim 0}} \bar\phi_L(y) 
+ 2\sum_{\substack{y\in Q_{R_L}^{\neq 0}\\y\sim x}} \bar\phi_L(y) - 2 \bar\phi_L(x) \cS_{L}(x),
\end{equs}
which implies
\begin{equ}
\bar\phi_L(x)=\frac{1}{\cS_L(x)}\Big[1+(\bar\phi_L(0)-1) - \frac{ \bar\phi_L(x)}{\bar\phi_L(0)}  \sum_{\substack{y\in Q_{R_L}^{\neq 0}\\y\sim 0}} \bar\phi_L(y)+\sum_{\substack{y\in Q_{R_L}^{\neq 0}\\y\sim x}} \bar\phi_L(y) \Big]\,.
\end{equ}
By~\eqref{e:x0} and~\eqref{e:Decay}, the last three summands in the parenthesis are $\cO(d_L/a_L)$. In addition,~\eqref{e:AlmostDecay} and~\eqref{e:defdL} imply that
$1-v_L(x) \asymp 1/d_L$ for all $x\in Q_1^{\neq 0}$. Using~\eqref{e:Shape}, we thus deduce that $\cS_L(x) \asymp a_L/d_L$ for any $x\in Q_1^{\neq 0}$, 
from which~\eqref{e:barphiL} follows. To prove~\eqref{e:barlambdafinal}, it suffices to compute $\bar\cD_L(\bar\phi_L)$ using the estimates that we collected for $\bar\phi_L$.
\end{proof}

\section{Asymptotics of $\tau_L$}\label{App:tauL}

We concentrate on the smooth case (the other two cases being simpler). Recall that $\zeta_{L,0}$ is a centred Gaussian field that vanishes at the origin with a covariance given by~\eqref{e:covZeta}. We thus obtain
\begin{align*}
		\tau_L^2 &= \var\Big(\sum_{x\in Q_{r_L}}  \bar{\varphi}_L^2(x) \zeta_{L,0}(x) \Big)\\
		&= \E\left[\sum_{x,y\in Q_{r_L}}  \bar{\varphi}_L^2(x)\bar{\varphi}_L^2(y) \zeta_{L,0}(x)\zeta_{L,0}(y)\right]\\
		&=\sum_{x,y\in Q_{r_L}\backslash\{0\}} \bar{\varphi}_L^2(x)\bar{\varphi}_L^2(y) \big(v_L(x-y) - v_L(x) v_L(y)\big)\;.
\end{align*}
To evaluate the asymptotics of $\tau_L^2$, we first estimate the covariance of $\zeta_{L,0}$ that appears on the r.h.s.
Since $u$ is a radial function, $v=u*u$ is also radial and we have $v'(0)= v'''(0) =0$ (here $v'$ is the radial derivative of $v$) so that for a radial function $w$:
$$ v(x) = 1 + \fa |x|^2 + w(x)\;,\qquad  \fa \eqdef \frac{(u^{*2})''(0)}{2} < 0\;,$$
where $|w(x)| \leq C |x|^4$ for some constant $C>0$ (recall that $u$ is compactly supported). Recall that $v_L(x) = v(x/m_L)$ where $m_L\ge 1$ is a sequence that diverges to $\plusinfty$. The definition of $d_L$ given in~\eqref{e:defdL} combined with the above expression allows us to deduce that $d_L \sim m_L^2 / (-\fa)$. Writing $w_L(x) := w(x/m_L)$ we obtain 
\begin{align*}
	v_L(x-y) - v_L(x) v_L(y)&  =  1 + \frac{\fa}{m_L^2} |x-y|^2 + w_L(x-y) \\
	& \qquad	- (1 + \frac{\fa}{m_L^2} |x|^2 + w_L(x)) (1 + \frac{\fa}{m_L^2} |y|^2 + w_L(y))\\
	& =  -2\frac{\fa}{m_L^2} x \cdot y + \epsilon_L(x,y)\;,
\end{align*}
with $|\epsilon_L(x,y)| \leq C' \frac{|x|^4+|y|^4}{m_L^4}$ for some constant $C'>0$ independent of $x,y,L$. The radiality of $u$ implies that $\bar\varphi_L(-x)=\bar\varphi_L(x)$ for all $x\in\Z^d$. We thus have for any $x,y\in Q_{r_L}\backslash\{0\}$
$$ \bar \varphi_L^2(x) x \cdot  y + \bar \varphi_L^2(-x) (-x) \cdot  y = 0\;,$$
and consequently for any $y\in Q_{r_L}\backslash\{0\}$
\begin{align*}
	&  \sum_{x \in Q_{r_L}\backslash\{0\}} \bar \varphi_L^2(x) x \cdot  y =0.
\end{align*}
Therefore
\begin{align*}
	\tau_L^2 
	& = \sum_{x,y \in Q_{r_L}\backslash\{0\}} \bar \varphi_L^2(x) \bar \varphi_L^2(y) \epsilon_L(x,y)\;,
\end{align*}
and the exponential decay of $\bar \varphi_L$ stated in Lemma \ref{l:DecayPsi}, combined with the bound on $\epsilon_L$, shows that $\tau_L^2$ is of order at most $d_L^2/a_L^4$.

%
\markboth{References}{References}

\bibliographystyle{alpha}
\bibliography{references}

\end{document}